\documentclass[11pt]{article}
\usepackage{tikz-cd}
\usepackage{nicematrix}
\usepackage{colortbl}
\usepackage{xcolor}
\usepackage{color}

\usepackage{fullpage}
\usepackage{authblk}
\usepackage{blkarray}
\usepackage{amsmath}
\usepackage{amsthm}
\usepackage{amssymb}
\usepackage{bm}
\usepackage{setspace}

\usepackage{amsfonts}
\usepackage{eucal}
\usepackage{latexsym}
\usepackage{mathdots}
\usepackage{mathrsfs}
\usepackage{enumitem}

\numberwithin{equation}{section}
\numberwithin{subsection}{section}

\newtheorem{thm}{Theorem}

\newtheorem{lemma}{Lemma}
\newtheorem{prop}{Proposition}
\newtheorem{cor}{Corollary}
\newtheorem*{thm*}{Theorem}
\newtheorem*{cor*}{Corollary}
\newtheorem{thmA}{Theorem}

\newtheorem{quest}{Question}

\theoremstyle{definition}

\newtheorem{remark}{Remark}

\title{Sub-Hardy Hilbert spaces in the non-commutative unit row-ball}

\author{Michael T. Jury\thanks{Supported by NSF grant DMS-1900364}}
\affil[1]{\footnotesize University of Florida}

\author[2]{Robert T.W. Martin\thanks{Supported by NSERC grant 2020-05683}}
\affil[2]{\footnotesize University of Manitoba}

\date{}

\begin{document}

%\NiceMatrixOptions{cell-space-top-limit=1pt}
%\NiceMatrixOptions{cell-space-bottom-limit=1pt}

%\bibliographystyle{unsrt}
\maketitle

\begin{abstract}
In the classical Hardy space theory of square--summable Taylor series in the complex unit disk there is a circle of ideas connecting Szeg\"o's theorem, factorization of positive semi-definite Toeplitz operators, non-extreme points of the convex set of contractive analytic functions, de Branges--Rovnyak spaces and the Smirnov class of ratios of bounded analytic functions in the disk. We extend these ideas to the multi-variable and non-commutative setting of the full Fock space, identified as the \emph{free Hardy space} of square--summable power series in several non-commuting variables.  As an application, we prove a Fej\'er--Riesz style theorem for non-commutative rational functions. 
\end{abstract}

\section{Introduction} \label{sect:intro}

The purpose of this paper is to consider some factorization problems for non-commutative analogues of Toeplitz operators, in terms of various interconnections between these operators and (non-commutative versions of) the theory of de Branges--Rovnyak spaces. To explain what we mean, we will first recall some well-known results from the classical theory of Toeplitz operators. 

A bounded Toeplitz operator, $T$, is the compression of multiplication by an $L ^\infty$ function, $f$, on the complex unit circle, $\partial \mathbb{D}$, to the Hardy space, $H^2$, of $L^2$ functions which extend analytically to the complex unit disk. That is, $T= T_f := P_{H^2} M_f | _{H^2}$ and $f \in L ^\infty (\partial \mathbb{D} )$ is called the \emph{symbol} of $T$. Here, $H^2$ can also be defined as the Hilbert space of analytic functions in the complex unit disk, $\mathbb{D}$, with square--summable Taylor series coefficients at $0$, and $H^2$ embeds isometrically into $L^2$ of the circle by taking non-tangential limits almost everywhere with respect to normalized Lebesgue measure, $m$. A theorem of Brown and Halmos, \cite[Theorem 6]{BH-Toep}, characterizes the bounded Toeplitz operators as the bounded linear operators, $T \in \mathscr{L} (H^2)$, obeying the simple algebraic condition:
$$ S^* T S = T; \quad \quad S := M_z, $$ where $S$, the isometry of multiplication by the independent variable, $z$, is called the \emph{shift} on $H^2$.
We employ the notation $h(S) = T_h$, clearly $p(S) = T_p$ for any complex polynomial $p \in \mathbb{C} [z ]$. We say that $T_f \geq 0$ is \emph{factorable} if we can write $ T_f = T_{\overline{h}} T_h = h(S) ^* h(S)$, for some outer $h \in H^\infty$, where $H^\infty$ is the unital Banach algebra of uniformly bounded holomorphic functions in the disk. Here, we say an element, $h \in H^\infty$ is \emph{outer} if the analytic Toeplitz operator $h(S) =T_h$ has dense range and we say $h$ is \emph{inner} if $h(S)$ is an isometry.

A celebrated theorem of G. Szeg\"o, (later strengthened by Kolmogoroff and Kre\u\i n) implies that $T_f \geq 0$ is factorable if and only if its symbol, $f \geq 0$, is log--integrable, $$ \int _{\partial \mathbb{D} } \log f  > - \infty, $$ \cite{Szego}, see \cite[Chapter IV]{Hoff}.  As a particular corollary, suppose $T=T_f$ has a {\em factorable minorant}: that is, there is an $H^\infty$ function $g$ so that $T_f\geq T_g^*T_g$. In symbols, this means that $f\geq |g|^2$ on the unit circle. Since $\int \log |g|^2\, dm>-\infty$, the same holds for $f$, and thus there in fact exists an (outer) $a\in H^\infty$ so that $T_f=T_a^*T_a$. Thus,

\begin{thmA} \label{Opszego}
$T = T_f \geq 0$ is factorable if and only if it has a non-trivial factorable minorant. 
\end{thmA}

Factorizability of positive semi-definite Toeplitz operators is intimately connected to the theory of extreme points of the closed unit ball of $H^\infty$, as well as to the factorization of functions in the Smirnov class. Indeed, it is known that a contractive analytic function, $b \in [ H^\infty ] _1$, is an extreme point if and only if $1 -|b|$ is not log--integrable on the circle and this is in turn equivalent to $b$ being \emph{column--extreme} (CE) in the sense that if $a \in H^\infty$ is any multiplier and $c := \left( \begin{smallmatrix} b \\ a \end{smallmatrix} \right)$ is contractive as a multiplier from one to two copies of $H^2$, then $a \equiv 0$. Note that $c$ will be contractive if and only if $I - b(S) ^* b(S) \geq a(S) ^* a(S)$, so that $b$ is not CE if and only if the Toeplitz operator $T = I - b(S) ^* b(S)$ has a non-trivial factorable minorant. Hence there are three possibilities: (i) $I - b(S) ^* b(S)$ vanishes identically, in which case $b$ is inner, (ii) $I-b(S)^* b(S) \neq 0$ has no non-trivial factorable minorant, in which case $b$ is an extreme point (or equivalently CE), or, (iii) $b$ is not an extreme point and $I - b(S) ^* b(S)$ is factorable. Given a non-extreme $b \in [ H^\infty ] _1$, the existence of an outer $a$ so that the column $c = \left( \begin{smallmatrix} b \\ a \end{smallmatrix} \right)$ is inner is in fact equivalent to the statement that a positive semi-definite Toeplitz operator is factorable if and only if it has factorable minorant.  Namely, if $T = T_f \geq 0$ has a factorable minorant and we set $\mu = f \cdot m$ as before, then $\mu =\mu _b$ is the \emph{Aleksandrov--Clark measure} of a contractive $b \in [H^\infty ] _1$. Here, recall that there is (essentially) a bijection between positive, finite and regular Borel measures on the circle and contractive analytic functions in the disk \cite{Clark,Aleks1,Aleks2}. Since $\mu _b$ is absolutely continuous, Fatou's theorem implies that $b$ is not inner and that 
$$ T = (I - b(S) ^* ) ^{-1} (I - b(S) ^* b(S) ) (I - b(S) ) ^{-1}. $$ From this formula, it is clear that since $T$ has a factorable minorant, so does the numerator $I- b(S) ^* b(S)$, so that $b$ cannot be an extreme point. Moreover, $T$ will be factorable if and only if $I - b(S) ^* b(S)$ is factorable, \emph{i.e.} if and only if there exists an outer $a$ so that $c = \left( \begin{smallmatrix} b \\ a \end{smallmatrix} \right)$ is inner. Assuming that $b \in [ H^\infty ] _1$ is non-CE so that $I - b(S) ^* b(S)$ is factorable by Theorem \ref{Opszego}, there is then an outer function, $a \in [ H^\infty ] _1$ so that $a(0) >0$ and $a(S)^* a(S) = I- b(S) ^* b(S)$. This outer function, $a$, is uniquely determined by $b$ and we call $a$ the \emph{Sarason outer function of $b$} after D. Sarason who performed a detailed analysis and computed the Taylor coefficients of $a$ in \cite{Sarason-afun, Sarason-doubly}.

The Smirnov class, $\mathscr{N} ^+$, consists of the analytic functions in the disk obtained as ratios of $H^\infty$ functions with outer denominators. By a theorem of D. Sarason, any $f \in \mathscr{N} ^+$ factors uniquely as $f = \frac{b}{a}$, where $b, a \in [ H^\infty ] _1$ are contractive, $a$ is outer and the column $c := \left( \begin{smallmatrix} a \\ b \end{smallmatrix} \right)$ is inner \cite[Proposition 3.1]{Sarason-unbdd}. This is essentially a consequence of the general form of Beurling's theorem -- given any $f = \frac{b_0}{a_0} \in \mathscr{N} ^+$, multiplication by $f$ is densely--defined on $\mathrm{Ran} \, a _0 (S)$ as $a _0$ is outer and it is easily seen to be closable. Applying the Beurling theorem to the closed $S \otimes I_2-$invariant graph, $\mathrm{Gr} \, f(S)  \subseteq H^2 \otimes \mathbb{C} ^2$, where $f(S)$ is the closure of multiplication by $f$, gives a two-component column--inner $c = \left( \begin{smallmatrix} a \\ b \end{smallmatrix} \right)$ so that $\mathrm{Gr} \, f(S)  = \mathrm{Ran} \, c(S)$ and $a$ is necessarily outer since $T_f =f(S)$ is densely--defined. One can further show that, equipping the domains of $f(S)$ and $f(S)^*$ with the graph--norm yields $$ \mathrm{Dom} \, f(S) = \mathscr{M} (a) \quad \mbox{and} \quad \mathrm{Dom} \, f(S) ^* = \mathscr{H} (b), $$ where $\mathscr{M} (a) = \mathscr{M} (a(S))$ is the operator--range space of $a(S)$ and $\mathscr{H} (b) := \mathscr{M} \, \sqrt{I - b(S) b(S) ^* }$ is the \emph{de Branges--Rovnyak space} of $b$, the complementary space of $b(S)$. Here, if $A : \mathcal{H} \rightarrow \mathcal{J}$ is a linear operator, the operator--range space, $\mathscr{M} (A)$, is a Hilbert space equal to $\mathrm{Ran} \, A$ as a vector space, equipped with the norm that makes $A$ a co-isometry onto its range. Both $\mathscr{M} (a)$ and $\mathscr{H} (b)$ are contractively contained in $H^2$. In summary, given any non-extreme contractive $b \in H^\infty$, there is a unique outer $a \in [ H^\infty ] _1$ so that $c = \left( \begin{smallmatrix} a \\ b \end{smallmatrix} \right)$, is inner, $\mathrm{Ran} \, c(S) = \mathrm{Gr} \, f(S) $, for $f = b /a \in \mathscr{N} ^+$ and equipping the domains of $f(S)$ and $f(S)^*$ with the graph--norm yields $\mathrm{Dom} \, f(S) = \mathscr{M} (a)$ and $\mathrm{Dom} \, f(S) ^* = \mathscr{H} (b)$. In particular, the de Branges--Rovnyak space of any non-extreme $b \in [ H^\infty ] _1$ is the domain of the adjoint of some (generally unbounded) closed Smirnov multiplier and vice versa. \\

It turns out that a surprising amount of the above discussion can be carried over to the non-commutative realm, where we consider so-called ``multi-Toeplitz'' operators, first considered by Popescu. We now briefly describe the non-commutative framework. 

The full Fock space, $\mathbb{H} ^2 _d$, is a natural multi-variable and non-commutative (NC) generalization of $H^2$. The Fock space can be defined as the Hilbert space of all square--summable power series in several non-commuting variables, $\mathfrak{z} := (\mathfrak{z} _1 , \cdots , \mathfrak{z} _d )$, for some fixed $d \in \mathbb{N} \cup \{ +\infty \}$. Namely, let $\mathbb{F} ^d$ denote the \emph{free monoid}, the set of all finite \emph{words} in the $d$ letters $1, ... , d$. Any word, $\omega \in \mathbb{F} ^d$, has the form $\omega = i_1 \cdots i_n$,  $i_k \in \{ 1 , \cdots , d \}$ and the unit of $\mathbb{F} ^d$ is the \emph{empty word}, $\emptyset$, containing no letters. Here, $|\omega | =n$ is called the length of the word $\omega$ and $| \emptyset | := 0$. Employing the notation $\mathfrak{z} ^\omega := \mathfrak{z} _{i_1} \mathfrak{z} _{i_2} \cdots \mathfrak{z} _{i_n}$ and $1 := \mathfrak{z} ^\emptyset$, any $h \in \mathbb{H} ^2 _d$ is a formal power series:
$$ h (\mathfrak{z} ) = \sum _{\omega \in \mathbb{F} ^d} \hat{h} _\omega \mathfrak{z} ^\omega; \quad \quad \hat{h} _\omega \in \mathbb{C}, \quad \quad \sum _{\omega \in \mathbb{F} ^d} | \hat{h} _\omega | ^2 = \sum _{k=0} ^\infty \sum _{|\omega | = k} | \hat{h} _\omega | ^2 < + \infty. $$ As first shown by Popescu, these formal power series converge absolutely in operator-norm when evaluated at any strict contraction $Z = (Z_1, \cdots , Z_d ) : \mathcal{H} \otimes \mathbb{C} ^d \rightarrow \mathcal{H}$, $Z_k \in \mathscr{L} (\mathcal{H} )$ from $d$ copies of a Hilbert space $\mathcal{H}$ into itself, so that elements of $\mathbb{H} ^2 _d$ can be viewed as analytic, non-commutative functions on the open \emph{NC unit row-ball} of all strict \emph{row contractions} on a fixed, separable Hilbert space \cite{Pop-freeholo}.  Moreover, as in the single-variable setting, the unital Banach algebra, $\mathbb{H} ^\infty _d$, of all uniformly bounded NC functions in the NC unit row-ball, can be identified, completely isometrically, with the \emph{left multiplier algebra} of $\mathbb{H} ^2 _d$, the algebra of all NC functions in the unit row-ball which left multiply $\mathbb{H} ^2 _d$ into itself \cite{Pop-freeholo,SSS}. In particular, left multiplication by any of the $d$ non-commuting variables defines a $d-$tuple of isometries, $L_k := M^L _{\mathfrak{z} _k}$, with pairwise orthogonal ranges which we call the \emph{left free shifts}. It follows that the row $d-$tuple $L = (L_1, \cdots , L_d ) : \mathbb{H} ^2 _d \otimes \mathbb{C} ^d \rightarrow \mathbb{H} ^2 _d$ is an isometry from several copies of $\mathbb{H} ^2 _d$ into itself. That is, the \emph{left free shift}, $L$, is a \emph{row isometry} and this plays the role of the shift in this non-commutative and multi-variable Hardy space theory. It follows that we can view Fock space, $\mathbb{H} ^2 _d$, as the \emph{free} or \emph{NC Hardy space} and $\mathbb{H} ^\infty _d$ as the \emph{NC Hardy algebra}. 

Toeplitz operators in this NC setting were first defined and studied by Popescu \cite{Pop-multi,Pop-factor,Pop-entropy}. A bounded operator $T \in \mathscr{L} (\mathbb{H} ^2 _d )$ will be called \emph{left Toeplitz} if it obeys the left NC Brown--Halmos Toeplitz condition:
$$ L_j^* T L_k = \delta _{j,k} T. $$ Popescu also initiated the study of factorization of positive left Toeplitz operators in \cite{Pop-multi,Pop-factor,Pop-entropy}. Namely, in addition to the left free shifts, $L_k := M^L _{\mathfrak{z} _k}$, one can also define the right free shifts, $R_k := M^R _{\mathfrak{z} _k}$, on $\mathbb{H} ^2 _d$. As before, these are isometries with pairwise orthogonal ranges and the weakly-closed, unital operator algebra they generate can be identified with the \emph{right multiplier algebra} of Fock space. Here, a formal power series, $f (\mathfrak{z} )$, is said to be a \emph{right multiplier} if right multiplication by $f (\mathfrak{z} )$ maps $\mathbb{H} ^2 _d$ into itself. A positive left Toeplitz operator, $T>0$, is then said to be \emph{factorable} if there is a right multiplier, $f (\mathfrak{z} )$, so that $T = (M^R _f) ^* M^R _f$. There is an exact version of the inner--outer factorization in this NC setting and any right (or left) multiplier, $f$, factors as $\Theta \cdot h$, where $\Theta$ is an \emph{inner}, \emph{i.e.} isometric right multiplier and $h$ is \emph{outer}, \emph{i.e.} $M^R _h$ has dense range \cite[Theorem 4.2]{Pop-factor} \cite[Corollary 2.2]{DP-inv}. Hence if $T =(M^R _f) ^* M^R _f$ then $T = (M^R _h ) ^* M^R _h$, and we can assume that $f =h$ is outer. Here, observe that right multipliers and left multipliers commute, so that any operator of the form $T = (M^R _f) ^* M^R _f$ is left Toeplitz and positive. By applying his NC Wold decomposition for row isometries and using a similar operator--theoretic proof as in that of a corresponding weakened form of Theorem \ref{Opszego}, Popescu established the following factorization results in Theorem \ref{Pop1} below \cite{Pop-entropy}. In the statement below, the \emph{entropy} of a positive semi-definite left Toeplitz operator, $T$, is defined as the Szeg\"o--type quantity,
$$ \mathrm{e} (T) := \mathrm{log} \, \mathrm{inf} _{p \in \mathbb{C} \{ \mathbb{\mathfrak{z}} \}  _0} \| \sqrt{T} (1 -p ) \| ^2 _{\mathbb{H} ^2} \in \mathbb{R} \cup \{ - \infty \}, $$ where $\mathbb{C} \{ \mathbb{\mathfrak{z}} \}  = \mathbb{C} \{ \mathfrak{z} _1 , \cdots , \mathfrak{z} _d \}$ denotes the ring of free or non-commutative complex polynomials in several NC variables and $\mathbb{C} \{ \mathbb{\mathfrak{z}} \}  _0$ denotes the free polynomials which vanish at $0 := (0, \cdots , 0 ) \in \mathbb{B} ^d _1$. Equivalently, $\mathbb{C} \{ \mathbb{\mathfrak{z}} \}  _0$ is the ring of all complex free polynomials with vanishing constant term, $\hat{p} _\emptyset =0$. 

\begin{thmA}[Popescu] \label{Pop1}
A positive left Toeplitz operator, $T \in \mathscr{L} (\mathbb{H} ^2 _d )$, has a non-zero factorable minorant if and only if it has finite entropy $\mathrm{e} (T) > - \infty$. If a positive left Toeplitz operator is bounded below, $T \geq \epsilon I >0$, $\epsilon >0$, then it is factorable.
\end{thmA}
If $T \geq T_0 \geq 0$ are positive semi-definite left Toeplitz operators and $T_0 =(M^R _h ) ^* (M ^R _h)$ is a factorable left Toeplitz operator bounded above by $T$, where $h$ is outer, we say that $T_0$ is a \emph{factorable Toeplitz minorant} or, more simply, a \emph{factorable minorant} of $T$. If $T_0$ is maximal, it is called the \emph{maximal factorable minorant} and $h$ is the \emph{maximal outer factor}. A natural question, posed and left unresolved in \cite{Pop-entropy} is whether Theorem~\ref{Opszego} holds in this setting: %, as in the classical setting of Hardy space, any positive left Toeplitz operator with non-trivial factorable minorant is necessarily factorable.
\begin{quest} \label{openQ}
Let $T \geq 0$ be a positive semi-definite left Toeplitz operator on the Fock space. If $T$ has a non-trivial factorable minorant, is $T$ factorable?
\end{quest}
While we do not have an answer to this question at this time, we will provide a detailed analysis and our results, Theorem \ref{main1} and Corollary \ref{answer}, will show that this question is equivalent to several open questions on (column) extreme points of $[\mathbb{H} ^\infty _d ]_1$, free de Branges--Rovnyak spaces, purity of free Gleason solutions and free Smirnov multipliers of $\mathbb{H} ^2 _d$.

One of the central results of the paper is the following theorem, where we explicitly construct the maximal outer factor, $a$ of $I - (M^R _b) ^* M^R _b$, which we call the \emph{free} or \emph{NC Sarason function of $b$}, where $b$ is a non-column--extreme contractive left multiplier of the Fock space and we generalize several results of D. Sarason on the relationship between $b$ and $a$, see Section \ref{sect:CE} and Corollary \ref{Sarmaxouter}.
\begin{thm*}
Given a non-CE $b \in [\mathbb{H} ^\infty _d ]_1$, let $\mathscr{H} ^\mathrm{t} ( b) := \mathscr{M} \, \sqrt{I - b(R) b(R) ^*}$ be the right free de Branges--Rovnyak space of $b$, let $X := L^* | _{\mathscr{H} ^\mathrm{t} (b)}$ and set $\boldsymbol{b} := L^* b^\mathrm{t} \in \mathscr{H} ^\mathrm{t} (b) \otimes \mathbb{C} ^d$. Define a formal power series, 
$ a (\mathfrak{z} ) := \sum _{\omega \in \mathbb{F} ^d} \hat{a} _\omega \mathfrak{z} ^\omega$, by 
$$ a(0) = \hat{a} _\emptyset := 1 - |b(0) |^2 - \| \boldsymbol{b} \| ^2 _{\mathscr{H} ^\mathrm{t} (b)} >0 \quad \quad  \mbox{and} \quad \quad  \hat{a}_\omega := - a(0) \left\langle b^\mathrm{t}, X^\omega b^\mathrm{t}\right\rangle_{\mathscr{H} ^\mathrm{t} (b)}, \ \omega \neq \emptyset. $$ Then $a$ belongs to $[\mathbb{H} ^\infty _d ]_1$, $a$ is outer and $a(R) ^* a(R)$ is the {\bf maximal} factorable left Toeplitz minorant of $I - b(R) ^* b(R)$ so that $c := \left( \begin{smallmatrix} b \\ a \end{smallmatrix} \right)$ is CE and 
$$ a(0) ^2 = \inf_{p \in \mathbb{C} \{ \mathbb{\mathfrak{z}} \}  _0} \left\langle 1-p , (I-b(R)^* b(R)) (1-p) \right\rangle_{\mathbb{H} ^2}. $$
\end{thm*}
In the above statement, $b^\mathrm{t} \in \mathbb{H} ^2 _d$ is the NC function obtained by reversing the order of products in all NC monomials in the NC Taylor series of $b$ and $b(R) = M^R _{b ^\mathrm{t}}$. The above formulas for the Taylor series coefficients of $a$ recover Sarason's formulas when $d=1$ \cite{Sarason-afun}. As described above, Theorem \ref{main1} then connects factorizability of positive semi-definite left Toeplitz operators, graphs of free Smirnov multipliers and doubly--free shift invariant de Branges--Rovnyak spaces to the condition that $c := \left( \begin{smallmatrix} b \\ a \end{smallmatrix} \right)$ is inner:
\begin{thm*}[Theorem \ref{main1}]
Given a non-column--extreme left multiplier, $b \in [ \mathbb{H} ^\infty _d ] _1$, with outer Sarason function, $a \in [ \mathbb{H} ^\infty _d ] _1$, let $\left( \begin{smallmatrix} B \\ A \end{smallmatrix} \right) D := CD$ be the inner--outer factorization of $c:= \left( \begin{smallmatrix} b \\ a \end{smallmatrix} \right)$. The following are equivalent: 
\begin{enumerate}
\item[\emph{(i)}] $c = \left( \begin{smallmatrix} b \\ a \end{smallmatrix} \right)$ is inner so that $D \equiv 1$. 
\item[\emph{(ii)}] If $X := L^* | _{\mathscr{H} ^\mathrm{t} (b)}$, then $X^*$ is \emph{pure}, \emph{i.e.} $ \lim \sum _{|\omega | =n} \| X ^\omega h \| _{\mathscr{H} ^\mathrm{t} (b)} =0 $ for any $h \in \mathscr{H} ^\mathrm{t} (b)$. 
\item[\emph{(iii)}] $X^*$ obeys the \emph{weak purity condition}, $ \sum _{|\omega | =n } \| X^\omega b ^\mathrm{t} \|_{\mathscr{H} ^\mathrm{t} (b)} ^2 \rightarrow 0$.
\item[\emph{(iv)}] The free polynomials are dense in $\mathscr{H} ^\mathrm{t} (b)$.
\item[\emph{(v)}] The right free de Branges--Rovnyak space, $\mathscr{H} ^\mathrm{t} (b)$, is the domain of the adjoint of the closed, densely--defined right Smirnov multiplier, $H(R) = b(R) a(R) ^{-1}$. 
\end{enumerate}
\end{thm*}

As a consequence of this theorem, Corollary \ref{answer} shows that Question \ref{openQ} has a positive answer if and only if the above equivalent conditions hold for every non-column--extreme $b \in [\mathbb{H} ^\infty _d ]_1$:

\begin{cor*}[Corollary \ref{answer}]
The following statements are equivalent:
\begin{itemize}
    \item[\emph{(i)}] The equivalent conditions of Theorem \ref{main1} hold for every non-column--extreme $b \in [\mathbb{H} ^\infty _d ] _1$.
    
    \item[\emph{(ii)}] A positive semi-definite left Toeplitz operator, $T \in \mathscr{L} (\mathbb{H} ^2 _d )$, is factorable if and only if it has a factorable minorant.

\end{itemize}
\end{cor*}

While the above theorem, Theorem \ref{main1}, is admittedly formidable and technical in appearance, it turns out that in many cases of interest its conditions are in fact checkable. In particular, it turns out to give very satisfactory results in the case of NC rational symbols $\mathfrak{b}$. In Section \ref{sect:NCratmult} we apply our results to the special case of contractive, non-commutative rational multipliers of the Fock space to establish the following dichotomy:
\begin{thm*}[Theorem \ref{CEisinner}]
If $\mathfrak{b}$ is a contractive, non-commutative rational multiplier of the Fock space, $\mathfrak{b}$ is either \emph{inner} or $\mathfrak{b}$ is not column--extreme. If $\mathfrak{b}$ is not column--extreme, then its NC outer Sarason function, $\mathfrak{a} \in [ \mathbb{H} ^\infty _d ] _1$, is also NC rational and the column $\mathfrak{c} := \left( \begin{smallmatrix} \mathfrak{b} \\ \mathfrak{a} \end{smallmatrix} \right)$ is inner.
\end{thm*}
This fact certainly holds in one variable as a consequence of the Fej\'er--Riesz theorem and we further show that a Fej\'er--Riesz theorem holds for NC rational multipliers of the Fock space, see Theorem \ref{ncFR}.
\begin{thm*}[NC rational Fej\'er--Riesz]
If $\mathfrak{r}$ is an NC rational left multiplier of Fock space and $T := \mathrm{Re} \, \mathfrak{r} (R) \geq 0$ is a positive semi-definite left Toeplitz operator, then $T$ is factorable, $T = \mathfrak{h} (R) ^* \mathfrak{h} (R)$, where $\mathfrak{h}$ is an outer NC rational multiplier.
\end{thm*}
In fact, we obtain a more detailed conclusion, which gives some ``degree'' control over $\mathfrak{h} $ (where ``degree'' is defined in a suitable sense, namely as the size of a so-called {\em minimal realization}), and we also obtain some control over the domain of regularity of $\mathfrak{h}$. We prove the theorem by converting the problem to an instance of computing a Sarason function $a$, and read off consequences from the de Branges--Rovnyak realization. As a consequence, we obtain an exact analogue of a classical corollary to Fatou's theorem for contractive NC rational multipliers in Corollary \ref{ncratFatcor}: A contractive NC rational left multiplier $\mathfrak{b} \in [ \mathbb{H} ^\infty _d ] _1$ is inner if and only if its NC Clark measure is singular with respect to NC Lebesgue measure, $m$. Here, NC measures are defined as positive linear functionals on the \emph{free disk system}, $\mathscr{A} _d := (\mathbb{A} _d + \mathbb{A} _d ^* ) ^{-\| \cdot \|}$ where $\mathbb{A} _d := \mathrm{Alg} \{ I , L_1 , \cdots , L_d \} ^{-\| \cdot \|}$ is the \emph{free disk algebra}. When $d=1$, NC measures can be identified with positive linear functionals on the $C^*-$algebra of continuous functions on the unit circle and hence with positive, finite and regular Borel measures by the Riesz--Markov theorem.  Theorem \ref{ncratRN} then provides an explicit formula for the non-commutative Radon--Nikodym derivative of the absolutely continuous part of $\mu _{\mathfrak{b}}$ with respect to a canonical NC Lebesgue measure.

\section{Background and Notation} \label{sect:back}

\subsection{Multipliers of Fock space}

Left multipications by the $d$ independent NC variables, $\mathfrak{z} = ( \mathfrak{z} _1 , \cdots , \mathfrak{z} _d )$, define isometries on the Fock space with pairwise orthogonal ranges:
$$ L_k := M^L _{\mathfrak{z} _k}, \quad \quad L_j ^* L_k = \delta _{j,k} I. $$ It follows that the row $d-$tuple $L := (L_1 , \cdots , L_d ) : \mathbb{H} ^2 _d \otimes \mathbb{C} ^d \rightarrow \mathbb{H} ^2 _d$ is an isometry from several copies of $\mathbb{H} ^2 _d$ into itself. Such an isometry is called a \emph{row isometry}, we call $L$ the \emph{left free shift} and its components left free shifts. Similarly, one can define the right free shifts $R_k := M^R _{\mathfrak{z} _k}$ and the row isometric right free shift, $R$. The letter reversal map $\mathrm{t} : \mathbb{F} ^d \rightarrow \mathbb{F} ^d$, which reverses the order of letters in any word $\omega  \in \mathbb{F} ^d$ defines an involution on the free monoid,
$$  \omega = i_1 \cdots i_n \mapsto \omega ^\mathrm{t} := i_n \cdots i_1. $$ The free monomials $ \{ e_\omega := \mathfrak{z} ^\omega | \ \omega \in \mathbb{F} ^d \}$ define a standard orthonormal basis of $\mathbb{H} ^2 _d \simeq \ell ^2 (\mathbb{F} ^d )$ and the letter reversal map gives rise to a unitary involution of the Fock space, $U_\mathrm{t}$, defined by $U_\mathrm{t} \mathfrak{z} ^\omega = \mathfrak{z} ^{\omega ^\mathrm{t}}$. Here, $e_\emptyset = \mathfrak{z} ^\emptyset =: 1$ is called the vacuum vector of the Fock space. It is straightforward to verify that $U_\mathrm{t} L_k U_\mathrm{t} = R_k$, so that the left shifts are isomorphic to the right shifts.

The NC Hardy algebra, $\mathbb{H} ^\infty _d$, of uniformly bounded NC functions can be identified, completely isometrically, with the unital Banach algebra of left multipliers of the NC Hardy space, $\mathbb{H} ^2 _d$. That is, given any NC function, $F \in \mathbb{H} ^\infty _d$ and $h \in \mathbb{H} ^2 _d$, the left multiplication operator $M^L _F : \mathbb{H} ^2 _d \rightarrow \mathbb{H} ^2 _d$, defined by 
$$ h(Z) \mapsto F(Z) \cdot h(Z), $$ is bounded and $\| M ^L _F \|  = \| F \| _\infty$, where $\| \cdot \| _\infty$ denotes the supremum norm over $\mathbb{B} ^d _\mathbb{N}$ \cite[Theorem 3.1]{SSS}, \cite[Theorem 3.1]{Pop-freeholo}. For any free polynomial, $p \in \mathbb{C} \{ \mathfrak{z} \}$, one can check that $p(L) = M^L _p$ and so we employ the notation $F(L) := M^L _F$. Similarly, if $p \in \mathbb{C} \{ \mathbb{\mathfrak{z}} \} $, then $M^R _{p ^\mathrm{t}} = p  (R) = U_\mathrm{t} p(L) U_\mathrm{t}$, where if $h$ is a formal power series $h(\mathfrak{z} ) = \sum \hat{h} _\omega \mathfrak{z} ^\omega$,
$$ h^\mathrm{t} (\mathfrak{z} ) := \sum \hat{h} _\omega \mathfrak{z} ^{\omega ^\mathrm{t}} = \sum \hat{h} _{\omega ^\mathrm{t}} \mathfrak{z} ^\omega. $$ In particular, if $h \in \mathbb{H} ^2 _d$, $h^\mathrm{t} = U_\mathrm{t} h$. The left and right multiplier algebras of $\mathbb{H} ^2 _d$ are unitarily equivalent via the unitary letter reversal involution $U_\mathrm{t}$ and can be identified with the left and right analytic Toeplitz algebras, $\mathscr{L} ^\infty _d := \mathrm{Alg} \{ I , L_1 , \cdots , L_d \} ^{-WOT}$ and $\mathscr{R} ^\infty _d =\{ I , R_1 , \cdots , R_d \} ^{-WOT} = U_\mathrm{t} \mathscr{L} ^\infty _d U_\mathrm{t}$. Here, $WOT$ denotes the weak operator topology. Since $p (R) = M^R _{p ^\mathrm{t}}$ for any $p \in \mathbb{C} \{ \mathbb{\mathfrak{z}} \} $, we will write $G (R) = M^R _{G ^\mathrm{t}}$ for any $G \in \mathbb{H} ^\infty _d$. Namely $G$ belongs to the left multiplier algebra, $\mathbb{H} ^\infty _d$, if and only if $G^\mathrm{t}$ belongs to the right multiplier algebra, $\mathbb{H} ^{\infty; \mathrm{t}} _d := \mathrm{t} \circ \mathbb{H} ^\infty _d$. (If $F = F(R) = M^R _{F^\mathrm{t}}$ for $F \in \mathbb{H} ^\infty _d$, one can show that $\| F (R) \|$ is equal to the supremum norm of $F^\mathrm{t} (Z)$ over an \emph{NC unit column--ball}.) We will use the following terminology: A left or right multiplier is \emph{inner} if it is isometric and \emph{outer} if it has dense range.

\subsection{Non-commutative reproducing kernel Hilbert spaces}

As in classical Hardy space theory, the Fock space is a (non-commutative) reproducing kernel Hilbert space, in the sense that for any $Z \in \mathbb{B} ^d _n$ and vectors $y,v \in \mathbb{C} ^n$, the \emph{matrix--entry point evaluation}, $\ell _{Z,y,v} : \mathbb{H} ^2 _d \rightarrow \mathbb{C}$,
$$ h \mapsto y^* h(Z) v, $$ is a bounded linear functional. Equivalently, $h \mapsto h(Z)$ is bounded as a linear map from $\mathbb{H} ^2 _d$ into the Hilbert space $\mathbb{C} ^{n \times n}$ equipped with the Hilbert--Schmidt inner product. By the Riesz lemma, $\ell _{Z,y,v}$ is implemented by inner products against vectors $K \{ Z , y , v \} \in \mathbb{H} ^2 _d$ which we call \emph{NC Szeg\"o kernel vectors}. 

More generally, a Hilbert space, $\mathcal{H}$, of free non-commutative functions in $\mathbb{B} ^d _\mathbb{N}$ is a \emph{non-commutative reproducing kernel Hilbert space} (NC-RKHS), if the linear point evaluation functional, $\ell _{Z,y,v}$, is bounded on $\mathcal{H}$ for any $Z \in \mathbb{B} ^d _n$, $y,v \in \mathbb{C} ^n$ and $n \in \mathbb{N}$ \cite{BMV}. (NC-RKHS can, of course, be defined on general NC sets. However, all NC-RKHS in this paper are Hilbert spaces of NC functions in $\mathbb{B} ^d _\mathbb{N}$ so we omit the general definition.) As before, the Riesz lemma implies that these functionals are implemented by taking inner products against \emph{point evaluation} or \emph{NC kernel vectors} $k \{ Z, y ,v \} \in \mathcal{H}$. Given any such NC-RKHS, $Z \in \mathbb{B} ^d _n$ and $W \in \mathbb{B} ^d _m$, one can define a completely bounded linear map on $n \times m$ complex matrices by: $k(Z,W) [ \cdot ] : \mathbb{C} ^{n\times m} \rightarrow \mathbb{C} ^{n\times m}$,
$$  y^* k(Z,W) [vu^*] x := \left\langle k\{ Z,y,v \}, k \{ W,x ,u \} \right\rangle_{\mathcal{H}}, $$ and this map is completely positive if $Z=W$ \cite{BMV}. Following \cite{BMV}, we call $k(Z,W)[\cdot ]$ the \emph{completely positive non-commutative (CPNC) reproducing kernel} of $\mathcal{H}$ and we write $\mathcal{H} = \mathcal{H} _{nc} (k)$. One can check that adjoints of left and right multipliers of an NC-RKHS have a familiar action on NC kernel vectors:
$$ (M^L _F ) ^* k\{ Z , y ,v \} = k\{ Z , F(Z) ^* y , v \} \quad \mbox{and} \quad (M^R _G ) ^* k \{ Z , y ,v \} = k \{ Z , y , G (Z) v \}. $$ 

All NC-RKHS in this paper will be Hilbert spaces of free NC functions in the unit row-ball $\mathbb{B} ^d _\mathbb{N}$, 
$$ \mathbb{B} ^d _\mathbb{N} = \bigsqcup \mathbb{B} ^d _n, \quad \quad \mathbb{B} ^d _n := \left\{ \left. Z \in \mathbb{C} ^{n\times n} \otimes \mathbb{C}  ^{1\times d} \right| \, ZZ^* = Z_1 Z_1 ^* + \cdots + Z_d Z_d ^* < I_n \right\}. $$ In the case of the Fock space, $\mathbb{H} ^2 _d = \mathcal{H} _{nc} (K)$, where $K$ is the \emph{NC Szeg\"o kernel}: Given $Z \in \mathbb{B} ^d _n, W \in \mathbb{B} ^d _m$ and $P \in \mathbb{C} ^{n \times m}$,
$$ K(Z,W) := \left( \mathrm{id} _{n,m} [ \cdot ] - \mathrm{Ad} _{Z, W^*} [\cdot ] \right)  ^{-1} \circ P = \sum _{j=0} ^\infty \mathrm{Ad} _{Z,W^*} ^{(j)} [P] = \sum _{\omega \in \mathbb{F} ^d} Z^\omega P W^{*\omega}, $$ 
 $\mathrm{Ad} _{Z,W^*} [P] := Z_1 P W_1 ^* + \cdots + Z_d P W_d ^*$.

\section{Column--extreme multipliers and the Sarason outer function}
\label{sect:CE}

Recall that an element $b  \in [\mathbb{H} ^\infty _d ] _1$ is said to be \emph{column--extreme (CE)} if there is no non-zero $a \in \mathbb{H} ^\infty _d$ so that the two-component column:
$$ \begin{pNiceMatrix} b(L) \\ a(L) \end{pNiceMatrix} : \mathbb{H} ^2 _d \rightarrow \mathbb{H} ^2 _d \otimes \mathbb{C} ^2, $$ is a contractive left multiplier from one to two copies of Fock space \cite{JM-freeCE}. Any CE multiplier of Fock space is necessarily an extreme point of the closed convex set of contractive left multipliers \cite[Corollary 6.8]{JM-freeCE}. Classically, a multiplier of $H^2$ is an extreme point if and only if it is CE and this also holds in the commutative multi-variable setting of Drury--Arveson space as a consequence of the column--row property of M.P. Hartz \cite[Theorem 1.2, Theorem 1.6]{Hartz}. However, the column--row property does not hold in the NC setting of Fock space and so whether or not every extreme point of $[\mathbb{H} ^\infty _d ] _1$ is CE remains an open problem \cite{AJP-nofreeCR}.

This concept of column--extreme extends readily to the case of operator--valued left multipliers between vector--valued Fock spaces \cite{JM-freeCE}. Let $\mathcal{H}, \mathcal{J}$ be separable Hilbert spaces. We define $\mathbb{H} ^\infty _d \otimes \mathscr{L} (\mathcal{H} , \mathcal{J} )$ as the closure of this algebraic tensor product in the weak operator topology (WOT) of $\mathscr{L} \left( \mathbb{H} ^2 _d \otimes \mathcal{H} , \mathbb{H} ^2 _d \otimes \mathcal{J} \right)$. Elements of $\mathbb{H} ^\infty _d \otimes \mathscr{L} (\mathcal{H} , \mathcal{J} )$ can be viewed as operator--valued multipliers from $\mathbb{H} ^2 _d \otimes \mathcal{H}$ into $\mathbb{H} ^2 _d \otimes \mathcal{J}$. A contractive left multiplier $b \in [ \mathbb{H} ^\infty _d \otimes \mathscr{L} (\mathcal{H} , \mathcal{J} ) ] _1$ is said to be CE if and only if $a \in [ \mathbb{H} ^\infty _d \otimes \mathscr{L} (\mathcal{H} , \mathcal{J} ) ] _1$ and $c := \left( \begin{smallmatrix} b \\ a \end{smallmatrix} \right) \in [ \mathbb{H} ^\infty _d \otimes \mathscr{L} (\mathcal{H} , \mathcal{J} )  \otimes \mathbb{C} ^2 ] _1$ implies that $a \equiv 0$. 

Recall that any contractive $b \in [H^\infty ] _1$ has a \emph{de Branges--Rovnyak realization}:
$$ b(z) = D + C (I-zA) ^{-1} zB, \quad \quad z \in \mathbb{D}, $$ where, 
$$ A := S^* | _{\mathscr{H} (b)}, \quad B := S^* b, \quad C := (k_0 ^b) ^* \quad \mbox{and} \quad D:= b(0), $$ and $\mathscr{H} (b)$ is the de Branges--Rovnyak space of $b$ \cite[Theorem 1.1, Theorem 1.2, Theorem 1.3]{BBF}. This space is a Hilbert space of analytic functions in $\mathbb{D}$ which is contractively contained in $H^2$, and it is a RKHS with reproducing kernel:
$$ k^b (z,w) := \frac{1-b(z) b(w) ^*}{1-zw^*}. $$ In particular, $b$ is inner if and only if $\mathscr{H} (b)$ is contained isometrically in $H^2$ and in this case $\mathscr{H} (b) = (bH ^2 ) ^\perp$. The contractive multiplier, $b$, is said to be the \emph{transfer function} of the (observable, co-isometric) \emph{de Branges--Rovnyak colligation}, $U_b$, defined by:
$$ U_b := \begin{pNiceMatrix} A & B \\ C & D \end{pNiceMatrix} : \begin{pNiceMatrix} \mathscr{H} (b ) \\ \mathbb{C} \end{pNiceMatrix} \rightarrow \begin{pNiceMatrix} \mathscr{H} (b) \\ \mathbb{C} \end{pNiceMatrix}, $$ see \cite{BBF}. 

These de Branges--Rovnyak realizations were extended to the NC multi-variable setting of contractive multipliers between vector--valued Fock spaces in \cite{BBF}. Given $b \in [ \mathbb{H} ^\infty _d \otimes \mathscr{L} ( \mathcal{H} , \mathcal{J} ) ] _1$, we define the \emph{right free de Branges--Rovnyak space} of $b$, $\mathscr{H} ^\mathrm{t} (b)$, as the complementary space of $b(R)$, $\mathscr{H} ^\mathrm{t} (b) := \mathscr{M} \left( \sqrt{I - b(R) b(R) ^*} \right)$. We will sometimes denote the norm and inner product of $\mathscr{H} ^\mathrm{t} (b)$ by $\| \cdot \| _b := \| \cdot \| _{\mathscr{H} ^\mathrm{t} (b)}$ and $\left\langle \cdot, \cdot \right\rangle _b := \left\langle \cdot, \cdot \right\rangle_{\mathscr{H} ^\mathrm{t} (b)}$, respectively. We will also employ the notations $\mathscr{M} ^\mathrm{t} (b) := \mathscr{M} (b (R) )$, $\mathscr{M} ^\mathrm{t} (b^*) := \mathscr{M} (b(R) ^*)$ and $\mathscr{H} ^\mathrm{t} (b ^* ) := \mathscr{M} \, \sqrt{I - b(R) ^* b(R)}$, where recall that if $A \in \mathscr{L} (\mathcal{H} , \mathcal{J})$ is a bounded linear operator, $\mathscr{M} (A)$ is the operator--range space of $A$. The right free de Branges--Rovnyak space, $\mathscr{H} ^\mathrm{t} (b)$, is a $\mathcal{J}-$valued NC-RKHS on $\mathbb{B} ^d _\mathbb{N}$ with CPNC kernel
$$ K ^b (Z,W) [P \otimes I_\mathcal{H}] := K (Z,W)[P] \otimes I_\mathcal{H} - K(Z,W) \otimes \mathrm{id} _{\mathcal{H}} [ b^\mathrm{t} (Z) (P \otimes I _{\mathcal{H}}) b^\mathrm{t} (W) ^* ], $$ for $Z \in \mathbb{B} ^d _n$, $W \in \mathbb{B} ^d _m$ and $P \in \mathbb{C} ^{n \times m}$. Set $X := L^* \otimes I_\mathcal{J} | _{\mathscr{H} ^\mathrm{t} (b)}$ and $\boldsymbol{b} := L^* \otimes I _{\mathcal{J}} \, b (R) 1 \otimes I _\mathcal{H}  \in \mathscr{L} \left( \mathcal{H} ,  \mathscr{H} ^\mathrm{t} (b) \otimes \mathbb{C} ^d \right)$. Namely,
$$ \boldsymbol{b} \, h := L^* \otimes I _{\mathcal{J}} \, b (R)  1 \otimes h \in \mathscr{H} ^\mathrm{t} (b) \otimes \mathbb{C} ^d, $$ for any $h \in \mathcal{H}$. Here, in parallel with the classical theory, any right free de Branges--Rovnyak space, $\mathscr{H} ^\mathrm{t} (b)$ for $b \in [ \mathbb{H} ^\infty _d ] _1$ is contractively contained in $\mathbb{H} ^2 _d$, is left shift co-invariant and while $b ^\mathrm{t}$ generally does not belong to $\mathscr{H} ^\mathrm{t} (b)$, its backward left shifts always do, $L^* b^\mathrm{t} \in \mathscr{H} ^\mathrm{t} (b ) \otimes \mathbb{C} ^d$ \cite[Proposition 4.2]{BBF}. Analogous statements hold for operator--valued $b$.

Any such $b \in [ \mathbb{H} ^\infty _d \otimes \mathscr{L} (\mathcal{H} , \mathcal{J} ) ] _1$ is realized as the transfer--function of the co-isometric right \emph{free de Branges--Rovnyak colligation} $U_b$, defined by 
\begin{align*} U_b &:= \begin{pNiceMatrix} A & B \\ C & D \end{pNiceMatrix} : \begin{pNiceMatrix} \mathscr{H} ^\mathrm{t} ( b) \\ \mathcal{H} \end{pNiceMatrix} \rightarrow \begin{pNiceMatrix} \mathscr{H} ^\mathrm{t} (b) \otimes \mathbb{C} ^d \\ \mathcal{J} \end{pNiceMatrix}, \quad  \mbox{where} \quad
 A := X := L^* \otimes I _\mathcal{J} | _{\mathscr{H} ^\mathrm{t} (b)}, \\ \quad B & := L^* \otimes I_\mathcal{J} \,  b(R) \,  1 \otimes I_\mathcal{H} =\boldsymbol{b}, \quad 
C := (K ^b _0 ) ^* \quad \mbox{and} \quad D := b(0), \end{align*} see \cite{BBF}. For any $b \in [ \mathbb{H} ^\infty _d \otimes \mathscr{L} (\mathcal{H} , \mathcal{J} ) ] _1$, recall that 
$$ X^* X \leq  I - K_0 ^b (K_0 ^b)^* \quad \mbox{and} \quad \boldsymbol{b}^* \boldsymbol{b} \leq I_\mathcal{H} - b(0) ^* b(0), $$ \cite[Proposition 4.2]{BBF}. Further recall that $b$ is column--extreme if and only if equality holds in either (and hence both) of these formulas \cite[Theorem 6.4]{JM-freeCE}. That is, $b$ is CE if and only if
$$ X ^* X = I - K_0 ^b (K_0 ^b)^*, $$ or equivalently,
$$ \boldsymbol{b} ^* \boldsymbol{b} = I _\mathcal{H} - b(0) ^* b(0). $$  
Given $X := L^* \otimes I_\mathcal{J} | _{\mathscr{H} ^\mathrm{t} (b)}$ for some $b \in [ \mathbb{H} ^\infty _d  \otimes \mathscr{L} (\mathcal{H} , \mathcal{J} )] _1$, let
$$ X ^{(n)} :=  (X \otimes I_d \otimes I_{n-1}) \cdots (X \otimes I_d \otimes I_d ) (X \otimes I_d) X, $$ so that for any $h \in \mathscr{H} ^\mathrm{t} (b)$, 
\begin{equation} \| X ^{(n)} h \| ^2 _b = \sum _{|\alpha | = n} \| X ^\alpha h \| ^2 _b = \left\langle h, \mathrm{Ad} _{X^*, X} ^{(n)} (I) h \right\rangle_b, \label{amplify} \end{equation} where recall that $\| \cdot \| _b$ and $\left\langle \cdot, \cdot \right\rangle_b$ denote the norm and inner product in $\mathscr{H} ^\mathrm{t} (b)$. Here, given a row $d-$tuple of operators, $T = ( T_1 , \cdots , T_d ) : \mathcal{K} \otimes \mathbb{C} ^d \rightarrow \mathcal{K}$, the completely positive linear map of \emph{adjunction by $T$ and $T^*$}, $\mathrm{Ad} _{T, T^*} : \mathscr{L} (\mathcal{K} ) \rightarrow \mathscr{L} (\mathcal{K} )$, is defined as
\begin{equation} \mathrm{Ad} _{T, T^*} ( A ) := T_1 A T_1^* + \cdots + T_d A T_d ^*. \label{adjunction} \end{equation} The row $d-$tuple, $T$, is said to be \emph{pure} if $\mathrm{Ad} ^{(n)} _{T, T^*} (I) \stackrel{SOT}{\rightarrow} 0$, where $SOT$ denotes the strong operator topology.

\begin{thm} \label{pureisinner}
Let $b \in [ \mathbb{H} ^\infty _d \otimes \mathscr{L} (\mathcal{H} , \mathcal{J} ) ] _1$ be column--extreme (CE). Then $b$ is inner if and only if $X^*$ obeys the weak purity condition, $\| X ^{(n)} \boldsymbol{b} h \| \rightarrow 0$ for any $h \in \mathcal{H}$. In particular, $X^*$ is pure if and only if $b$ is inner.
\end{thm}
By \cite[Theorem 4.6]{BBF}, $b \in [ \mathbb{H} ^\infty _d \otimes \mathscr{L} (\mathcal{H} , \mathcal{J} ) ] _1$ is inner if and only if $X^*$ is pure (or \emph{strongly--stable} in the language of \cite{BBF}). The above theorem is slight weakening of this assumption in the case of a CE multiplier -- instead of checking that $\mathrm{Ad} _{X^*, X} ^{(n)} (I ) x$ converges to $0$ for any $x \in \mathscr{H} ^\mathrm{t} (b)$, it suffices to show that $ \| X^{(n)} \boldsymbol{b} h \| ^2 = \left\langle  \boldsymbol{b} h, \mathrm{Ad} _{X^*,X} ^{(n)} (I ) \boldsymbol{b} h \right\rangle \rightarrow 0$ for any $h \in \mathcal{H}$. (Here, recall that $WOT$ convergence of a sequence of self-adjoint operators to $0$ implies $SOT$ convergence.) In particular, if $b \in [\mathbb{H} ^\infty _d ] _1$ is scalar, it suffices to check that $\| X ^{(n)} \boldsymbol{b} \| \rightarrow 0$.
\begin{proof}
If $b$ is inner then $\mathscr{H} ^\mathrm{t} (b) = \mathrm{Ran} \, b(R)  \, ^\perp \subseteq \mathbb{H} ^2 _d \otimes \mathcal{J}$, so that $\mathscr{H} ^\mathrm{t} (b)$ is a closed, co-invariant subspace for $L \otimes I_\mathcal{J}$. It is then easily checked that $X^*$ is pure since $L \otimes I_\mathcal{J}$ is.

Conversely, if $b$ is CE then we have that $X ^* X = I - K_0 ^b (K_0 ^b ) ^*$. Suppose now that $X$ is pure and consider $\| X^{(n)} \boldsymbol{b} h \| ^2 _b$ where $\boldsymbol{b} h = L^* \otimes I_\mathcal{J} \, b (R) 1 \otimes h \in \mathscr{H} ^\mathrm{t} (b) \otimes \mathbb{C} ^d$ and $\| h \| _\mathcal{H} =1$. Hence $\| X^{(n)} \boldsymbol{b} h \| ^2 _b \rightarrow 0$ as $n \uparrow + \infty$. Then,
\begin{eqnarray} \| X^{(n)} \boldsymbol{b} h \| ^2 & = & \left\langle  X^{(n-1)} \boldsymbol{b} h, X^*X \otimes I_d \otimes I_{n-1} X^{(n-1)} \boldsymbol{b} h \right\rangle_b \nonumber \\ & = & \| X^{(n-1)} \boldsymbol{b} h \| ^2 _b -  \left\langle  X^{(n-1)} \boldsymbol{b} h, K_0 ^b (K_0 ^b)^* \otimes I_d \otimes I_{n-1} X^{(n-1)} \boldsymbol{b} h \right\rangle _b  \nonumber \\
& = & \| X^{(n-1)} \boldsymbol{b} h \| ^2 _b - \left\langle  L^{*(n)} b(R) 1 \otimes h, K_0 ^b (K_0 ^b)^* \otimes I_d \otimes I_{n-1} L^{*(n)} b (R) 1 \otimes h \right\rangle _b \nonumber \\
& = & \| X ^{(n-1)} \boldsymbol{b} h \| ^2 _b - \sum _{|\alpha | = n} \| \hat{b} _\alpha h  \| _\mathcal{J} ^2 \nonumber \\
\cdots & = & \| \boldsymbol{b} h \| ^2 _b - \sum _{|\alpha | =1} ^n \| \hat{b} _\alpha h \| _\mathcal{J} ^2 \nonumber \\
& =& 1 - \sum _{|\alpha | \leq n} \| \hat{b} _\alpha h \| _{\mathcal{J}} ^2 \nonumber \\
& \rightarrow & 1 - \sum _{\alpha} \| \hat{b} _\alpha h  \| _{\mathcal{J} } ^2 = 1 - \| b(R) 1 \otimes h \| _{\mathbb{H} ^2 \otimes \mathcal{J}} ^2. \nonumber \end{eqnarray} In the above we used that if $b$ is CE, then $\| \boldsymbol{b} h \| ^2 _b = 1 - \| b(0) h \| _\mathcal{J} ^2$. Since $X^*$ is pure, it follows that $\| b (R) 1 \otimes h \| ^2 _{\mathbb{H} ^2 \otimes \mathcal{J}} = 1$. Since we also have that $b \in [ \mathbb{H} ^\infty _d \otimes \mathscr{L} (\mathcal{H} , \mathcal{J} ) ] _1$, an observation of Davidson--Pitts implies that $b$ is inner \cite[Proposition 2.4]{DP-inv}. Indeed, if $b \in [ \mathbb{H} ^\infty _d \otimes \mathscr{L} (\mathcal{H} , \mathcal{J} ) ] _1$ is a contractive multiplier so that $\| b(R) 1 \otimes h \| _{\mathbb{H} ^2 \otimes \mathcal{J} } =1$ for any unit norm $h \in \mathcal{H}$, then for any $\omega \in \mathbb{F} ^d$, 
\begin{eqnarray} \| b(R) L^\omega 1 \otimes h \| ^2 _{\mathbb{H} ^2 \otimes \mathcal{J} } & = & \| L ^\omega \otimes I_\mathcal{J} \,  b(R) 1 \otimes h \| ^2 _{\mathbb{H} ^2 \otimes \mathcal{J} } \nonumber \\
& = & \| b (R) 1 \otimes h  \| ^2   =1 = \| L^\omega 1 \otimes h \| ^2  _{\mathbb{H} ^2 \otimes \mathcal{H}}. \nonumber \end{eqnarray} Hence,
\begin{eqnarray} 1 & = & \left\langle L^\omega 1 \otimes h, b(R) ^* b(R) L^\omega 1 \otimes h \right\rangle_{\mathbb{H} ^2 \otimes \mathcal{H}} \nonumber \\
& \leq & \| L ^\omega 1 \otimes h \|  \| b(R) ^* b(R) L^\omega 1 \otimes h \|  \nonumber \\
& \leq & \| b (R) ^* b(R) \| \| L^\omega 1 \otimes h \| \leq 1. \nonumber \end{eqnarray} 
Since equality holds in the Cauchy--Schwarz inequality, we must have that $b(R) ^* b(R) L^\omega 1 \otimes h = \zeta L^\omega 1 \otimes h$, for some $\zeta \in \mathbb{C}$. Since $b(R) ^* b(R)$ is positive semi-definite, $\zeta \geq 0$ and since $\| b (R) L^\omega 1 \otimes h \| =1$, $\zeta =1$. It follows that for any free polynomial, $p \in \mathbb{C} \{ \mathbb{\mathfrak{z}} \} $ and unit norm $h \in \mathcal{H}$, $$ b(R) ^* b(R) p(L) 1 \otimes h = b(R) ^* b(R) \sum _{|\omega | \leq N} \hat{p} _{\omega} L^\omega 1 \otimes h = p (L) 1 \otimes h. $$ By the density of the free polynomials in Fock space, $b(R)^* b(R) =I$ and $b$ is inner.
\end{proof}
\begin{prop} \label{rowuniCE}
$b \in [\mathbb{H} ^\infty _d  \otimes \mathscr{L} (\mathcal{H} , \mathcal{J} ) ] _1$ is column--extreme if and only if its canonical de Branges--Rovnyak colligation is both isometric and co-isometric.
\end{prop}
\begin{proof}
For simplicity, assume that $b \in [ \mathbb{H} ^\infty _d ] _1$. Proof of the general case is analogous.
The de Branges--Rovnyak colligation of $b$ is always co-isometric \cite{BBF}. Here, the colligation is:
$$ U_b := \begin{pNiceMatrix} A & B \\ C & D \end{pNiceMatrix} : \begin{pNiceMatrix} \mathscr{H} ^\mathrm{t} (b) \\ \mathbb{C} \end{pNiceMatrix} \rightarrow \begin{pNiceMatrix} \mathscr{H} ^\mathrm{t} (b) \otimes \mathbb{C} ^d \\ \mathbb{C} \end{pNiceMatrix}, $$ 
$$ A := X := L^* | _{\mathscr{H} ^\mathrm{t} (b)}, \ B := \boldsymbol{b} := L^* b ^\mathrm{t} \in \mathscr{H} ^\mathrm{t} (b) \otimes \mathbb{C} ^d, \ C := (K_0 ^b ) ^*, \ D := b(0). $$ 
As we know, $b$ is CE if and only if $ X ^* X = I - K_0 ^b (K_0 ^b) ^*, $ or equivalently,
$ \boldsymbol{b} ^* \boldsymbol{b} = 1 - |b(0)| ^2. $ So now, we calculate:
$$ U_b ^* U_b = \begin{pNiceMatrix} \overbrace{X^* X + K_0 ^b (K_0 ^b) ^*}^{=I} & X^* \boldsymbol{b} + K_0 ^b b(0) \\ * & \underbrace{\boldsymbol{b} ^* \boldsymbol{b} + |b(0)| ^2}_{=1} \end{pNiceMatrix}. $$ It remains to show the off-diagonal components vanish. Indeed,
\begin{eqnarray} (K_Z ^b) ^* \left( X^* \boldsymbol{b} + K_0 ^b b(0) \right) & = & (X K_Z ^b) ^* \boldsymbol{b} + (K_Z ^b ) ^* K_0 ^b b(0) \nonumber \\
& = & \left( K_Z ^b Z^* - \boldsymbol{b} b ^\mathrm{t} (Z) ^* \right) ^* \boldsymbol{b}  + (I - b ^\mathrm{t} (Z) b(0) ^*) b(0) \nonumber \\
& = & Z \boldsymbol{b} (Z) - b ^\mathrm{t} (Z) \boldsymbol{b}^* \boldsymbol{b} + b(0) - b^\mathrm{t} (Z) |b(0)| ^2 \nonumber \\
& = & b ^\mathrm{t} (Z) - b(0) - b ^\mathrm{t} (Z) + b  ^\mathrm{t} (Z) |b(0)| ^2 + b(0) - b ^\mathrm{t} (Z) |b(0) | ^2 \nonumber \\
& = & 0. \nonumber \end{eqnarray} 
\end{proof}

\subsection{The Sarason function}

In one variable, as discussed in the introduction, $b \in [H^\infty ] _1$ is not an extreme point if and only if $1 -|b| ^2$ is log--integrable on the circle and in this case there is a unique outer function, $a \in H^\infty$,  so that $a(0) >0$ and
$$ |a (\zeta ) | ^2 = 1 - | b (\zeta ) | ^2; \quad \quad \zeta \ a.e. \ \partial \mathbb{D}. $$ This implies, in turn, that $c := \left( \begin{smallmatrix} b \\ a \end{smallmatrix} \right)$ is inner.
If $b \in [H^\infty ] _1$ is non-extreme, D. Sarason showed that this unique outer $a$ is realized as the transfer--function of the colligation:
$$ U_a := \begin{pNiceMatrix}[margin] \rowcolor{green!15} A & B \\ \rowcolor{blue!15} C_a & D_a \end{pNiceMatrix} : \begin{pNiceMatrix} \mathscr{H}  (b) \\ \mathbb{C} \end{pNiceMatrix} \rightarrow \begin{pNiceMatrix} \mathscr{H}  (b) \\ \mathbb{C} \end{pNiceMatrix}, $$ where 
$$ U_b = \begin{pNiceMatrix}[margin] \rowcolor{green!15} A & B \\ \rowcolor{green!15} C & D \end{pNiceMatrix} : \begin{pNiceMatrix} \mathscr{H}  (b) \\ \mathbb{C} \end{pNiceMatrix} \rightarrow \begin{pNiceMatrix} \mathscr{H}  (b) \\ \mathbb{C} \end{pNiceMatrix}, $$ is the de Branges--Rovnyak realization of $b$,
$$ C_a := -a(0) \left\langle b, \cdot \right\rangle_{\mathscr{H} (b)}, \quad D_a := a(0) >0 \quad \mbox{and} \quad a(0) ^2 := 1 - \| S^* b \| ^2 _{\mathscr{H} (b)} - |b(0) |^2 >0.$$ The column multiplier, $c = \left( \begin{smallmatrix} b \\ a \end{smallmatrix} \right)$, is then the transfer--function of the colligation:
$$ U_c := \begin{pNiceMatrix}[margin] \rowcolor{green!15} A & B \\ \rowcolor{green!15} C & D \\ \rowcolor{blue!15} C_a & D _a \end{pNiceMatrix} : \begin{pNiceMatrix} \mathscr{H} (b) \\ \mathbb{C}  \end{pNiceMatrix} \rightarrow \begin{pNiceMatrix} \mathscr{H} (b) \\ \mathbb{C} \\ \mathbb{C} \end{pNiceMatrix}. $$ (Sarason didn't state his results in the language of realization theory, but he computes the Taylor series coefficients of $a$ and these are easily seen to coincide with the formula given by the above realization \cite{Sarason-afun}.)  

Given a non-CE $b \in [\mathbb{H} ^\infty _d ]_1$ (or more generally any non-CE $b \in [ \mathbb{H} ^\infty _d \otimes \mathscr{L} (\mathcal{H} , \mathcal{J} ) ] _1$) we can construct an $a \in [\mathbb{H} ^\infty _d ] _1$ so that the column, $c = \left( \begin{smallmatrix} b \\ a \end{smallmatrix} \right) \in [\mathbb{H} ^\infty _d \otimes \mathbb{C} ^2 ]_1$ is contractive by checking that the analogous colligation:
$$ U_c := \begin{pNiceMatrix}[margin] \rowcolor{green!15}  \quad \quad  \mbox{---}  \quad  \boxed{U_b} & \quad  \mbox{---}  \\ \rowcolor{blue!15} -a(0) \left\langle b ^\mathrm{t}, \cdot \right\rangle_{\mathscr{H} ^\mathrm{t} (b)} & a(0) \end{pNiceMatrix} : \begin{pNiceMatrix} \mathscr{H} ^\mathrm{t} (b) \\ \mathbb{C}  \end{pNiceMatrix} \rightarrow \begin{pNiceMatrix} \mathscr{H} ^\mathrm{t} (b) \otimes \mathbb{C} ^d \\ \mathbb{C} ^2 \end{pNiceMatrix}, $$ is contractive. Here, as above, we define $a(0) >0$ and 
$$ a(0) ^2 := 1 -|b(0)| ^2 - \boldsymbol{b} ^* \boldsymbol{b} >0. $$ In fact, in \cite{JM-freeCE}, we check that $U_c$ is both isometric and co-isometric and it follows that $c$ is column--extreme by Proposition \ref{rowuniCE}. Further recall that $b^\mathrm{t} \in \mathscr{H} ^\mathrm{t} (b)$ if and only if $b$ is non-CE \cite[Theorem 6.4]{JM-freeCE}. In particular, it immediately follows that
$$ a(Z) = a(0)\left( I - \sum _{\omega \neq \emptyset} Z^\omega \left\langle b ^\mathrm{t}, X^\omega b^\mathrm{t} \right\rangle_b \right), $$ so that the non-vacuum coefficients of $a$ are given by the Sarason formulas:
$$ \hat{a} _\omega = -a(0) \left\langle b^\mathrm{t}, X^\omega b^\mathrm{t} \right\rangle_b; \quad \quad \omega \neq \emptyset.$$ 

To simplify notation, we will assume for the remainder of this section that $b \in [\mathbb{H} ^\infty _d ] _1$ or $c = \left( \begin{smallmatrix} b \\ a \end{smallmatrix} \right) \in [\mathbb{H} ^\infty _d \otimes \mathbb{C} ^2] _1$, where $a$ is the Sarason outer function of $b$. The proofs for general $b \in [ \mathbb{H} ^\infty _d \otimes \mathscr{L} (\mathcal{H} , \mathcal{J} ) ] _1$ are virtually identical.
\begin{cor} \label{Sarmaxouter}
Let $b \in [\mathbb{H} ^\infty _d ] _1$ be non-CE and let $a \in [\mathbb{H} ^\infty _d ]_1$ be its unique Sarason function.
Then $a$ is the maximal outer factor of $T = I - b(R)^* b(R)$. In particular,
$$ a(0) ^2 = \inf_{\substack{p \in \mathbb{C} \{ \mathbb{\mathfrak{z}} \} ; \\ p(0) =0}} \left\langle 1-p, (I-b(R)^* b(R)) (1-p) \right\rangle_{\mathbb{H} ^2}.$$ If $c = \left( \begin{smallmatrix} b \\ a \end{smallmatrix} \right)$, then $c$ is CE and
$\| \boldsymbol{c} \| ^2 _c =  \| \boldsymbol{b} \| ^2 _b, $ where $\boldsymbol{c} = L^* \otimes I_2 \, c ^\mathrm{t}$, $\boldsymbol{b} = L^* b ^\mathrm{t}$.
\end{cor}
\begin{proof}
Consider $T = I - b(R) ^* b(R)$ and let $T_0 = A(R) ^* A(R) \leq T$ be the maximal factorable minorant with $A \in [\mathbb{H} ^\infty _d ] _1$ outer.
By \cite[Theorem 1.3]{Pop-entropy},
$$ |A(0) | ^2 = \inf_{p \in \mathbb{C} \{ \mathbb{\mathfrak{z}} \}  _0} \left\langle 1-p, (I-b(R)^* b(R)) (1-p) \right\rangle_{\mathbb{H} ^2},$$ and $A$ is maximal in the sense that if $D \in \mathbb{H} ^\infty _d \otimes \mathcal{H}$ is any column--outer so that $D(R) ^* D(R) \leq T$, then $D(R) ^* D(R) \leq A(R) ^* A(R)$. We can and do assume that $A(0) >0$. Note that $D(R)^* D(R) \leq T$ would imply that the column $C := \left( \begin{smallmatrix} b \\ D \end{smallmatrix} \right)$ is contractive. By applying Douglas factorization, there is a contraction, $F$, so that 
$ D(R) ^* =A(R) ^* F^*$, or equivalently, $F A(R) = D(R)$. Since $A$ is outer, $F$ commutes with the left shifts so that $F = F(R) \in [\mathbb{H} ^\infty _d ]_1$ \cite[Theorem 1.2]{DP-inv}. In fact, $F$ is also outer since $D$ is. It follows that the column, $\left( \begin{smallmatrix} b \\ A \end{smallmatrix} \right)$ is CE. If not, then there is some non-zero $D$ so that 
$$ \begin{pNiceMatrix} b \\ A \\ D \end{pNiceMatrix}, $$ is contractive and we can assume, without loss in generality that $A, D$ are outer. The previous argument then implies that $A(R)^* A(R) + D(R) ^* D(R) \leq A(R) ^* A(R)$ since $A$ is maximal and hence $D \equiv 0$.

To prove that $A=a$, where $A$ is the Popescu maximal outer factor of $I - b^* b$ and $a$ is the outer Sarason function of $b$, consider the column: $C := \left( \begin{smallmatrix} b \\ A \end{smallmatrix} \right)$. This is contractive and CE, so that 
\begin{eqnarray} \| \boldsymbol{C} \| ^2 _C & = & \boldsymbol{C} ^* \boldsymbol{C}  =   1 - |b(0) | ^2 - A(0)^2 \nonumber \\
& = & \| \boldsymbol{b} \| ^2 _b + a(0) ^2 - A(0) ^2, \nonumber \end{eqnarray} where recall that $a(0) >0$ by definition and we can assume that $A(0) > 0$ as well. 
It follows that 
$$ \| \boldsymbol{C} \| ^2 _C  - \| \boldsymbol{b} \| ^2 _b  = a(0)  ^2 - A(0)  ^2. $$ 
The CPNC kernel for $\mathscr{H} ^\mathrm{t} (C)$ is:
$$ K^C (Z,W) := K(Z,W) [ \cdot ] - K(Z,W) \left[ \begin{pNiceMatrix} b ^\mathrm{t} (Z) (\cdot) b ^\mathrm{t} (W) ^* & * \\ * & A ^\mathrm{t} (Z) (\cdot ) A ^\mathrm{t} (W) ^* \end{pNiceMatrix}\right] \geq 0. $$ 
It follows that the norm of $\| \boldsymbol{C} \| ^2 _C $ is the infimum over all $t>0$ so that
$$ \boldsymbol{C} (Z) (\cdot ) \boldsymbol{C} (W) ^* = \left( \begin{smallmatrix} \boldsymbol{b} (Z) \\ \boldsymbol{A} (Z) \end{smallmatrix} \right) [ \cdot ] \left( \begin{smallmatrix} \boldsymbol{b} (W) ^*, & \boldsymbol{A} (W) ^* \end{smallmatrix} \right) \leq t^2 K^C (Z,W) \otimes I_d \, [ \cdot ], $$ as CPNC kernels \cite[Lemma 6.6]{JM-freeCE}. Comparing $(1,1)$ entries, we see that
$$ \| \boldsymbol{b} \| ^2 _b \leq \| \boldsymbol{C} \| ^2 _C. $$ 
In particular, $$ 0 \leq  \| \boldsymbol{C} \| ^2 _C  - \| \boldsymbol{b} \| ^2 _b = a(0) ^2 - A (0) ^2, $$ and $a(0) \geq A(0)$.

If $a$ was not outer, then $a = \Theta \cdot F$ for some non-trivial inner $\Theta$ and outer $F$ and then $a(0) ^2 < |F(0)| ^2$ and 
$F(R) ^* F(R) = a(R) ^* a(R) \leq I- b(R) ^* b(R) =: T$. However, $A$ is the maximal outer factor of $T$, so that 
$F(R)^* F(R) \leq A(R) ^* A(R)$ and $F(R) = h(R) A(R)$ for some contractive and non-trivial outer $h$ \cite[Theorem 3.1]{MS-dBB}. In particular $a(0) ^2 < |F(0)| ^2 < A(0) ^2 $, contradicting that $a(0) ^2 \geq A(0)^2$. This proves that $A=a$ and it follows that if $c = \left( \begin{smallmatrix} b \\ a \end{smallmatrix} \right)$, then $c=C$ is column--extreme and 
$$ \| \boldsymbol{c} \| ^2 _c = \left\|  \left( \begin{smallmatrix} L^* b ^\mathrm{t} \\  L^* a ^\mathrm{t} \end{smallmatrix} \right)  \right\| ^2 _c  = \| \boldsymbol{b} \| ^2 _b.$$ 
\end{proof}
\begin{cor} \label{SarSar}
Let $b = \theta \cdot f$ be the inner--outer factorization of a non-CE $b \in [\mathbb{H} ^\infty _d ] _1$ where $f(0) > 0$. If $a$ is the Sarason outer function of $b$, then $f$ is the Sarason outer function of $a$. 
\end{cor}

\begin{cor} \label{CEcolumn}
Suppose that $b \in [\mathbb{H} ^\infty _d ] _1$ is non-CE and let $a \in [\mathbb{H} ^\infty _d ] _1$ be its Sarason function. We have that $c = \left( \begin{smallmatrix} b \\ a \end{smallmatrix} \right)$ is inner if and only if $\| X ^{(n)} \boldsymbol{b} \| ^2 _b = \| X ^{(n)} \boldsymbol{c}  \| _c ^2 \rightarrow 0$. In particular, $c$ is inner if and only if $X _c ^* = \left( L^* \otimes I_2 | _{\mathscr{H} ^\mathrm{t} (c)} \right) ^*$ is pure, or equivalently if and only if $X_b ^*$ is pure.
\end{cor}

\begin{lemma} \label{rank2defect}
If $b \in [\mathbb{H} ^\infty _d ] _1$ is not CE, then $X$ obeys the rank--two defect condition:
$$ X^* X = I - K_0 ^b (K_0 ^b) ^* - a(0) ^2 \left\langle b^\mathrm{t}, \cdot \right\rangle_b b^\mathrm{t}, $$ where $a$ is the free Sarason outer function of $b$. 
\end{lemma}
This is an analogue of \cite[Equation (7)]{Sarason-afun}.
\begin{proof}
This is readily established using the formula of \cite[Proposition 5.6]{JM-freeCE}, the fact that $b  ^\mathrm{t} \in \mathscr{H} ^\mathrm{t} (b)$ by \cite[Theorem 6.4]{JM-freeCE} since $b$ is non-CE and that $a(0) ^2 = 1 - |b (0) | ^2 - \| \boldsymbol{b} \| ^2 _{b}$ by definition of the Sarason function. Namely,
\begin{eqnarray} \| X K^b \{ Z, y , v \} \| ^2 _{\mathscr{H} ^\mathrm{t} (b)} & = & \| K^b \{ Z, Z^* y , v \} - \boldsymbol{b} v^* b^\mathrm{t} (Z) ^* y \| ^2 \nonumber \\
& = & y^* K ^b (Z,Z) [vv^*] y - 2 \mathrm{Re} \, \left\langle K^b \{Z,Z^*y , v \}, \boldsymbol{b} v^* b^\mathrm{t} (Z) ^* y \right\rangle + \boldsymbol{b} ^* \boldsymbol{b} \left| v^* b^\mathrm{t} (Z) ^* y \right| ^2 \nonumber \\
& = & \| K^b \{ Z , y ,v \} \| ^2 - y^* \left( vv^* - b ^\mathrm{t} (Z) vv^* b^\mathrm{t} (Z ) ^* \right) y -2 \left| v^* b^\mathrm{t} (Z) ^* y \right| ^2 \nonumber \\
& & + 2 \mathrm{Re} \, b(0) y^*v v^* b^\mathrm{t} (Z) ^* y + \left( 1 - |b(0)| ^2 - a(0) ^2 \right) \left| v^* b^\mathrm{t} (Z) ^* y \right| ^2 \nonumber \\
& = & \| K^b \{Z , y ,v \} \| ^2 - yvv^* y + 2 \mathrm{Re} \, b(0) y^* v v^* b ^\mathrm{t} (Z ) ^* y - \left( |b(0) | ^2 + a(0) ^2 \right) \left| v^* b^\mathrm{t} (Z) ^* y \right| ^2. \nonumber \end{eqnarray} 
On the other hand, applying the formula from the lemma statement,
\begin{eqnarray} & &  \left\langle K^b \{ Z , y ,v \}, \left( I - K_0 ^b (K_0 ^b ) ^* - a(0) ^2 \left\langle b^\mathrm{t}, \cdot \right\rangle_{\mathscr{H} ^\mathrm{t} (b)} b ^\mathrm{t} \right) K^b \{ Z , y ,v \} \right\rangle \nonumber \\
& = & \| K ^b \{ Z , y ,v \} \| ^2 - \left| \left\langle K^b \{ Z , y ,v \}, K_0 ^b \right\rangle \right| ^2 - a(0) ^2 \left| y^* b^\mathrm{t} (Z) v \right| ^2, \nonumber \end{eqnarray} 
yields the same expression. The claim then follows by polarization.
\end{proof}
\begin{proof}{ (of Corollary \ref{CEcolumn})}
If $c$ is inner then $\mathscr{H} ^\mathrm{t} (c) = \mathrm{Ran} \, c(R) ^\perp \subseteq \mathbb{H} ^2 _d \otimes \mathbb{C} ^2$, so that $\mathscr{H} ^\mathrm{t} (c)$ is a closed, co-invariant subspace for $L \otimes I_2$. It is then easily checked that $X^*$ is pure since $L$ is. Conversely if $\| X^{(n)} \boldsymbol{b} \| ^2 _b \rightarrow 0$, then $c$ is inner by Theorem \ref{pureisinner}.

To see that $\| X ^{(n)} \boldsymbol{b} \| ^2 _b = \| X^{(n)} \boldsymbol{c} \| ^2 _c$ for all $n \geq 0$, we apply the rank$-2$ defect condition of Lemma \ref{rank2defect} repeatedly:
For all $n \in \mathbb{N}$,
\begin{eqnarray} \| X^{(n)} \boldsymbol{b} \| ^2 _b & = & \left\langle b^\mathrm{t}, \mathrm{Ad} _{X^* ,X} ^{(n+1)} (I) b^\mathrm{t} \right\rangle_b \nonumber \\
& = & \left\langle b^\mathrm{t}, \mathrm{Ad} _{X^* ,X} ^{(n)} (I) b^\mathrm{t} \right\rangle_b  - \left| \left\langle K_0 ^b, X ^{(n)} b^\mathrm{t} \right\rangle_b \right| ^2 - a(0) ^2  \left| \left\langle b^\mathrm{t}, X^{(n)} b^\mathrm{t} \right\rangle_b \right| ^2 \nonumber \\
& = & \left\langle b^\mathrm{t}, \mathrm{Ad} _{X^* ,X} ^{(n)} (I) b^\mathrm{t} \right\rangle_b - \sum _{|\omega |  = n } |\hat{b} _\omega | ^2 - \sum _{|\omega | = n} |\hat{a} _\omega | ^2 \nonumber \\
... & = & \| X b^\mathrm{t} \| ^2 _b  - \sum _{ 0 < |\omega |  \leq n } |\hat{b} _\omega | ^2 - \sum _{0 <|\omega | \leq n} |\hat{a} _\omega | ^2 \nonumber \\
& = & \| \boldsymbol{c} \| ^2 _c - \sum _{ 0 < |\omega |  \leq n } \hat{c} _\omega ^* \hat{c} _\omega; \quad \quad \hat{c} _\omega \in \mathbb{C} ^2  \nonumber \\
& = & 1 - |c(0) | ^2 - \sum _{ 0 < |\omega |  \leq n } \hat{c} _\omega ^* \hat{c} _\omega  \nonumber \\ 
& = & 1 -  \sum _{ |\omega |  \leq n } \hat{c} _\omega ^* \hat{c} _\omega  = \| X^{(n)} \boldsymbol{c} \| ^2 _c. \nonumber \end{eqnarray}
In the above we used that $\| \boldsymbol{c} \| _c = \| \boldsymbol{b} \| _b$ as proven in Corollary \ref{Sarmaxouter}.
\end{proof}

\section{Free de Branges--Rovnyak space and Smirnov graph analysis}

If $c = \left( \begin{smallmatrix} b  \\ a \end{smallmatrix} \right) \in [ \mathbb{H} ^\infty _d \otimes \mathbb{C} ^2 ]_1$ is inner and $a$ is outer, we will say that $(a,b)$ is a \emph{Smirnov column--inner pair}. This terminology is motivated by the results of \cite{JM-freeSmirnov}. Namely, the \emph{left free Smirnov class}, $\mathscr{N} _d ^+$, is the set of all free NC functions, $h \in \mathscr{O} (\mathbb{B} ^d _\mathbb{N} )$, defined as `fractions' of bounded left multipliers with outer denominators. That is, $h \in \mathscr{N} _d ^+$ if there are $a,b \in  \mathbb{H} ^\infty _d  $ with $a$ outer so that $h (Z) = b(Z) a(Z) ^{-1}$. In \cite{JM-freeSmirnov} we showed that a closed and densely--defined linear operator $A \in \mathscr{L} (\mathbb{H} ^2 _d )$ is \emph{affiliated to $\mathscr{L} ^\infty _d$} in the sense that $\mathrm{Dom} \, A$ is $R-$invariant and 
$ R_k A x = A R_k x$ for all $x \in \mathrm{Dom} \, A$ if and only if $A = M^L _h = h(L)$ is an unbounded, closed and densely--defined left multiplier with symbol $h \in \mathscr{N} _d ^+$. Here, recall that the commutant of $\mathscr{L} ^\infty _d$ is $\mathscr{R} ^\infty _d$, \cite[Theorem 1.2]{DP-inv}. Moreover, the graph, $\mathrm{Gr} \, h(L) \subseteq \mathbb{H} ^2 _d \otimes \mathbb{C} ^2$, is then $R \otimes I_2-$invariant and 
$$ \mathrm{Gr} \, h(L) = \mathrm{Ran} \, c(L), \quad c(L) : \mathbb{H} ^2 _d \rightarrow \mathbb{H} ^2 _d \otimes \mathbb{C} ^2 \quad c(L) =: \begin{pNiceMatrix} a(L) \\ b(L) \end{pNiceMatrix}, $$ where $c(L)$ is inner so that $\mathrm{Ran} \, a(L) = \mathrm{Dom} \, h(L)$, $a,b \in [\mathbb{H} ^\infty _d ] _1$ and $a(L)$ is outer since $h(L)$ is densely--defined \cite[Corollary 4.26, Corollary 4.27]{JM-freeSmirnov}. Hence $(a,b)$ is a Smirnov column--inner pair and $h(L) = b(L) a(L) ^{-1}$. Smirnov column--inner pairs which represent a given $h \in \mathscr{N} _d ^+$ are unique \cite[Corollary 5.2]{JM-freeSmirnov}. Equivalently, $h \in \mathscr{N} ^+ _d$ if and only if $h(R) := M^R _{h ^\mathrm{t}}$, $h(R) = b(R) a(R) ^{-1}$, is a densely--defined and closed operator affiliated to $\mathscr{R} ^\infty _d$ and we write $h(R) \sim \mathscr{R} ^\infty _d$ in this case. 

Assume that $b \in [ \mathbb{H} ^\infty _d ] _1$ is non-CE and let $a$ be its unique Sarason outer function so that the column $c := \left( \begin{smallmatrix} a \\ b \end{smallmatrix} \right)$ is CE. Consider the NC inner--outer factorization of $c$, $c = C \cdot D$ where $C$ is inner and $D$ is outer. Here, 
$$ \mathrm{Ran} \, C (R) = \mathrm{Ran} \, c(R) ^{- \| \cdot \| _{\mathbb{H} ^2}}, $$ so that $\mathrm{Ran} \, C(R)$ is a closed, $L\otimes I_2-$invariant subspace of $\mathbb{H} ^2 _d \otimes \mathbb{C} ^2$ and it has the $L\otimes I_2-$cyclic vector $x := \left( \begin{smallmatrix} b^\mathrm{t} \\ a ^\mathrm{t} \end{smallmatrix} \right) = c ^\mathrm{t}$. Since $\mathrm{Ran} \, C(R)$ has a cyclic vector, it follows that the wandering space of this $L\otimes I_2-$invariant subspace is also one--dimensional so that 
$$ C (R) = \left( \begin{smallmatrix} A(R) \\ B(R) \end{smallmatrix} \right), $$ for some Smirnov column--inner pair $A, B \in [ \mathbb{H} ^\infty _d ] _1$. Hence $D(R) \in [ \mathbb{H} ^\infty _d ] _1$ is outer and since $a(R) = A(R) D(R)$ is outer, $A(R)$ must also be outer. Moreover, since $C(R)$ is inner, 
$$ I - B(R) ^* B(R) = A(R) ^* A(R), $$ so that if we set $A(0) >0$, $A$ is the Sarason outer function of $B$ by Corollary \ref{Sarmaxouter}. The outer multiplier $D$ must be CE:
\begin{eqnarray} I -b(R) ^* b(R) &= & I - D(R) ^* B(R) ^* B(R) D(R) \nonumber \\
& = & I - D(R) ^* D(R) + D(R) ^* A(R) ^* A(R) D(R) \nonumber \\
& = & I - D(R) ^* D(R) + a(R) ^* a(R). \nonumber \end{eqnarray} Since $a(R)$ is the maximal outer factor of $I - b(R) ^* b(R)$, $I- D(R) ^* D(R)$ cannot have any non-trivial factorable minorant, \emph{i.e.} $D$ is column--extreme. If we set $H(R) := B(R) A(R) ^{-1}$, then this is a closed and densely--defined right Smirnov multiplier of the Fock space, with $\mathrm{Gr} \, H(R)  = \mathrm{Ran} \, C(R)$ \cite{JM-freeSmirnov}. Since $C$ is inner, basic facts about operator--range spaces and complementary spaces imply that 
$$ \mathscr{H} ^\mathrm{t} (c) = \mathscr{H} ^\mathrm{t} (C) \oplus C(R) \mathscr{H} ^\mathrm{t} (D), $$ and that $C(R)$ is an isometric right multiplier of $\mathscr{H} ^\mathrm{t} (D)$ into $\mathscr{H} ^\mathrm{t} (c)$ \cite[Theorem 18.8]{FM2}.
\begin{thm} \label{orthodecomp}
Given $b = BD$ as above, 
$$ \mathscr{H} ^\mathrm{t} (b) = \mathscr{H} ^\mathrm{t} (B) \oplus B(R) \mathscr{H} ^\mathrm{t} (D), $$ and $B(R)$ is an isometric right multiplier of $\mathscr{H} ^\mathrm{t} (D)$ into $\mathscr{H} ^\mathrm{t} (b)$. If we define $E _1 : \mathscr{H} ^\mathrm{t} (c) \rightarrow \mathscr{H} ^\mathrm{t} (b)$ by projection onto the first co-ordinate, $E_1 \left( \begin{smallmatrix} f \\ g \end{smallmatrix} \right) = f$, then $E_1$ is an onto isometry. 
\end{thm}
\begin{prop} \label{afunconverge}
Let $b_n \in [\mathbb{H} ^\infty _d ]_1$ be a sequence (or net) of contractive non-CE multipliers which converge $SOT-*$ to a contractive, non-CE $b \in [\mathbb{H} ^\infty _d ]_1$. Then the corresponding Sarason outer multipliers $a_n (R)$ converge to $a(R)$ in the weak operator topology. If $c := \left( \begin{smallmatrix} b \\ a \end{smallmatrix} \right)$ is inner then $a_n (R)$ converges to $a(R)$ SOT$-*$.
\end{prop}
This lemma applies, for example, to the free polynomial sequence $b_n$ of $n$th Ces\`aro sums of $b$, or to the nets $b(rR)$ or $r b(R)$ for $0 < r <1$, taking the limit as $r \uparrow 1$. 
\begin{proof}
Recall that 
$$ a _n (0) ^2 = 1 - |b_n (0) | ^2 - \boldsymbol{b} _n ^* \boldsymbol{b_n}, $$ $a_n (0) > 0$ and $\boldsymbol{b} _n = L^* b_n ^\mathrm{t}$. Since $b$ is non-CE, its Sarason function, $a$, also obeys $a(0) >0$ and $a(0) ^2 = 1 - |b(0) | ^2 - \boldsymbol{b} ^* \boldsymbol{b} >0$. Since $b_n (R) \stackrel{SOT-*}{\rightarrow} b(R)$, it follows that $a_n (0) \rightarrow a(0) >0$.

Since every $a_n$ is a contractive multiplier, we can assume that (by possibly passing to a subsequence) $a_n (R) \stackrel{WOT}{\rightarrow} \widetilde{a} (R)$, where $\widetilde{a} \in [ \mathbb{H} ^\infty _d ] _1$ and so that $$ \widetilde{a} (0) = a(0) >0.$$ Hence, 
$$ c_n (R) = \begin{pNiceMatrix} b_n (R) \\ a_n (R) \end{pNiceMatrix} \stackrel{WOT}{\rightarrow} \begin{pNiceMatrix} b(R) \\ \widetilde{a} (R) \end{pNiceMatrix}, $$ so that $\left( \begin{smallmatrix} b \\ \widetilde{a} \end{smallmatrix} \right) \in [ \mathbb{H} ^\infty _d \otimes \mathbb{C} ^2 ] _1$. Since $\widetilde{a} (0) = a(0) >0$, $\widetilde{a} \neq 0$ and since $\widetilde{a}$ is not CE, it has an inner--outer factorization $\widetilde{a} = \Theta \hat{a}$ where $\Theta (R)$ is inner and $\hat{a} (R)$ is outer. We can and do further assume that $\hat{a} (0) >0$. Now observe that 
$$ \begin{pNiceMatrix} b (R) \\ \hat{a} (R) \end{pNiceMatrix} = \begin{pNiceMatrix} I & 0 \\ 0 & \Theta (R) ^* \end{pNiceMatrix} \begin{pNiceMatrix} b(R) \\ \widetilde{a} (R) \end{pNiceMatrix} \in [ \mathbb{H} ^\infty _d \otimes \mathbb{C} ^2 ] _1, $$ so that it follows that $\left( \begin{smallmatrix} b \\ \hat{a} \end{smallmatrix} \right)$ is a contractive multiplier, where now $\hat{a}$ is outer. However, the Sarason function, $a(R)$, is the maximal outer factor of $I-b(R) ^* b(R)$ by Corollary \ref{Sarmaxouter}. This means that since $\hat{a} (R) ^* \hat{a} (R) \leq I - b(R) ^* b(R)$, it must be that 
$$ \hat{a} (R) ^* \hat{a} (R) \leq a(R) ^* a(R). $$ By Douglas factorization, there is a contraction, $C ^*$ so that 
$$ \hat{a} (R) ^* = a(R) ^* C^* \quad \Rightarrow \quad \hat{a} (R) = C a(R). $$ Since both $\hat{a} , a$ are outer, it follows that $C$ commutes with the left shifts, so that $C = C(R) \in [ \mathscr{R} ^\infty _d ] _1$ is a contractive right multiplier \cite[Theorem 1.2]{DP-inv}. This would imply that 
$$ \hat{a} (0) = C(0) a(0) = C(0) \widetilde{a} (0) = C(0) \Theta (0) \hat{a} (0), $$ so that 
$$ C(0) \Theta (0) =1. $$ However, $C(R) \Theta (R)$ is a contractive right multiplier, so that the NC maximum modulus principle implies that $C(R) \Theta (R) \equiv I$ \cite[Lemma 6.11]{SSS}. Since $\Theta (R)$ is inner, it is not invertible unless $\Theta (R) \equiv I$. It follows that $\Theta (R) = I = C(R)$ and we conclude that $\hat{a} (R) = \widetilde{a} (R) = a(R)$. We can conclude that the entire sequence $a_n (R) \stackrel{WOT}{\rightarrow} a(R)$, since any $WOT$-convergent subsequence must 
converge to $a(R)$ by the above argument.

Now suppose that $c = \left( \begin{smallmatrix} b \\ a \end{smallmatrix} \right)$ is inner. Let $h_n (R) := b_n (R) a_n (R) ^{-1}$ and $\tau _n := h_n (R) ^* h_n (R)$. Then for any $z \in \mathbb{C} \setminus [ 0 , +\infty )$,
\begin{eqnarray} ( \tau _n - z I ) ^{-1} & = & \left( a_n (R) ^{-*} ( b_n (R) ^* b_n (R) - z  a_n (R) ^* a_n (R) ) a_n (R) ^{-1} \right) ^{-1} \nonumber \\
& = & a_n (R) \left( (1+z) b_n (R) ^* b_n (R) -zI \right) ^{-1} a_n (R) ^*. \nonumber \end{eqnarray} Since $a_n (R) \stackrel{WOT}{\rightarrow} a(R)$ we have that $a_n (R) ^* \stackrel{SOT}{\rightarrow} a(R)^*$ and $b_n (R) \stackrel{SOT-*}{\rightarrow} b(R)$ by \cite[Lemma 9]{JM-ncFatou}, so that the above converges in the weak operator topology to 
\begin{eqnarray} a(R) \left( (1+z) b(R) ^* b(R) -zI \right) ^{-1} a(R) ^* & = & a(R) \left( (1+z) b(R) ^* b(R) - z b(R) ^* b(R) - z a(R) ^* a(R) \right) ^{-1} a(R) ^* \nonumber \\
& = & (\tau - zI ) ^{-1} \nonumber \end{eqnarray} where $\tau = h(R) ^* h(R)$ and $h(R) = b(R) a(R) ^{-1}$. In the above we used the assumption that $a(R) ^* a(R) + b(R) ^* b(R) =I$. Hence $\tau _n$ converges to $\tau$ in the weak resolvent sense \cite{RnS1}. A simple argument using the resolvent formula then shows that $(\tau _n - zI ) ^{-1}$ converges in the strong operator topology to $(\tau -zI) ^{-1}$ for any $z \in \mathbb{C} \setminus [ 0 , + \infty )$. In particular, taking $z=-1$ gives
$$ (\tau + I ) ^{-1} = a(R) \left(  (1+z) b(R) ^* b(R) -zI \right) ^{-1} a(R) ^* = a(R) a(R) ^*, $$ so that $a_n (R) a_n (R) ^* \stackrel{SOT}{\rightarrow} a(R) a(R) ^*$ and it follows easily from this that $a_n (R) \stackrel{SOT}{\rightarrow} a(R)$. Recall that for any net of right (or left) multipliers, $WOT$ convergence implies $SOT$ convergence of $a_n (R) ^*$, hence $a_n (R) \stackrel{SOT-*}{\rightarrow} a(R)$ by \cite[Lemma 9]{JM-ncFatou}.
\end{proof}

\iffalse

\begin{lemma} Let $b$ be a contractive multiplier. For each $0\leq r < 1$, let $a_r$ be the (unique, invertible) outer multiplier defined by
  \[
    I-r^2b(R)^*b(R) =a_r (R) ^* a_r (R).
  \]
  Then for all $0\leq r<1$, 
  \[
    (I-r^2b(R)b(R) ^*)^{-1} = I+ r^2 b(R) a_r (R) ^{-1} a_r (R) ^{-*}b(R) ^*.
  \]
\end{lemma}
\begin{proof} We have
  \begin{align*}
    I &= (I-r^2 b(R) b(r) ^*)^{-1}(I-r^2b(R) b(R) ^*) \\
      &= (I-r^2 b(R) b(R) ^*)^{-1} -r^2 (I-r^2 b(R) b(R)^*)^{-1}b(R) b(R) ^* \\
      &= (I-r^2 b(R) b(R) ^*)^{-1} -r^2 b(R) (I-r^2 b(R)^* b(R) )^{-1}b(R)^* \\
      &= (I-r^2 b(R) b(R) ^*)^{-1} -r^2b(R) a_r (R) ^{-1}a_r (R) ^{-*}b(R) ^*.
  \end{align*}
\end{proof}

If we set $H_r (R):= b(R) a_r (R) ^{-1}$, the above formula becomes
\[
  (I-r^2b(R)b(R) ^*)^{-1} = I +H_r (R) H_r (R) ^*.
\]

\fi

Consider the (generally) unbounded multiplier $H(R) := b(R) a(R) ^{-1}$ with dense domain $\mathrm{Dom} \, H(R) := \mathrm{Ran} \, a(R)$. By \cite[Corollary 3.9]{JM-freeSmirnov}, $H(R)$ is closeable and both the linear span of the NC Szeg\"o kernel vectors $K\{Z , y ,v \}$ and the free polynomials are cores for $H(R) ^*$ \cite[Corollary 3.13]{JM-freeSmirnov}. Here, recall that a dense linear subspace, $\mathscr{D} \subseteq \mathrm{Dom} \, T$ in the domain of a closed, densely--defined linear operator $T$ is called a \emph{core} for $T$ if $\mathrm{Gr} \, T = \{ x \oplus Tx \ | \, x \in \mathscr{D} \} ^{-\| \cdot \| }$. That is, $T$ is the closure of its restriction to $\mathscr{D}$.

Given any $b \in [\mathbb{H} ^\infty _d ] _1$, for each $0\leq r < 1$, let $a_r$ be the (unique, invertible) outer multiplier defined by
  \[
    I-r^2b(R)^*b(R) =a_r (R) ^* a_r (R).
  \]

\begin{lemma} \label{dBRnorm}
Let $b \in [\mathbb{H} ^\infty _d ] _1$ be a contractive multiplier.  Then $f\in \mathscr{H} ^\mathrm{t} (b) $ if and only if
  \[
    \sup_{0<r<1} \|a_r^{-*} (R) b(R) ^* f \|_{\mathbb{H} ^2} <\infty,
  \]
  in which case
\begin{eqnarray}  \|f\|_{\mathscr{H} ^\mathrm{t} (b) } ^2 & = & \|f \|_{\mathbb{H} ^2}^2 +\lim_{r\uparrow 1} \|a_r^{-*} (R) b(R) ^* f \|_{\mathbb{H} ^2} \nonumber \\
& = & \| f \| ^2 _{\mathbb{H} ^2} + \| H(R) ^* f \| ^2 _{\mathbb{H} ^2} = \| f \| ^2 _{\mathrm{Gr} (H(R) ^*)}. \nonumber \end{eqnarray} 
\end{lemma}
\begin{proof}
This follows from basic facts on operator--range spaces and complementary spaces, see \cite[Theorem 16.17, Theorem 16.18]{FM2}.
\end{proof}

The following can be pieced together from results in \cite{FM2}, in particular see \cite[Theorem 18.8]{FM2}.

\begin{thm*} \label{dBRdecomp}
Let $a=bc$ where $a,b,c \in [ \mathbb{H} ^\infty _d ]_1$. Then, $\mathscr{H} ^\mathrm{t} (b^*) \cap \mathscr{H} ^\mathrm{t} (c) = \{ 0 \}$ if and only if 
$\mathscr{H} ^\mathrm{t} (b) \cap b (R)  \mathscr{H} ^\mathrm{t} (c) = \{ 0 \}$ and in this case:
\begin{enumerate}
    \item $\mathscr{H} ^\mathrm{t} (a) = \mathscr{H} ^\mathrm{t}  (b) \oplus b (R) \mathscr{H} ^\mathrm{t} (c)$.
    \item $\mathscr{H}^\mathrm{t} (b)$ embeds isometrically into $\mathscr{H} ^\mathrm{t} (a)$
    \item $M^R _{b^\mathrm{t}}$ is an isometric multiplier of $\mathscr{H} ^\mathrm{t} (c)$ into $\mathscr{H} ^\mathrm{t} (a)$.
\end{enumerate}
\end{thm*}
\begin{proof}
We prove the first equivalent statement, the rest of the proof can be found in \cite[Theorem 18.8]{FM2}. If $x \in \mathscr{H} ^\mathrm{t}  (b^*) \cap \mathscr{H} ^\mathrm{t} (c)$ then $b(R) x \in \mathscr{H} ^\mathrm{t} (b) \cap b(R) \mathscr{H} ^\mathrm{t} (c)$. Conversely, if $b(R) x \in \mathscr{H} ^\mathrm{t} (b) \cap b (R) \mathscr{H} ^\mathrm{t} (c)$ for some $x \in \mathscr{H} ^\mathrm{t} (c)$ then $b(R) x \in \mathscr{H} ^\mathrm{t} (b) \cap \mathscr{M} ^\mathrm{t} (b) = b (R) \mathscr{H} ^\mathrm{t} (b^*)$ by \cite[Lemma 16.20]{FM2}. Hence $b(R) x = b (R) y$ for some $y \in \mathscr{H} ^\mathrm{t} (b^*)$ and $x = y \in \mathscr{H} ^\mathrm{t} (c) \cap \mathscr{H} ^\mathrm{t} (b^*)$.  
\end{proof}

\begin{proof}{ (of Theorem \ref{orthodecomp})}
Consider any NC Szeg\"o kernel vector $K\{ Z, y ,v \}$. Then,
$$ \sup_{0<r<1} \|a_r^{-*} (R) b(R) ^*  K \{ Z , y ,v \}  \|_{\mathbb{H} ^2}  =  \sup \| K \{ Z , y , a_r ^\mathrm{t} (Z) ^{-1} b ^\mathrm{t} (Z) v \} \| < + \infty, $$
since $a_r (R) \stackrel{WOT}{\rightarrow} a(R)$ by Proposition \ref{afunconverge} and WOT convergence implies pointwise convergence. We conclude, by Lemma \ref{dBRnorm}, that the linear span of the NC Szeg\"o kernels is contained in $\mathscr{H} ^\mathrm{t} (b)$. (Since $b$ is non-CE, this also follows from the results of \cite{JM-freeCE}.) Again by Lemma \ref{dBRnorm},
\begin{eqnarray} \| K \{ Z , y ,v \} \| ^2 _b & = & \| K \{ Z ,y , v \} \| ^2 _{\mathbb{H} ^2} + \lim_{r\uparrow 1} \| K \{Z ,y , a _r ^\mathrm{t} (Z) ^{-1} b^\mathrm{t} (Z) v\}  \|_{\mathbb{H} ^2} ^2 \nonumber \\
& = & \| K \{ Z ,y , v \} \| ^2 _{\mathbb{H} ^2} + \| K \{Z ,y , H ^\mathrm{t} (Z) v \}  \|_{\mathbb{H} ^2} ^2 \nonumber \\
& = & \| K \{ Z , y ,v \} \| ^2 _B. \nonumber \end{eqnarray} Recall that $C = \left( \begin{smallmatrix} A \\ B \end{smallmatrix} \right)$ is inner, $\mathrm{Ran} \, C(R) = \mathrm{Gr} \, H (R) $ and the NC Szeg\"o kernels are a core for $H(R) ^*$. By \cite[Theorem 4.17]{JM-freeSmirnov}, $\mathscr{H} ^\mathrm{t} (B) = \mathrm{Dom} \, H(R) ^*$ if we equip $\mathrm{Dom} \, H(R) ^*$ with the graph norm. It follows that the NC Szeg\"o kernels (or free polynomials) are dense in $\mathscr{H} ^\mathrm{t} (B)$ and that $\mathscr{H} ^\mathrm{t} (B)$ embeds isometrically into $\mathscr{H} ^\mathrm{t} (b)$.

The fact that $\mathrm{Ran} \, b(R) \subseteq \mathrm{Ran} \, B(R)$ and $ b(R) b(R) ^* = B(R) D(R) D(R) ^* B(R) ^* \leq B(R)  B(R) ^*$ implies that $\mathscr{H} ^\mathrm{t} (B)$ is contractively contained in $\mathscr{H} ^\mathrm{t} (b)$. Moreover, by \cite[Theorem 16.23]{FM2}, 
$$ \mathscr{H} ^\mathrm{t} (b) = \mathscr{H} ^\mathrm{t} (B) + B(R) \mathscr{H} ^\mathrm{t} (D), $$ $M^R _{B^\mathrm{t}}$ defines a contraction from $\mathscr{H} ^\mathrm{t} (D)$ into $\mathscr{H} ^\mathrm{t} (b)$ and if $x \in \mathscr{H} ^\mathrm{t} (B)$, $y \in \mathscr{H} ^\mathrm{t} (D)$ then,
$$ \| x + B(R) y \| ^2 _b \leq \| x \| ^2 _B + \| y \| ^2 _D. $$ 
Using that $\mathscr{H} ^\mathrm{t} (B)$ is isometrically contained in $\mathscr{H} ^\mathrm{t} (b)$,
$$ \| x + B (R) y \| ^2 _b  =  \| x \| ^2 _B + 2 \mathrm{Re} \, \left\langle x, B (R) y \right\rangle_b + \| B (R) y \| ^2 _b. $$ Subtracting these two formulas gives:
$$ 0 \leq \| y \| ^2 _D - \| B (R) y \| ^2 _b - 2 \mathrm{Re} \, \left\langle x, B(R) y \right\rangle_b. $$ Here, 
$$ 0 \leq \| y \| ^2 _D - \| B (R) y \| ^2 _b, $$ since $B(R) y = M^R _{B ^\mathrm{t}} y$ and right multiplication by $B^\mathrm{t}$ defines a contraction from $\mathscr{H} ^\mathrm{t} (D)$ into $\mathscr{H} ^\mathrm{t} (b)$. Hence 
$$ 2 \mathrm{Re} \, \left\langle x, B (R) y \right\rangle_b \leq \| y \| ^2 _D - \| B (R) y \| ^2 _b. $$ Given any such $x,y$, choose $\lambda \in \partial \mathbb{D}$ so that 
$\left\langle \lambda x, B (R) y\right\rangle_b = | \left\langle x, B (R) y\right\rangle_b |$. Then for any $r >0$, 
$$ 2 r | \left\langle x, B(R) y\right\rangle_b | = 2 \mathrm{Re} \, \left\langle r\lambda x, B (R) y\right\rangle _b  \leq  \| y \| ^2 _D - \| B (R)  y \| ^2 _b, $$ and it follows that 
$\left\langle x, B(R) y\right\rangle_b =0$. That is,
$$\mathscr{H} ^\mathrm{t} (b) = \mathscr{H} ^\mathrm{t} (B) \oplus B (R) \mathscr{H} ^\mathrm{t} (D). $$ By reproducing kernel theory, this is equivalent to the condition that $\mathscr{H} ^\mathrm{t} (B) \cap B (R) \mathscr{H} ^\mathrm{t} (D) =0$, so that the remainder of the claim follows from the previous theorem. 
\end{proof}

\begin{cor} \label{AstarD}
With $a,b,A,B$ and $D$ as above,
$$ \mathscr{H} ^\mathrm{t} (D) \cap \mathscr{M} ^\mathrm{t} (A^* ) = \{ 0 \}. $$ 
\end{cor}
\begin{proof}
By RKHS theory, if $K, k$ are CPNC kernels on the same NC set, $\mathcal{H} _{nc} (K + k) = \mathcal{H} _{nc} (K ) \oplus \mathcal{H} _{nc} (k)$ if and only if 
$$ \mathcal{H} _{nc} (K) \cap \mathcal{H} _{nc} (k) = \{ 0 \},$$ \cite[Section 6]{Aron-rkhs},\cite[Theorem 4.4]{JM-ncld}. By Theorem \ref{orthodecomp}, 
$$ \mathscr{H} ^\mathrm{t} (b) = \mathscr{H} ^\mathrm{t} (B) \oplus B(R) \mathscr{H} ^\mathrm{t} (D). $$ Adding together the CPNC kernels, $K^B$, of $\mathscr{H} ^\mathrm{t} (B)$ and $k$ of $B(R) \mathscr{H} ^\mathrm{t} (D)$ gives:
\begin{eqnarray} K^B (Z,W) [\cdot] + k (Z,W) [ \cdot ] & = & K(Z,W) \left[ (\cdot) - B^\mathrm{t} (Z) (\cdot ) B^\mathrm{t} (W) ^*   \right] \nonumber \\
& & + K(Z,W) \left[ B^\mathrm{t} (Z) (\cdot)  B^\mathrm{t} (W) ^* - D^\mathrm{t} (Z) B^\mathrm{t} (Z) ( \cdot ) B^\mathrm{t} (W) ^* D^\mathrm{t} (W) ^* \right ] \nonumber \\
& = & K (Z,W) [ \cdot ] - K(Z,W) [ b^\mathrm{t} (Z) (\cdot ) b^\mathrm{t} (W) ^* ] = K^b (Z,W)[\cdot ]. \nonumber \end{eqnarray} 
Hence $\mathscr{H} ^\mathrm{t} (b) = \mathcal{H} _{nc} (K^B + k )$ and the orthogonal decomposition of Theorem \ref{orthodecomp} then implies that 
$$ \mathscr{H} ^\mathrm{t} (B) \cap B(R) \mathscr{H} ^\mathrm{t} (D) = \{ 0 \}. $$ By the theory of operator--range spaces \cite[Section 16.6]{FM2}, 
$$ \mathscr{H} ^\mathrm{t} (B) \cap \mathscr{M} ^\mathrm{t} (B) = B(R) \mathscr{H} ^\mathrm{t} ( B^* ). $$ Hence,
$$ \{ 0 \} = \mathscr{H} ^\mathrm{t} (B) \cap B(R) \mathscr{H} ^\mathrm{t} (D )  \subseteq  \mathscr{H} ^\mathrm{t} (B) \cap \mathscr{M} ^\mathrm{t} (B) = B(R) \mathscr{H} ^\mathrm{t} (B ^* ). $$ Since $B(R)$ is injective,
$$ \mathscr{H} ^\mathrm{t} (D) \cap \mathscr{H} ^\mathrm{t} (B^* ) = \{ 0 \}. $$ 
Finally, the column $\left( \begin{smallmatrix} A \\ B \end{smallmatrix} \right)$ is inner so that
$$ \mathscr{H} ^\mathrm{t} (B ^* ) = \mathrm{Ran} \, \sqrt{I - B(R) ^* B(R)} = \mathrm{Ran} \, \sqrt{ A(R) ^* A(R)} = \mathscr{M} ^\mathrm{t} (A^* ). $$ 
\end{proof}

If $b \in [\mathbb{H} ^\infty _d ] _1$ is non-CE, then by \cite[Corollary 6.15]{JM-freeCE}, the right free de Branges--Rovnyak space, $\mathscr{H} ^\mathrm{t} (b)$ is both $L-$co-invariant and $L-$invariant and, following \cite{Sarason-doubly}, we then say that $\mathscr{H} ^\mathrm{t} (b)$ is \emph{doubly free shift invariant}. In this case, let $Y:= L | _{\mathscr{H} ^\mathrm{t} (b)}$.

\begin{lemma} \label{adjX}
Let $b \in [ \mathbb{H} ^\infty _d ] _1$ be non-CE. The restriction, $Y$, of $L$ to $\mathscr{H} ^\mathrm{t} (b)$ obeys the formula:
$$ Y = X^* + \left\langle \boldsymbol{b}, \cdot \right\rangle_b b^\mathrm{t}.$$
\end{lemma}
\begin{proof}
Given any $\boldsymbol{h} \in \mathscr{H} ^\mathrm{t} (b) \otimes \mathbb{C} ^d$, we calculate $X^* \boldsymbol{h}$:
\begin{eqnarray} \left\langle K^b \{ Z,y,v \}, X^* \boldsymbol{h} \right\rangle_b & = & \left\langle X K^b \{ Z,y,v\}, \boldsymbol{h} \right\rangle_b \nonumber \\
& = & \left\langle L^* \left(K \{ Z,y,v\} - K \{ Z,y, b^\mathrm{t} (Z) v\} b^\mathrm{t} \right), \boldsymbol{h}\right\rangle_b \nonumber \\
& = & \left\langle \left( K\{ Z , Z^* y , v \} - K \{Z , Z^* y , b^\mathrm{t} (Z) v \} b^\mathrm{t} -  K \{Z ,  y , b^\mathrm{t} (Z) v \} (0) \boldsymbol{b} \right), \boldsymbol{h} \right\rangle_b \nonumber \\
& = & \left\langle K^b \{ Z , Z^* y , v \} - \overline{y^* b^\mathrm{t} (Z) v } \boldsymbol{b}, \boldsymbol{h}\right\rangle _{b} \nonumber \\
& = & y^* Z \boldsymbol{h} ^\mathrm{t} (Z) v - y^* b^\mathrm{t} (Z) v \left\langle \boldsymbol{b}, \boldsymbol{h} \right\rangle_b. \nonumber \end{eqnarray} 
This proves that 
$$ X^* = Y - \left\langle \boldsymbol{b}, \cdot \right\rangle_b b^\mathrm{t}, $$ and the formula for $Y$ follows.
\end{proof}

The next two lemmas provide an analogue of a formula of Sarason for the norm of $b^\mathrm{t}$ in $\mathscr{H} ^\mathrm{t} (b)$ when $b$ is non-CE \cite[Lemma 2]{Sarason-afun}.
\begin{lemma}
Let $A,B \in [\mathbb{H} ^\infty _d ]_1 $ be a Smirnov column--inner pair and $C = \left( \begin{smallmatrix} A \\ B \end{smallmatrix} \right)$ be inner. Then,
$$ \| B ^\mathrm{t} \| ^2 _{B} = \frac{1}{A(0) ^2} -1. $$
\end{lemma}
\begin{proof}
This follows from \cite[Theorem 4.17]{JM-freeSmirnov}: $\mathrm{Ran} \, C(R)$ is the graph, $\mathrm{Gr} \, H (R)$ of the right Smirnov multiplier $H(R) := B(R) A(R) ^{-1}$ and $\mathscr{H} ^\mathrm{t} (B) = \mathrm{Dom} \, H(R) ^*$. For any $x \in \mathscr{H} ^\mathrm{t} (B) = \mathrm{Dom} \, H(R) ^*$, 
$$ \| x \| ^2 _B = \| x \| ^2 _{\mathbb{H} ^2} + \| H(R) ^* x \| ^2 _{\mathbb{H} ^2}. $$ 
In particular,
\begin{eqnarray} \| B^\mathrm{t} \| ^2 _B & = & \| B^\mathrm{t} \| ^2 _{\mathbb{H} ^2 } + \| A(R) ^{-*} B(R) ^* B(R) 1 \| ^2 _{\mathbb{H} ^2} \nonumber \\
& = & \| B^\mathrm{t} \| ^2 _{\mathbb{H} ^2 } + \left\langle A(R) ^{-*} B(R) ^* B(R) 1, A(R) ^{-*} (I - A(R) ^* A(R) ) 1 \right\rangle_{\mathbb{H} ^2} \nonumber \\
& = & \| B^\mathrm{t} \| ^2 _{\mathbb{H} ^2 } + \left\langle A(R) ^{-*} B(R) ^* B(R) 1, A(R) ^{-*} 1\right\rangle_{\mathbb{H} ^2}  - \left\langle A(R) ^{-*} B(R) ^* B(R) 1, A(R) 1 \right\rangle_{\mathbb{H} ^2} \nonumber \\
& = &  \| B^\mathrm{t} \| ^2 _{\mathbb{H} ^2 }  + \frac{1}{A(0)} \left\langle A(R) ^{-*} (I - A(R) ^* A(R) 1, 1 \right\rangle_{\mathbb{H} ^2} - \| B ^\mathrm{t} \| ^2 _{\mathbb{H} ^2} \nonumber \\
& = & \frac{1}{A(0) ^2} - \frac{1}{A(0)}\left\langle 1, A(R)1\right\rangle_{\mathbb{H} ^2} \nonumber \\
& = & \frac{1}{A(0) ^2} -1. \nonumber \end{eqnarray}
\end{proof}
The conclusion of the previous lemma holds more generally with a different argument:
\begin{lemma} \label{zeroform}
If $b \in [\mathbb{H} ^\infty _d ] _1$ is non-CE and $a \in [ \mathbb{H} ^\infty _d ] _1$ is its outer Sarason function, then 
$$ \| b ^\mathrm{t} \| ^2 _b = \frac{1}{a(0) ^2} -1. $$
\end{lemma}
\begin{proof}
For $\omega \neq \emptyset$, recall that the Taylor coefficients of $a$ are given by the formula:
$$ \hat{a} _\omega = - a(0) \left\langle b^\mathrm{t}, X^\omega b^\mathrm{t}\right\rangle _b. $$ 
Note that we can write $b^\mathrm{t} = YX b^\mathrm{t} + b(0) 1$ where $Y= L | _{\mathscr{H} ^\mathrm{t} (b)}$. Substituting this into the above formula yields:
$$ \hat{a} _\omega  =  - a(0) \left( \left\langle YX b^\mathrm{t}, X^\omega b^\mathrm{t}\right\rangle_b + \overline{b(0)} \left\langle 1, X^\omega b^\mathrm{t}\right\rangle_b \right). $$
Then, applying Lemma \ref{adjX} and Lemma \ref{rank2defect},
\begin{eqnarray} \left\langle 1, X^\omega b^\mathrm{t}\right\rangle_b & = & \left\langle K_0 ^b + \overline{b} (0) b^\mathrm{t}, X^\omega b^\mathrm{t} \right\rangle_b \nonumber \\
& = & \hat{b} _\omega - \frac{b(0)}{a(0)} \hat{a} _\omega. \nonumber \end{eqnarray} Now calculate,
\begin{eqnarray} -a(0) \left\langle YXb^\mathrm{t}, X^\omega b^\mathrm{t}\right\rangle _b & = & -a(0)\left( \left\langle X^*X b^\mathrm{t}, X^\omega b^\mathrm{t}\right\rangle _b + \boldsymbol{b} ^* \boldsymbol{b} \left\langle b^\mathrm{t}, X^\omega b^\mathrm{t}\right\rangle_b \right) \nonumber \\
& = & -a(0) \left\langle b^\mathrm{t}, X^\omega b^\mathrm{t}\right\rangle_b + a(0) \overline{b(0)} b_\omega + a(0) ^3 \| b ^\mathrm{t} \| ^2 _b \left\langle b^\mathrm{t}, X^\omega b^\mathrm{t}\right\rangle_b \\ 
& &  - a(0) (1 - |b(0) | ^2 - a(0) ^2 ) \left\langle b^\mathrm{t}, X^\omega b^\mathrm{t}\right\rangle_b \nonumber \\
& = & \hat{a}_\omega + a(0) \overline{b(0)} \hat{b} _\omega - a(0) ^2 \| b ^\mathrm{t} \| ^2 _b \hat{a} _\omega + (1 - |b (0) |^2 - a(0) ^2 ) \hat{a}_\omega. \nonumber \end{eqnarray} 
Putting this all together gives:
\begin{eqnarray} \hat{a} _\omega & = & \hat{a}_\omega + a(0) \overline{b(0)} \hat{b}_\omega - a(0) ^2 \| b ^\mathrm{t} \| ^2 _b \hat{a} _\omega + (1 - |b (0) |^2 - a(0) ^2 ) \hat{a}_\omega \nonumber \\
& & -a(0) \overline{b} (0) \left( \hat{b} _\omega - \frac{b(0)}{a(0)} \hat{a}_\omega \right)  \nonumber \\
& = & 2 \hat{a} _\omega - a(0) ^2 \| b ^\mathrm{t} \| ^2 _b \hat{a}_\omega - (|b(0) | ^2 + a(0) ^2 ) \hat{a} _\omega + |b (0) | ^2 \hat{a} _\omega. \nonumber \end{eqnarray} 
Solving for $\| b ^\mathrm{t} \| ^2 _b$ gives:
$$ a(0) ^2 \| b ^\mathrm{t} \| ^2 _b \hat{a} _\omega = \hat{a} _\omega - a(0) ^2 \hat{a} _\omega, $$ or equivalently,
$$ \| b ^\mathrm{t} \| ^2 _b = \frac{1}{a(0) ^2} -1. $$ This is of course provided that $\hat{a} _\omega \neq 0$ for at least one $\omega \neq \emptyset$. However if this were true $a = a(0)1$ which can only happen trivially.
\end{proof}

\begin{prop} \label{bdecomp}
The vector $b^\mathrm{t}$ has the direct sum decomposition 
$$ b^\mathrm{t} = \frac{1}{D(0)} B^\mathrm{t} \oplus \frac{-1}{D(0)} B(R) K_0 ^D, $$ with respect to $\mathscr{H} ^\mathrm{t} (b) = \mathscr{H} ^\mathrm{t} (B) \oplus B(R) \mathscr{H} ^\mathrm{t} (D)$. 
\end{prop}
\begin{proof}
As before we write $b^\mathrm{t} = YX b^\mathrm{t} + b(0) 1$. Since $B$ is non-CE, we have that $1 \in \mathscr{H} ^\mathrm{t} (B)$. To compute the projection of $YX b^\mathrm{t}$ onto $\mathscr{H} ^\mathrm{t} (B)$, calculate $\left\langle K^B \{Z ,y , v \}, YX b^\mathrm{t}\right\rangle_b. $ First,
\begin{eqnarray} Y^* K^B \{Z, y ,v \} & = & X K^B  \{ Z , y , v \} + \left\langle b^\mathrm{t}, K^B \{Z , y ,v \}\right\rangle_b \boldsymbol{b} \nonumber \\
& = & L^* \left( K \{Z, y ,v \} - K\{ Z , y , B^\mathrm{t} (Z) v \} B^\mathrm{t} \right) + \left\langle b^\mathrm{t}, K^B \{Z , y ,v \}\right\rangle_b \boldsymbol{b} \nonumber \\
& = & K \{Z , Z^* y , v \} - K \{ Z , Z^* y , B^\mathrm{t}{Z} v\} B^\mathrm{t} - K\{ Z , y , B^\mathrm{t} (Z) v \} (0) \boldsymbol{B}  \nonumber \\
& & + \left\langle b^\mathrm{t}, K^B \{Z , y ,v \}\right\rangle_b \boldsymbol{b} \nonumber \\ 
& = & K^B \{ Z , Z^*y , v \} - \overline{y^* B^\mathrm{t} (Z) v } \boldsymbol{B} + \left\langle b^\mathrm{t}, K^B \{Z , y ,v \}\right\rangle_b \boldsymbol{b}. \nonumber \end{eqnarray} 
Putting this all together yields:
\begin{eqnarray}  \left\langle K^B \{Z ,y , v \}, YX b^\mathrm{t}\right\rangle_b  & = &  \left\langle  Y^* K^B \{Z , y, v \}, \boldsymbol{b}\right\rangle_b \nonumber \\
& = & \left\langle  Y^* K^B \{Z , y, v \}, B(R) \boldsymbol{D} \oplus D(0) \boldsymbol{B}\right\rangle_b \nonumber \\
& = & D(0) \left\langle K^B \{ Z , Z^*y , v \} - \overline{y^* B^\mathrm{t} (Z) v } \boldsymbol{B}, \boldsymbol{B}\right\rangle_B + \left\langle K^B \{ Z , y ,v \}, b^\mathrm{t}\right\rangle_b \boldsymbol{b} ^* \boldsymbol{b} \nonumber \\
& = & D(0)  y^* (B^t (Z) - B(0) ) v  - D(0) y^* B^\mathrm{t} (Z)v (1 - |B(0) | ^2 - A(0) ^2 )    \nonumber \\
& & + \left\langle  K^B \{ Z , y ,v \}, b^\mathrm{t}\right\rangle (1 - |b(0) | ^2 - a(0) ^2 ) \nonumber \\
& = & - b(0) y^* v + D(0) (|B(0) | ^2 + A(0) ^2 ) y^* B^\mathrm{t} (Z) v  \nonumber \\
& & + \left\langle  K^B \{ Z , y ,v \}, b^\mathrm{t}\right\rangle (1 - |b(0) | ^2 - a(0) ^2 ). \nonumber \end{eqnarray}

Finally,
\begin{eqnarray}  \left\langle K^B \{Z ,y , v \},  b^\mathrm{t}\right\rangle_b  & = &  - b(0) y^* v + D(0) (|B(0) | ^2 + A(0) ^2 ) y^* B^\mathrm{t} (Z) v  \nonumber \\
& & +\left\langle  K^B \{ Z , y ,v \}, b^\mathrm{t}\right\rangle (1 - |b(0) | ^2 - a(0) ^2 ) + y^*v b(0). \nonumber \end{eqnarray} Solving for $\left\langle  K^B \{ Z , y ,v \}, b^\mathrm{t}\right\rangle$ gives:
$$ D(0) ^2 (|B(0) |^2 + A(0) ^2 ) \left\langle  K^B \{ Z , y ,v \}, b^\mathrm{t}\right\rangle  =  D(0) (|B(0) | ^2 + A(0) ^2 ) y^* B^\mathrm{t} (Z) v. $$ Equivalently,
$$ Pb^\mathrm{t} = \frac{1}{D(0)} B^\mathrm{t}, $$ where $P := P _{\mathscr{H} ^\mathrm{t} (B)}$ denotes the orthogonal projection of $\mathscr{H} ^\mathrm{t} (b)$ onto the subspace $\mathscr{H} ^\mathrm{t} (B)$.  

Let $Q:= P_{B(R) \mathscr{H} ^\mathrm{t} (D)} : \mathscr{H} ^\mathrm{t} (b) \rightarrow B(R) \mathscr{H} ^\mathrm{t} (D)$ be the orthogonal projection so that $Q=I-P$. It follows that 
$$ Qb^\mathrm{t} = b^\mathrm{t} - \frac{1}{D(0)} B^\mathrm{t}. $$ We can write:
\begin{eqnarray} Q b^\mathrm{t} & = & D^\mathrm{t} B^\mathrm{t} - \frac{1}{D(0)} B^\mathrm{t}  \nonumber \\
& = & - \frac{1}{D(0)} \left( B^\mathrm{t} - D^\mathrm{t} B^\mathrm{t} D(0) \right) \nonumber \\
& = & - \frac{1}{D(0)} B(R) K^D _0, \nonumber \end{eqnarray} so that the direct sum decomposition of $b^\mathrm{t}$ with respect to
$\mathscr{H} ^\mathrm{t} (b) = \mathscr{H} ^\mathrm{t} (B) \oplus B(R) \mathscr{H} ^\mathrm{t} (D)$ is:
$$ b^\mathrm{t} = \frac{1}{D(0)} B^\mathrm{t} \oplus \frac{-1}{D(0)} B(R) K^D _0. $$ 
\end{proof}

\begin{remark} \label{constructa}
Under the above assumptions, $a = AD$, $b=BD$, $A$ and $B$ are a Smirnov column--inner pair, $A$ and $a$ are the Sarason outer functions of $B$ and $b$, $D \in [\mathbb{H} ^\infty _d ] _1$ is outer, CE and $\mathscr{H} ^\mathrm{t} (D) \cap \mathscr{M} ^\mathrm{t} (A^* ) = \{ 0 \}$. If it is true that $c = \left( \begin{smallmatrix} b \\ a \end{smallmatrix} \right)$ is always inner, whenever $a,b \in [\mathbb{H} ^\infty _d ] _1$ and $a$ is the NC Sarason outer function of $b$, then $c = C D$ where $C := \left( \begin{smallmatrix} B \\ A \end{smallmatrix} \right)$ is also inner. However, if left multiplication by $c = CD$ and $C$ are both isometries, this implies that $D(L)$ is also an isometry, \emph{i.e.} inner so that $D$ is both inner and outer which means that $D(L)$ is an isometry with dense range, \emph{i.e.} a unitary. By Davidson--Pitts \cite[Corollary 1.5]{DP-inv}, the left multiplier algebra of $\mathbb{H} ^2 _d$ contains no non-constant normal elements. Hence $D = \zeta$ for some $\zeta \in \partial \mathbb{D}$ and since we can assume $D(0) >0$, $D =1$ and $c=C$.  

In order to show that $c = \left( \begin{smallmatrix} b \\ a \end{smallmatrix} \right)$ is always inner, one might hope to show that the fact that $a=AD$ is the Sarason function of $b=BD$ where $A$ is the Sarason function of $B$ can only hold if $D=1$ is constant. However, these assumptions do not seem to place any such restrictions on $D$. Indeed, since $a(R) = A(R) D(R)$, if 
$$ A^\mathrm{t} := \sum _\beta \hat{A} ^\mathrm{t} _\beta L^\beta 1, \quad \mbox{and} \quad D^\mathrm{t} = \sum _\alpha \hat{D} ^\mathrm{t} _\alpha L^\alpha 1, $$ 
then 
$$ a ^\mathrm{t} = D ^\mathrm{t} A ^\mathrm{t} = \sum _\gamma \hat{a}^\mathrm{t} _\gamma L^\gamma 1, $$ where 
$$ \hat{a}^\mathrm{t} _\gamma = \sum _{\alpha \cdot \beta = \gamma} \hat{D}^\mathrm{t} _\alpha \hat{A}^\mathrm{t} _\beta. $$ Using that $\hat{a}^\mathrm{t} _\gamma = \hat{a}_{\gamma ^\mathrm{t}}$ we obtain
\begin{eqnarray} \hat{a}_\gamma & = & \sum _{\alpha \cdot \beta = \gamma ^\mathrm{t}} \hat{D} _{\alpha ^\mathrm{t}} \hat{A} _{\beta ^\mathrm{t} } \nonumber \\
& = & \sum _{\lambda \cdot \sigma = \gamma} \hat{D}_\sigma \hat{A}_\lambda. \nonumber \end{eqnarray} 
On the other hand, using the Sarason formula for the non-vacuum Taylor coefficients of $a$, 
\begin{eqnarray} \hat{a} _\gamma & = & -a(0) \left\langle b^\mathrm{t}, X^\gamma b^\mathrm{t}\right\rangle_b \nonumber \\
& = & -a(0) \left\langle \frac{1}{D(0)} B^\mathrm{t} \oplus \frac{-1}{D(0)} B(R) K^D _0, \sum _{\substack{\lambda \cdot \sigma = \gamma \\ \lambda \neq \emptyset}} \hat{D}_\sigma X^\lambda B^\mathrm{t} \oplus B(R) X^\gamma D^\mathrm{t}\right\rangle_b \nonumber \\
& = & -A(0) \sum _{\substack{\lambda \cdot \sigma = \gamma \\ \lambda \neq \emptyset}} \hat{D}_\sigma \left\langle B^\mathrm{t}, X^\lambda B^\mathrm{t}\right\rangle_B + A(0) \left\langle K^D _0, X^\gamma D ^\mathrm{t}\right\rangle_D \nonumber \\
& = &  \sum _{\substack{\lambda \cdot \sigma = \gamma \\ \lambda \neq \emptyset}} \hat{D} _\sigma \hat{A} _\lambda + A(0) \hat{D} _\gamma, \nonumber \end{eqnarray} which is the same as the previous formula for any $\gamma \neq \emptyset$. In the case of $\hat{a} _\emptyset = a(0)$,
\begin{align*} a(0) ^2 & =  \frac{1}{1 + \| b ^\mathrm{t} \| _b ^2},  && \mbox{by Lemma \ref{zeroform},} \\
& =   \frac{1}{ 1 + \left( \frac{\| B ^\mathrm{t} \| _B ^2 +1 - D(0) ^2}{D(0) ^2} \right),  } && \mbox{by Proposition \ref{bdecomp},}\\
& =  D(0) ^2 \frac{1}{1 + \| B ^\mathrm{t} \| _B ^2} = D(0) ^2 A(0) ^2, &&  \mbox{by Lemma \ref{zeroform} again.} \end{align*}
\end{remark}

\begin{remark} \label{Dzeroform}
Given $a,b,c,A,B,C$ and $D$ as above, one can readily check that 
$$ D(0) ^2 = \frac{1 + \| B ^\mathrm{t} \| ^2 _B}{1 + \| b ^\mathrm{t} \| ^2 _b} = \frac{1 - \| \boldsymbol{b} \| ^2 _b}{1 - \| \boldsymbol{B} \| ^2 _B}, $$ where recall that $B = bD$.

Also by Corollary \ref{SarSar}, if $b \in [\mathbb{H} ^\infty _d ] _1$ is non-CE and outer with Sarason outer function $a$, then $b$ is the Sarason outer function of $a$, so that by Theorem \ref{orthodecomp}, we have that 
\begin{align*} \mathscr{H} ^\mathrm{t} (c) &= \mathscr{H} ^\mathrm{t} (C) \oplus C(R) \mathscr{H} ^\mathrm{t} (D), \\
\mathscr{H} ^\mathrm{t} (b) &= \mathscr{H} ^\mathrm{t} (B) \oplus B(R) \mathscr{H} ^\mathrm{t} (D), \quad \quad  \mbox{and} \\
\mathscr{H} ^\mathrm{t} (a) &= \mathscr{H} ^\mathrm{t} (A) \oplus A(R) \mathscr{H} ^\mathrm{t} (D). \end{align*}
Moreover, one can verify, in this case, that for any $x \in \mathscr{H} ^\mathrm{t} (B)$ and $y \in \mathscr{H} ^\mathrm{t} (D)$, the linear map
$$ x \oplus B(R) y \ \stackrel{U}{\mapsto} \ H(R)^* x \oplus -A(R) y; \quad \quad H(R) := B(R) A(R) ^{-1}, $$ is an onto isometry that interwines $X ^{(b)} = L^* | _{\mathscr{H} ^\mathrm{t} (b)}$ with $X^{(a)} = L^* | _{\mathscr{H} ^\mathrm{t} (a)}$. 
\end{remark}

\section{Smirnov column--inner pairs and Toeplitz factorization}

We are now sufficiently prepared to prove one of our main results which shows how the open Question \ref{openQ} on factorizability of positive semi-definite left Toeplitz operators with factorable minorants is equivalent to several corresponding questions regarding de Branges--Rovnyak spaces for non-CE $b$, the relationship between a non-CE $b$ and its outer NC Sarason function and the graphs of left Smirnov multipliers and their adjoints:

\begin{thm} \label{main1}
Given a non-column--extreme left multiplier $b \in [ \mathbb{H} ^\infty _d ] _1$ with outer Sarason function $a \in [ \mathbb{H} ^\infty _d ] _1$, let $ \left( \begin{smallmatrix} B \\ A \end{smallmatrix} \right) D = CD$ be the inner--outer factorization of $c = \left( \begin{smallmatrix} b \\ a \end{smallmatrix} \right)$. The following are equivalent:
\begin{enumerate}
\item[\emph{(i)}] $c$ is inner so that $D \equiv 1$. That is $a,b$ are a column--inner Smirnov pair.
\item[\emph{(ii)}] If $X := L^* | _{\mathscr{H} ^\mathrm{t} (b)}$, then $X^*$ is pure.
\item[\emph{(iii)}] $X^*$ obeys the weak purity condition $\| X^{(n)} b ^\mathrm{t} \|_b ^2 \rightarrow 0$.
\item[\emph{(iv)}] The NC Szeg\"o kernels, or the free polynomials are dense in $\mathscr{H} ^\mathrm{t} (b)$.
\item[\emph{(v)}] $\mathscr{H} ^\mathrm{t} (b)$ is the domain of the adjoint of the closed, densely--defined Smirnov multiplier $H(R) = b(R) a(R) ^{-1}$. 
\end{enumerate}
\end{thm}

\begin{proof}
Equivalence of (i), (ii) and (iii) was proven in Corollary \ref{CEcolumn}.

(i)$\Leftrightarrow$(iv): Since $b$ is non-CE, $\mathscr{H} ^\mathrm{t} (b)$ always contains $\mathbb{C} \{ \mathbb{\mathfrak{z}} \} $ by \cite[Corollary 6.14]{JM-freeCE}. Moreover, $b^\mathrm{t} \in \mathscr{H} ^\mathrm{t} (b)$ by \cite[Theorem 4.2]{JM-freeCE} and $K_0 ^b = 1 - b^\mathrm{t} \overline{b(0)}$ so that $1 \in \mathscr{H} ^\mathrm{t} (b)$ as well. Alternatively, Theorem \ref{orthodecomp} implies that $\mathscr{H} ^\mathrm{t} (b) = \mathscr{H} ^\mathrm{t} (B) \oplus B(R) \mathscr{H} ^\mathrm{t} (D)$ and $\mathscr{H} ^\mathrm{t} (B) = \mathrm{Dom} \, H(R) ^*$, equipped with the graph--norm, where $H(R) = B(R) A(R) ^{-1}$ is a closed right Smirnov multiplier by \cite[Theorem 4.17]{JM-freeSmirnov}. By \cite[Corollary 3.13]{JM-freeSmirnov}, the free polynomials and the NC Szeg\"o kernels are both cores for $H(R)^*$, hence dense in $\mathscr{H} ^\mathrm{t} (B) \subseteq \mathscr{H} ^\mathrm{t} (b)$. It follows that the free polynomials/ NC Szeg\"o kernels are dense in $\mathscr{H} ^\mathrm{t} (b)$ if and only if $\mathscr{H} ^\mathrm{t} (B) = \mathscr{H} ^\mathrm{t} (b)$ in which case $D\equiv 1$ and $B = b$. This also establishes (iv)$\Leftrightarrow$(v). 
\end{proof}

In one variable, each of the above conditions is true. For example, here is an argument which proves that $c=C$ is inner using Theorem \ref{orthodecomp} and Corollary \ref{AstarD}.
\begin{prop}
    If $d=1$ then $c=C$ is inner.
\end{prop}
\begin{proof}
If $x \in \mathscr{H} (D)$ is any non-zero vector then 
$$ A(S) ^* x \in \mathscr{H} (D) \cap \mathscr{M} (A^*), $$ since $\mathscr{H} (D)$ is $S^*-$invariant. Corollary \ref{AstarD} then implies that $A(S) ^* x =0$ so that $x=0$ since $A$ is outer. Hence $\mathscr{H} (D) = \{ 0 \}$, $D =1$ and $c = C$ is inner.
\end{proof}
When $d>1$ the above does not work since $A(R)^* x$ need not belong to $\mathscr{H} ^\mathrm{t} (D)$ as $\mathscr{H} ^\mathrm{t} (D)$ is $L-$co-invariant but generally not $R-$co-invariant.

\begin{cor} \label{answer}
The following two statements are equivalent:
\begin{itemize}
    \item[\emph{(i)}] The equivalent conditions of Theorem \ref{main1} hold for every non-column--extreme $b \in [\mathbb{H} ^\infty _d ] _1$.

    \item[\emph{(ii)}] A positive semi-definite left Toeplitz operator, $T \in \mathscr{L} (\mathbb{H} ^2 _d )$, is factorable if and only if it has a non-trivial, positive semi-definite and factorable left Toeplitz minorant.

\end{itemize}
\end{cor}

Before proceeding with the proof, it will be convenient to recall the concept of a \emph{positive NC measure}. A positive NC measure, $\mu$, is any positive linear functional on the \emph{free disk system}, $\mathscr{A} _d$,
$$ \mathscr{A} _d := \left( \mathbb{A} _d  + \mathbb{A} _d ^* \right) ^{-\| \cdot \|}, \quad \mbox{where} \quad \mathbb{A} _d := \mathrm{Alg} \{ I , L_1 , \cdots , L_d \} ^{-\| \cdot \|}, $$ is the \emph{free disk algebra}, see \cite{JM-ncld,JM-ncFatou,JM-freeCE,JM-freeAC} for details. We denote the set of all positive NC measures by $ \left( \mathscr{A} _d   \right) ^\dag _+$. When $d=1$, $\mathscr{A} _d \simeq \mathscr{C} (\partial \mathbb{D} )$ can be identified with the commutative $C^*-$algebra of continuous functions on the complex unit circle, $\partial \mathbb{D}$ and then the Riesz--Markov theorem identifies any positive NC measure $\mu \in ( \mathscr{A} _1 ^\dag ) _+$ with a positive, finite and regular Borel measure on $\partial \mathbb{D}$. The appropriate analogue of normalized Lebesgue measure is the so-called \emph{vacuum state}, $m$,
$$ m(L^\omega ) := \left\langle 1, L^\omega 1\right\rangle_{\mathbb{H} ^2} = \delta _{\omega, \emptyset}. $$ Here, if $d=1$,
$$ m(S^k ) = \left\langle 1, S^k1\right\rangle_{H^2} = \int _{\partial \mathbb{D}} \zeta ^k \check{m} (d\zeta ) = \delta _{k, 0}, $$ is the unique positive linear functional corresponding to normalized Lebesgue measure, $\check{m}$, on the circle. If $T \geq 0$ is any positive semi-definite left Toeplitz operator, observe that the linear functional 
$$ \mu _T (L^\omega ) := \left\langle 1, TL^\omega 1\right\rangle_{\mathbb{H} ^2}, $$ extends to a positive NC measure on the free disk system, in which case $T$ can be thought of as the \emph{NC Radon--Nikodym derivative} of $\mu _T$ with respect to $m$. 

\begin{proof}

\noindent (i) $\Leftrightarrow$ (ii): To connect Toeplitz factorization to the theory we have developed for non-CE multipliers, define the positive NC measure $\mu _T \in  \left( \mathscr{A} _d   \right) ^\dag _+$ by
$$ \mu _T (L ^\omega ) := \left\langle 1, T L^\omega 1\right\rangle_{\mathbb{H} ^2}. $$ This is an absolutely continuous NC measure and $T$ is its NC Radon--Nikodym derivative. By \cite{JM-ncFatou}, since $T$ is bounded, there is a $b_T \in [ \mathbb{H} ^\infty _d ] _1$ so that 
$$ T = (I - b_T (R) ^* ) ^{-1} (I - b_T (R) ^* b_T (R) ) (I - b_T (R) ) ^{-1} \geq 0. $$ In more detail, since $T$ is bounded, $\mu _T \leq \| T \| m$, where recall that $m(L^\omega ) := \left\langle 1, L^\omega 1\right\rangle_{\mathbb{H} ^2}$ is NC Lebesgue measure \cite{JM-ncld}. Hence $\mu _T$ is absolutely continuous in the sense of the NC Lebesgue decomposition of \cite{JM-ncld,JM-ncFatou}. By the NC Fatou theorem of \cite{JM-ncFatou}, if $\mu _T = \mu _{b_T}$ is the NC Clark measure of $b_T \in [\mathbb{H} ^\infty _d ]_1$, 
$$ \mu _T (p(L) ^* q(L) ) = \left\langle p, Tq\right\rangle_{\mathbb{H} ^2} = \left\langle \widetilde{T}^{\frac{1}{2}} p, \widetilde{T} ^{\frac{1}{2}} q\right\rangle _{\mathbb{H} ^2}, $$ where 
$\widetilde{T}$ is the closed, positive semi-definite left Toeplitz operator obtained as the strong resolvent limit of the bounded and positive semi-definite left Toeplitz operators
$$ T_r := (I -b_T (rR) ^*) ^{-1}  (I - b_T (r R) ^* b_T (r R) ) (I - b_T (r R) ) ^{-1}, $$ as $r \uparrow 1$. The above equation shows that $T$ and $\widetilde{T}$ define the same bounded and positive semi-definite quadratic forms, so that $T = \widetilde{T}$ and $\widetilde{T}$ is bounded. Strong resolvent convergence of a bounded net of positive semi-definite operators to a bounded and positive semi-definite operator is equivalent to convergence in the strong operator topology \cite[Theorem VIII.20]{RnS1}. Hence $T_r \stackrel{SOT}{\rightarrow} T$. Finally, $b_T (rR) \stackrel{SOT-*}{\rightarrow} b_T (R)$ and the formula for $T$ follows.

Since we can assume that $T \neq 0$, $b_T (R)$ is not inner and it is clear that $T$ has a factorable minorant/ is factorable if and only if the numerator $I-b_T(R) ^* b_T (R)$ has factorable minorant/ is factorable. In particular, if such a $T$ with a factorable minorant is always factorable, this is equivalent to saying that any non-CE $b \in [ \mathbb{H} ^\infty _d ] _1$ is such that $c = \left( \begin{smallmatrix} b \\ a \end{smallmatrix} \right)$, where $a$ is the Sarason function of $b$, is always inner. 
\end{proof}

\section{Application to rational multipliers}
\label{sect:NCratmult}

In this section we will establish a Fej\'er--Riesz theorem for NC rational multipliers of Fock space. The classical Fejer-Riesz theorem states that if a trigonometric polynomial
\[ p(e^{i\theta}) = \sum_{n=-N}^N \hat{p} _ne^{in\theta}  \]
is nonnegative for all $\theta\in[0, 2\pi]$, then there is a single analytic polynomial $q(z) =\sum_{n=0}^N \hat{q}_n z^n$, of the same degree as $p$, so that
\[
  p(e^{i\theta})  = |q(e^{i\theta})|^2
\]
for all $\theta\in[0, 2\pi]$. As a trivial consequence, there is a similar factorization in the rational case: namely, if $\mathfrak{r} (z) = \frac{p_1(z)}{p_2(z)}$ is a rational function, regular in $|z|\leq 1$, and
\[
  \mathfrak{r} (e^{i\theta}) +\overline{ \mathfrak{r} (e^{i\theta}) } \geq 0
\]
for all $\theta$, then there is another rational function $\mathfrak{q}$, regular in $|z|\leq 1$ and with degree at most the degree of $\mathfrak{r}$, such that
\begin{equation}\label{eqn:classical-rat-FR}
  \mathfrak{r} (e^{i\theta}) + \overline{\mathfrak{r} (e^{i\theta}) }  = |\mathfrak{q} (e^{i\theta})|^2
\end{equation}
for all $\theta$. Indeed, suppose $\mathrm{deg} \, p_1, p_2\leq N$. Since, by hypothesis,
\[
  \frac{p_1(e^{i\theta})}{p_2(e^{i\theta})} +  \frac{\overline{p_1(e^{i\theta})}}{\overline{p_2 (e^{i\theta} )}} \geq 0,
\]
we can clear the denominators to obtain a trigonometric polynomial of degree at most $N$, which is positive on the unit circle, and hence factors:
\[
  \overline{p_2 (e^{i\theta}) } p_1(e^{i\theta}) + \overline{p_1(e^{i\theta}) } p_2(e^{i\theta}) =: |\widetilde{p}(e^{i\theta})|^2,
\]
where $\widetilde{p} \in \mathbb{C} [ e^{i\theta} ]$ is an analytic polynomial of degree at most $N$. Dividing by $|p_2 | ^2$ again, and setting $\mathfrak{q} =  \widetilde{p}/p_2$, we arrive at (\ref{eqn:classical-rat-FR}). We see immediately from the calculation that $\mathfrak{q}$ is regular everywhere that $\mathfrak{r}$ is, and the degree of $\mathfrak{q}$ (that is, the maximum of the degrees of numerator and denominator) does not exceed the degree of $p$. 

Relatedly, if now $\mathfrak{b}$ is a rational function regular in $|z|\leq 1$ and bounded by $1$ there, we can consider the function
\[
  1-|\mathfrak{b} (e^{i\theta})|^2 \geq 0
\]
for $\theta\in [0, 2\pi]$. If $\mathfrak{b}$ is a rational inner function (a finite Blaschke product) then this expression is identically $0$. Otherwise, we can apply the same argument as above to obtain a rational function $\mathfrak{a}$ (in fact an outer function, that is, having no zeroes in $|z|<1$), of degree not exceeding the degree of $\mathfrak{b}$, and regular everywhere that $\mathfrak{b}$ is, so that
\begin{equation}\label{eqn:classical-pythagorean-mate}
  1-|\mathfrak{b} (e^{i\theta)}|^2 = |\mathfrak{a} (e^{i\theta})|^2
\end{equation}
for all $\theta$. The function $\mathfrak{a}$ is sometimes called the {\em pythagorean mate} of $\mathfrak{b}$. Thus we observe a dichotomy: either $\mathfrak{b}$ is inner, or $1-\mathfrak{b} ^* \mathfrak{b}$ ``factors completely'' as $\mathfrak{a} ^* \mathfrak{a}$. In the latter case, it follows from Szeg\"o's theorem that $\log (1-|\mathfrak{b}|^2)$ is integrable on the unit circle, and hence $\mathfrak{b}$ is not an extreme point of the unit ball of $H^\infty(\mathbb D)$. (This dichotomy breaks down with the rationality assumption, there exist non-inner $b$ so that $1-|b|^2$ is not log-integrable, and hence $1-b^*b$ is nonzero, but does not factor. Recall that such a $b$ is necessarily a non-inner extreme point of $[H^\infty ]_1$.)

The Fej\'er--Riesz theorem has been extended to Fock space by G. Popescu \cite[Theorem 1.6]{Pop-multi}:

\begin{thm*}[NC Fej\'er--Riesz. Popescu]
Let $T := \mathrm{Re} \, p (R) \geq 0$ be a positive semi-definite left Toeplitz operator on $\mathbb{H} ^2 _d$, where $p \in \mathbb{C} \{ \mathbb{\mathfrak{z}} \} $ and $deg \ p=N$. Then there is an outer free polynomial $q \in \mathbb{C} \{ \mathbb{\mathfrak{z}} \} $ with $deg \ q =N$ so that $q(R) ^* q(R) =T$.
\end{thm*}

However, since NC rational functions are not, in general, expressible as quotients of polynomials, the simple arguments just given in the one-variable case do not generalize to the NC setting. Nonetheless, it turns out there are NC versions of (\ref{eqn:classical-rat-FR}) and (\ref{eqn:classical-pythagorean-mate}) available, though the proof will require completely different methods. In fact, (\ref{eqn:classical-rat-FR}) can be more or less reduced to (\ref{eqn:classical-pythagorean-mate}), by means of the Cayley transform, and (\ref{eqn:classical-pythagorean-mate}) can be proven directly by constructing an appropriate {\em realization} of $a$ from a suitable realization of $b$ as in Section \ref{sect:CE}. In particular, the realizations we work with arise from the de Branges--Rovnyak style functional model, and our various claims can be established by analysis of this model.

In particular, we replace ``degree'' with ``size of a minimal Fornasini--Marchesini realization'' (an analogue of the MacMillan degree), and are able to prove an NC version of (\ref{eqn:classical-pythagorean-mate}) with sharp degree bounds, and an NC version of (\ref{eqn:classical-rat-FR}) with a reasonable degree bound (but we do not know if it is sharp). Likewise in the NC version of (\ref{eqn:classical-pythagorean-mate}) we are able to prove that $\mathfrak{a}$ is regular everywhere that $\mathfrak{b}$ is, while in the NC version of (\ref{eqn:classical-rat-FR}) we settle for a somewhat weaker claim (that the $\mathfrak{q}$ in the right-hand side will at least be regular across the boundary of the row ball); we do not know if it can in fact be chosen to be regular everywhere that $\mathfrak{r}$ is.

%In this section we apply the preceding results to non-commutative rational multipliers of the Fock space.
\subsection{Preliminaries on NC rational functions}

Recall that a complex NC rational expression is any valid combination of the several NC variables $\mathfrak{z} _1 , \cdots , \mathfrak{z} _d$, the complex scalars, $\mathbb{C}$, the operations $+ , \cdot , ^{-1}$ and parentheses $\left( , \right)$ with domain
$$\mathrm{Dom} \, \mathrm{r} = \bigsqcup _{n=1} ^\infty \mathrm{Dom} _n \, \mathrm{r}, \quad  \mathrm{Dom} _n \,  \mathrm{r} := \bigsqcup _{n=1} ^\infty \left\{ \left. X = (X _1 , \cdots , X_d ) \in \mathbb{C} ^{n\times n} \otimes \mathbb{C} ^{1 \times d}  \right| \ \mathrm{r} (X) \ \mbox{is defined.} \right\}. $$ We will use the notation $\mathbb{C} ^d _n := \mathbb{C} ^{n\times n} \otimes \mathbb{C} ^{1\times d}$ for a row $d-$tuple of complex $n\times n$ matrices. Such an NC rational expression is said to be \emph{valid} if its domain is not empty. An NC rational function, $\mathfrak{r}$, is then the equivalence class of valid NC rational expressions with respect to the relation $\mathrm{r} _1 \equiv \mathrm{r} _2$ if $\mathrm{r} _1 (X) = \mathrm{r} _2 (X)$ for all $X \in \mathrm{Dom} \, \mathrm{r} _1 \bigcap \mathrm{Dom} \, \mathrm{r} _2$ (this intersection is always non-empty by \cite[Footnote, page 52]{KVV-ncratdiff}), with domain $\mathrm{Dom} \, \mathfrak{r} := \cup _{\mathrm{r} \in \mathfrak{r}} \mathrm{Dom} \, \mathrm{r}$, and we write $\mathfrak{r} (Z) := \mathrm{r} (Z)$ if $Z \in \mathrm{Dom} \, \mathrm{r}$ and $\mathrm{r} \in \mathfrak{r}$.  Any NC rational function in $d-$variables, $\mathfrak{r}$, which is regular at $0$ has a unique (up to joint similarity) \emph{minimal Fornasini--Marchesini (FM) realization}. Namely, there is a quadruple $(A, B,C,D )$ with $A \in \mathbb{C} ^d _n$, $B = \left( \begin{smallmatrix} B_1 \\ \vdots \\ B_d \end{smallmatrix} \right) \in \mathbb{C} ^n \otimes \mathbb{C} ^d$, $C \in \mathbb{C} ^{1\times n}$ and $D \in \mathbb{C}$ so that for any $X \in \mathbb{C} ^d _m \cap \mathrm{Dom} \, \mathfrak{r}$,
$$ \mathfrak{r} (X) = D I_m + C \otimes I_m L_A (X) ^{-1}  B \otimes X; \quad \quad L_A (X) := I_n \otimes I_m - \sum A_j \otimes X_j, $$ where $B \otimes X := B_1 \otimes X_1 + \cdots + B_d \otimes X_d$ and this realization is minimal in the sense that $n$ is as small as possible.  We will also sometimes write $A\otimes X$ in place of $\sum A_j \otimes X_j$. Here, $L_A (\cdot ) $ is called a (monic, affine) \emph{linear pencil}. A realization is minimal if and only if it is both \emph{controllable}:
$$ \bigvee_{\substack{\omega \in \mathbb{F} ^d \\ 1 \leq j \leq d}} A^{\omega} b_j = \mathbb{C} ^{n},$$  and \emph{observable} 
$$ \bigvee A^{*\omega} C ^* = \mathbb{C} ^n, $$ see \emph{e.g.} \cite[Subsection 3.1.2]{HMS-realize}. Minimal realizations are unique up to joint similarity \cite[Theorem 2.1]{Ball-control}. If $(A,B,C,D)$ is a minimal FM realization of $\mathfrak{r}$ and $A \in \mathbb{C} ^d _n$, we will say that the \emph{minimal FM realization size} of $\mathfrak{r}$ is $n$ and write $\mathrm{dim} _{FM} \, \mathfrak{r} = n$. We will also have occasion to use \emph{descriptor realizations}. Any NC rational function, $\mathfrak{r}$, which is \emph{regular at $0$}, \emph{i.e.} so that $0 = (0 , \cdots , 0) \in \mathrm{Dom} \, \mathfrak{r}$, has a minimal descriptor realization $(A,b,c) \in \mathbb{C} ^d _m \times \mathbb{C} ^m \times \mathbb{C} ^m$: For any $Z \in \mathrm{dom} _n \, \mathfrak{r} := \mathrm{Dom} \mathfrak{r} \, \cap \, \mathbb{C} ^d _n$, 
$$ \mathfrak{r} (Z) = b^* \otimes I_n L_A (Z) ^{-1} c \otimes I_n, $$ see \cite{Volcic}. Minimality again means that $n$ is as small as possible, and this is equivalent to $c$ being $A-$cyclic (the realization is controllable) and $b$ being $A^*-$cyclic (the realization is observable), and again, minimal descriptor realizations are unique up to joint similarity. We will sometimes write $\mathrm{dim} \, \mathfrak{r} = m$ for the \emph{minimal descriptor realization size} of $\mathfrak{r}$. The following lemma shows how minimal FM realizations can be constructed from minimal descriptor realizations.

\begin{lemma} \label{FMdescriptor}
Let $(A,b,c) \in \mathbb{C} ^d _n \times \mathbb{C} ^n \times \mathbb{C} ^n$ be a descriptor realization of an NC rational function $\mathfrak{r}$ with $0 \in \mathrm{Dom} \, \mathfrak{r}$. Let $\mathscr{M} _0 := \bigvee _{\omega \neq \emptyset} A ^\omega c \subseteq \mathbb{C} ^n$ with projector $Q_0$ and define $(A^{(0)} , B^{(0)} , C^{(0)} , D^{(0)})$ by
$$ A _j ^{(0)} := A_j | _{\mathscr{M} _0}, \quad B ^{(0)} _j:= A_j c, \quad C ^{(0)} := (Q_0 b) ^* \quad \mbox{and} \quad D ^{(0)} = \mathfrak{r} (0). $$ This is an FM realization of $\mathfrak{r}$ and 
it is minimal if $(A,b,c)$ is minimal. In particular, $ \mathrm{dim} \, \mathfrak{r} -1 \leq \mathrm{dim} _{FM} \, \mathfrak{r} \leq \mathrm{dim} \, \mathfrak{r}$.
\end{lemma}
\begin{proof}
It is easily checked that $(A^{(0)} , B^{(0)} , C^{(0)} , D^{(0)})$ is a realization of $\mathfrak{r}$ and this realization will be both controllable and observable, hence minimal, if $(A,b,c)$ is. Finally, $ \mathrm{dim} \, \mathfrak{r} -1 \leq \mathrm{dim} _{FM} \, \mathfrak{r} := \mathrm{dim} \, \mathscr{M} _0 \leq \mathrm{dim} \, \mathfrak{r}$. 
\end{proof}

As proven in \cite[Theorem A]{JMS-ncrat}, given an NC rational function, $\mathfrak{r}$, regular at $0$, the following are equivalent:
\begin{itemize}
    \item[(i)] $\mathfrak{r} \in \mathbb{H} ^2 _d$.
    \item[(ii)] $\mathfrak{r} \in \mathbb{H} ^\infty _d$. 
    \item[(iii)] $\mathfrak{r} = K \{ Z , y , v \}$ is an NC Szeg\"o kernel for some $Z,y,v \in \mathbb{B} ^d _n \times \mathbb{C} ^n \times \mathbb{C} ^n$. 
    \item[(iv)] $ r \mathbb{B} ^d _\mathbb{N} \subseteq \mathrm{Dom} \, \mathfrak{r}$ for some $r>1$.
\end{itemize}
Given an NC rational $\mathfrak{b} \in [\mathbb{H} ^\infty _d ] _1$, its minimal FM realization can be constructed by cutting down its de Branges--Rovnyak realization. Namely, let 
$$ \mathscr{M} _0 (\mathfrak{b} ) := \bigvee _{\omega \neq \emptyset} L^{*\omega} \mathfrak{b} ^\mathrm{t} \subseteq \mathscr{H} ^\mathrm{t} (\mathfrak{b} ), $$ with orthogonal projection $P_0$. It is easy to check that this space is finite--dimensional since $\mathfrak{b} ^\mathrm{t} =K \{Z , y , v \} \in \mathbb{H} ^2 _d$ is also NC rational, hence an NC Szeg\"o kernel and $L ^{* \omega } K \{ Z ,y , v \} = K\{ Z , Z^{*\omega} y, v \}$, see \cite[Lemma 2]{JMS-ncratClark}. Then defining $$ A := L^* | _{\mathscr{M} _0 (\mathfrak{b} )}, \quad B := L^* \mathfrak{b} ^\mathrm{t}, \quad C := (P_0 K_0 ^\mathfrak{b}) ^* \quad \mbox{and} \quad D := \mathfrak{b} (0), $$ yields a minimal FM realization of $\mathfrak{b}$:

\begin{lemma} \label{minFM}
Let $\mathfrak{b} \in [\mathbb{H} ^\infty _d ] _1$ be a contractive and NC rational multiplier of Fock space. The above finite FM realization $( A , B , C , D )$ of $\mathfrak{b}$ obtained by compression of the de Branges--Rovnyak realization of $\mathfrak{b}$ to $\mathscr{M} _0 (\mathfrak{b} )$ is minimal.
\end{lemma}
\begin{proof}
It is easily checked that $(A,B,C,D)$ is a finite--dimensional realization of $\mathfrak{b}$. This realization is controllable by construction since
$$ \bigvee _{\substack{\omega \in \mathbb{F} ^d \\ 1 \leq j \leq d }} A ^\omega B_j = \bigvee _{\omega \neq \emptyset} L^{*\omega} \mathfrak{b} ^\mathrm{t} = \mathscr{M} _0 (\mathfrak{b} ). $$ Suppose that there is an $h \in \mathscr{M} _0 (\mathfrak{b} )$ so that $h$ is orthogonal to $\bigvee _{\omega \in \mathbb{F} ^d} A^{*\omega} C^*$. Then, for any $\omega \in \mathbb{F} ^d$,
\begin{eqnarray} 0 & = & \left\langle A^{*\omega ^\mathrm{t}} C^*, h\right\rangle = \left\langle P_0 K_0 ^\mathfrak{b}, A^{\omega} h\right\rangle \nonumber \\
& = & \left\langle K_0 ^\mathfrak{b}, L^{*\omega } h^\mathrm{t}\right\rangle_\mathfrak{b} = \hat{h} _\omega. \nonumber \end{eqnarray} It follows that all Taylor coefficients of $h \in \mathscr{M} _0 (\mathfrak{b} ) \subseteq \mathbb{H} ^2 _d$ vanish so that $h \equiv 0$. This proves that the realization is also observable, hence minimal. 
\end{proof}

\begin{prop} \label{ratfactor}
If $\mathfrak{b} \in [\mathbb{H} ^\infty _d ] _1$ is non-CE and NC rational, then its Sarason function $\mathfrak{a} \in [\mathbb{H} ^\infty _d ] _1$ is also NC rational. The minimal FM realization size of $\mathfrak{a}$ is at most that of $\mathfrak{b}$.
\end{prop}
\begin{proof}{ (of Proposition \ref{ratfactor})}
As above, let $P_0$ denote the orthogonal projection of $\mathscr{H} ^\mathrm{t} (\mathfrak{b} )$ onto $\mathscr{M} _0 (\mathfrak{b} )$, and let $(A,B,C,D)$ be the minimal and finite FM realization of $\mathfrak{b}$ obtained by compression of the de Branges--Rovnyak realization of $\mathfrak{b}$ to $\mathscr{M} _0 (\mathfrak{b} )$. This compression also yields a finite--dimensional FM colligation and hence realization for $\mathfrak{c} = \left( \begin{smallmatrix} \mathfrak{b} \\ \mathfrak{a} \end{smallmatrix} \right) $: $$  U _\mathfrak{c} = \begin{pNiceMatrix} A & B \\  C & D \\ -\mathfrak{a} (0) \langle P _0 \mathfrak{b} ^\mathrm{t}, \cdot \rangle _\mathfrak{b}   & \mathfrak{a} (0) \end{pNiceMatrix} : \begin{pNiceMatrix} \mathscr{M} _0 (\mathfrak{b} ) \\ \mathbb{C} \end{pNiceMatrix} \rightarrow \begin{pNiceMatrix} \mathscr{M} _0 (\mathfrak{b} ) \otimes \mathbb{C} ^d \\ \mathbb{C} ^2 \end{pNiceMatrix}.$$  In particular,
$$ U _\mathfrak{a} := \begin{pNiceMatrix} A & B \\ -\mathfrak{a} (0) \left\langle P_0 \mathfrak{b} ^\mathrm{t}, \cdot \right\rangle_\mathfrak{b}  & \mathfrak{a} (0) \end{pNiceMatrix} :   \begin{pNiceMatrix} \mathscr{M} _0 (\mathfrak{b} ) \\ \mathbb{C} \end{pNiceMatrix} \rightarrow \begin{pNiceMatrix} \mathscr{M} _0 (\mathfrak{b} ) \otimes \mathbb{C} ^d \\ \mathbb{C}  \end{pNiceMatrix}. $$ is a finite--dimensional and contractive FM colligation for $\mathfrak{a} $, so that $\mathfrak{a}$ is NC rational and $\mathrm{dim} _{FM} \, \mathfrak{a} \leq \mathrm{dim} \, \mathscr{M} _0 (\mathfrak{b} ) = \mathrm{dim} _{FM} \, \mathfrak{b}$.
\end{proof}

\begin{thm} \label{CEisinner}
If $\mathfrak{b}  \in [\mathbb{H} ^\infty _d ] _1$ is NC rational, then either $\mathfrak{b}$ is inner or it is not column--extreme. If $\mathfrak{b}$ is non-CE and $\mathfrak{a}$ is its NC rational Sarason function, then $\mathfrak{c} := \left( \begin{smallmatrix} \mathfrak{b} \\ \mathfrak{a} \end{smallmatrix} \right)$ is inner, so that $T:= I - \mathfrak{b} (R) ^* \mathfrak{b} (R) = \mathfrak{a} (R) ^* \mathfrak{a} ( R)$ is factorable.
\end{thm}
\begin{proof}
Suppose that $\mathfrak{b} \in [\mathbb{H} ^\infty _d ] _1$ is NC rational. Consider the right free de Branges--Rovnyak space $\mathscr{H} ^\mathrm{t} (\mathfrak{b} )$ of $\mathfrak{b}$. We define the minimal de Branges--Rovnyak FM realization $(A,B,C,D)$ as above. Recall that since $\mathfrak{b} \in \mathbb{H} ^2 _d$ is NC rational, it also has a \emph{minimal descriptor realization}, $(A' , b' , c ') \in \mathbb{C} ^d _n \times \mathbb{C} ^n \times \mathbb{C} ^n$.  By \cite[Theorem A]{JMS-ncrat}, $A'$ is jointly similar to a strict row contraction. By Lemma \ref{FMdescriptor}, one can construct a minimal FM realization, $( A ^{(0)} , B ^{(0)} , C ^{(0)} , D ^{(0)} )$, from the minimal descriptor realization $(A' ,b' , c')$ by defining
$$ \mathscr{M} _0 := \bigvee _{\omega \neq \emptyset } A^{'w} c, $$ with projector $Q_0$ and 
$$ A ^{(0)} := A' | _{\mathscr{M} _0}, \quad B ^{(0)} := A' c', \quad C ^{(0)} := (Q_0 b') ^* \quad \mbox{and} \quad D ^{(0)} = \mathfrak{b} (0). $$ Since $A ^{(0)}$ is the restriction of $A'$ to an invariant subspace and $A'$ is pure, we obtain that $A ^{(0)}$ is also pure. Namely,  
$$ \| \mathrm{Ad} _{(A ^{(0)} ) ^{(n)}} (I) \| ^2 =  \| \mathrm{Ad} _{A'} ^{(n)} (Q_0 ) \| ^2 \leq  \| \mathrm{Ad} _{A'} ^{(n)} ( I ) \| ^2 \rightarrow 0, $$ 
since $A'$ is pure. Since minimal FM realizations are unique up to joint similarity, $A$ is jointly similar to $A^{(0)}$ and is hence pure and similar to a strict row contraction. Hence $A ^*$ is jointly similar to a strict column contraction. By \cite[Lemma 1]{JMS-ncratClark}, $A ^*$ is then jointly similar to a strict row contraction and $A ^*$ defines a pure row $d-$tuple, $( A_1 ^* , \cdots , A_d ^* )$. Since $A^* \in \mathbb{C} ^d _m$ is pure, we necessarily have that $\mathfrak{b}$ is inner if $\mathfrak{b}$ is CE and $\mathfrak{c}$ is inner if $\mathfrak{b}$ is non-CE by Theorem \ref{pureisinner}.
\end{proof}

\begin{remark}
Setting $d=1$, it follows that any contractive rational multiplier of Hardy space is either not an extreme point, or it is inner. Indeed, if $\mathfrak{b}$ is a contractive rational multiplier of $H^2$ which is not inner, then $I - \mathfrak{b} (S) ^* \mathfrak{b} (S) >0$ is factorable by the Fej\'er--Riesz theorem, so that there is a contractive, outer and NC rational multiplier, $\mathfrak{a} \in H^\infty$, so that the two-component column $ \left( \begin{smallmatrix} \mathfrak{b} \\ \mathfrak{a} \end{smallmatrix} \right)$ is inner. That is, $\mathfrak{b}$ is non-CE. Further recall that a contractive multiplier of $H^2$ is column--extreme if and only if it is an extreme point. Hence, the assertion of Theorem \ref{CEisinner} that any contractive NC rational multiplier of Fock space is either inner or non-CE, is an exact analogue of classical fact.
\end{remark}

\begin{remark}
If $b \in \mathbb{C} \{ \mathbb{\mathfrak{z}} \} $ is also a free polynomial, then the $A$ from a minimal FM realization of $b$ is jointly nilpotent and it follows that the Sarason function, $a \in \mathbb{C} \{ \mathbb{\mathfrak{z}} \} $ is also a free polynomial. Also note that if $p \in [\mathbb{H} ^\infty _d ] _1$, then either $p$ is inner or $p$ is non-CE by the NC Fej\'er--Riesz theorem of Popescu \cite[Theorem 1.6]{Pop-multi}.

The above results also hold if $b = K \{ Z , y ,v \}$ is an NC Szeg\"o kernel at an infinite point $Z \in \mathbb{B} ^d _{\infty}$. In this case, since $Z$ is a strict row contraction, the Sarason function, $a$ is also an NC Szeg\"o kernel at a potentially infinite point and so is the outer factor $D$, of $c = \left( \begin{smallmatrix} b \\ a \end{smallmatrix} \right)$. 
\end{remark}

\subsection{NC rational Fej\'er--Riesz}

%The classical Fej\'er--Riesz theorem states that any positive semi-definite trigonometric polynomial on the circle factors as $|q(\zeta ) | ^2$, $\zeta \in \partial \mathbb{D}$ for some analytic polynomial, $q \in \mathbb{C} [\zeta ]$. This can be restated in operator--theoretic language:

%\begin{thm*}[Fej\'er--Riesz]
%Let $T := \mathrm{Re} \, p (S ) \geq 0$ be a positive semi-definite Toeplitz operator, where $p \in \mathbb{C} [z]$.
%Then $T = q(S) ^* q(S)$ is factorable via some outer $q \in \mathbb{C} [ z ]$. 
%\end{thm*}

\begin{prop} \label{resquare}
Let $\mathfrak{r} \in \mathbb{H} ^\infty _d$ be an NC rational function. Then $\mathfrak{r} (R) ^* \mathfrak{r} (R) =  \mathrm{Re} \, \mathfrak{H} (R) \geq 0$ where $\mathfrak{H} \in \mathbb{H} ^\infty _d$ is NC rational, Herglotz and $\mathrm{dim} _{FM} \, \mathfrak{H} \leq \mathrm{dim} _{FM} \, \mathfrak{r} +1$.
\end{prop}
In the above statement, recall that an NC function, $H \in \mathscr{O} (\mathbb{B} ^d _\mathbb{N} )$ is called a \emph{left free Herglotz function}, if $\mathrm{Re} \, H(Z) \geq 0$ for all $Z \in \mathbb{B} ^d _\mathbb{N}$. An NC function, $G \in \mathscr{O} (B ^d _\mathbb{N} )$ is then called a \emph{right free Herglotz function} if $G ^\mathrm{t}$ is a left free Herglotz function. Since, in the above statement, $\mathrm{Re} \, H (R) \geq 0$, it follows also that $\mathrm{Re} \, H(L) = U_\mathrm{t} \mathrm{Re} \, H(R) U_\mathrm{t} \geq 0$, and taking inner products against NC Szeg\"o kernels shows that $H$ is a left NC Herglotz function and $H^\mathrm{t}$ is right Herglotz. 
\begin{proof}
Since $\mathfrak{r}$ is NC rational, $\mathfrak{r} ^\mathrm{t} \in \mathbb{H} ^2 _d$ is also NC rational by \cite[Lemma 2]{JMS-ncratClark} and $\mathfrak{r}, \mathfrak{r} ^\mathrm{t}$ have the same size minimal descriptor realizations. Hence $\mathfrak{r} ^\mathrm{t} \in \mathbb{H} ^\infty _d$ is an NC kernel, $\mathfrak{r} ^\mathrm{t} = K \{ Z ,y , v \}$ and $\mathfrak{r} (R) ^* \mathfrak{r} ^\mathrm{t} = K \{ Z , y , \mathfrak{r} ^\mathrm{t} (Z) v \} =: \mathfrak{s} ^{ \mathrm{t}} \in \mathbb{H} ^\infty _d$ is NC rational and the size of the minimal descriptor realization of $\mathfrak{s}$ is at most that of $\mathfrak{r}$. In more detail, if $\mathfrak{r}$ has minimal descriptor realization $(A,b,c)$, then the minimal descriptor realization of $\mathfrak{r} ^\mathrm{t}$ is $(A^\mathrm{t}, \overline{c} , \overline{b} )$, where $A^\mathrm{t} := (A_1 ^\mathrm{t} , \cdots , A_d ^\mathrm{t} )$ denotes component--wise matrix transpose and $\overline{c}$ denotes entry--wise complex conjugation \cite[Lemma 2]{JMS-ncratClark}. Moreover, if $\mathfrak{r} \in \mathbb{H} ^2 _d$ has minimal descriptor realization $(A,b,c)$, then by \cite[Theorem A]{JMS-ncrat}, we can assume, without loss in generality that $A \in \mathbb{B} ^d _\mathbb{N}$ is a strict row contraction and in this case, $\mathfrak{r}$ is equal to the NC Szeg\"o kernel vector $K \{ \overline{A} , \overline{b} , \overline{c} \}$. As before $\overline{A} = (\overline{A} _1 , \cdots , \overline{A} _d)$ and $\overline{A} _j$ denotes entry--wise complex conjugation. Putting this all together, if $(A,b,c)$ is a minimal descriptor realization of $\mathfrak{r}$, then $\mathfrak{s}$ has a descriptor realization $(A^\mathrm{t}, \overline{c} , \overline{\mathfrak{r} ^\mathrm{t}} (A^\mathrm{t} ) \overline{b} )$, where $\overline{\mathfrak{r} ^\mathrm{t}}$ denotes the power series obtained by complex conjugation of the Taylor coefficients of $\mathfrak{r} ^\mathrm{t}$. Applying Lemma \ref{FMdescriptor} to construct FM realizations from these descriptor realizations for $\mathfrak{r}$ and $\mathfrak{s}$ shows that the minimal FM realization size of $\mathfrak{s}$ obeys $\mathrm{dim} _{FM} \, \mathfrak{s} \leq \mathrm{dim} _{FM} \, \mathfrak{r} +1$. 

In particular, $\mathfrak{s} (R)$ is a bounded right multiplier. Define the bounded left Toeplitz operator, 
$$ T := \mathfrak{r} (R) ^* \mathfrak{r} (R) - \mathfrak{s} (R) ^*. $$ Since $T$ is a bounded left Toeplitz operator, it is completely determined by its `Fourier coefficients'
$$ \hat{T} _\omega := \left\langle L^\omega 1, T1\right\rangle_{\mathbb{H} ^2} \quad \mbox{and} \quad \left\langle 1, TL^\omega 1\right\rangle_{\mathbb{H} ^2} =: \hat{T} _{-\omega}, $$ and $T$ is the strong operator topology limit of its partial Ces\`aro sums \cite[Lemma 1.1]{DP-inv}. Each of these partial sums is a left Toeplitz operator of the form $p(R) + q(R) ^*$, for $p,q \in \mathbb{C} \{ \mathbb{\mathfrak{z}} \} $. We claim that $T \in \mathfrak{R} ^\infty _d$ is an analytic left Toeplitz operator, \emph{i.e.} that $T = y(R)$ is a bounded right multiplier, $y \in \mathbb{H} ^\infty _d$. To prove this, it suffices to show that the `negative Fourier coefficients' of $T$ vanish. For any $\omega \in \mathbb{F} ^d$,
\begin{eqnarray} \hat{T} _{-\omega} & = & \left\langle 1, TL^\omega 1\right\rangle_{\mathbb{H} ^2} \nonumber \\
& = & \left\langle T^*1, L^\omega 1\right\rangle = \left\langle \mathfrak{r}(R) ^* \mathfrak{r} ^\mathrm{t} - \mathfrak{s} ^\mathrm{t}, L^\omega 1\right\rangle \nonumber \\
& = & \left\langle \mathfrak{s} ^\mathrm{t} - \mathfrak{s} ^\mathrm{t}, L^\omega 1\right\rangle =0, \nonumber \end{eqnarray} and it follows that the partial Ces\`aro sums of $T$ have the form
$$ \sigma _n (T) = p_n (R), \quad \quad p_n \in \mathbb{C} \{ \mathbb{\mathfrak{z}} \} . $$ Since the partial Ces\`aro sum map, $T \mapsto \sigma _n (T)$, is a completely contractive linear map on $\mathscr{L} (\mathbb{H} ^2 _d )$, it follows that $p_n (R)$ is a uniformly bounded sequence of free polynomial right multipliers that converge to $T$ in the strong operator topology so that $T \in \mathfrak{R} ^\infty _d$, $T = y(R)$ for some $y \in \mathbb{H} ^\infty _d$ \cite[Lemma 1.1]{DP-inv}. Also note that $\hat{y} _\emptyset =0$. Hence, for any $\omega \neq \emptyset$, 
\begin{eqnarray} \hat{y} _\omega  & = & \left\langle L^{\omega ^\mathrm{t}} 1, y(R) 1\right\rangle_{\mathbb{H} ^2} \nonumber \\
& = & \left\langle L^{\omega ^\mathrm{t}}1, \mathfrak{s} ^\mathrm{t}\right\rangle_{\mathbb{H} ^2} - \left\langle L^{\omega ^\mathrm{t}}1, 1\right\rangle \overline{\mathfrak{s} (0)} \nonumber \\
& = & \hat{\mathfrak{s}} _\omega, \nonumber \end{eqnarray} so that $y(R) = \mathfrak{s} (R) - cI$ with $\overline{\mathfrak{s} (0)} = c$. In particular, $y = \mathfrak{y}$ is NC rational. Moreover,
$$ 0 \leq \mathfrak{r} (R) ^* \mathfrak{r} (R) = \mathfrak{y} (R) + \mathfrak{s} (R) ^*, $$ is positive so that $c \geq 0$ and we conclude that $$ \mathfrak{r} (R) ^* \mathfrak{r} (R) = \mathfrak{y} (R) + \mathfrak{y} (R) ^* + cI =: \mathrm{Re} \, \mathfrak{H} (R), $$ with $\mathfrak{H} := 2\mathfrak{y} + c \in \mathbb{H} ^\infty _d$, a bounded and NC rational Herglotz multiplier.
\end{proof}
\begin{lemma}
Let $T :=  \mathrm{Re} \, \mathfrak{H} (R) \geq 0$ be a positive semi-definite and NC rational left Toeplitz operator. If we define the positive NC measure $\mu _T (L^\alpha ) := \left\langle 1, TL^\alpha 1\right\rangle _{\mathbb{H} ^2}$, then $\mu _T$ is the \emph{NC Clark measure} of a contractive, NC rational $\mathfrak{b} _T \in [ \mathbb{H} ^\infty _d ] _1$ where $\mathfrak{b} _T = \frac{\mathfrak{H} -1}{\mathfrak{H} + 1}$. 
\end{lemma}
Recall that if $\mu \in  \left( \mathscr{A} _d   \right) ^\dag _+$ is any positive NC measure, then one can define its left \emph{free Herglotz--Riesz transform}, 
$$ H_\mu (Z) := \mathrm{id} _n \otimes \mu \circ (I + ZL^*) (I - ZL ^* ) ^{-1}; \quad \quad Z \in \mathbb{B} ^d _n, $$ where $ZL ^* := Z_1 \otimes L_1 ^* + \cdots + Z_d \otimes L_d ^*$ \cite{JM-freeAC,JM-freeCE}. The NC function, $H_\mu \in \mathscr{O} (\mathbb{B} ^d _\mathbb{N} )$, is an \emph{NC Herglotz function}, \emph{i.e.} for any $Z \in \mathbb{B} ^d _n$, $\mathrm{Re} \, H_\mu (Z) \geq 0$. As in the classical, $d=1$ setting, the inverse Cayley transform of $H_\mu$, 
$$ b _\mu (Z) := ( H_\mu (Z) - I_n )(H_\mu (Z) + I_n ) ^{-1}; \quad \quad Z \in \mathbb{B} ^d _n, $$ is then a contractive left multiplier of Fock space so that $\mu = \mu _{b _\mu}$ is the \emph{NC Clark measure} of $b_\mu$ \cite{JM-freeAC,JM-freeCE}.
\begin{proof}
The (left) NC Herglotz--Riesz transform of $\mu _T$ can be expanded as a power series:
\begin{eqnarray} \mathfrak{H}_T (Z) & = & 2\sum _\alpha Z^{\alpha ^\mathrm{t}} \left\langle L^\alpha 1, T1\right\rangle_{\mathbb{H} ^2} - \left\langle 1, T1\right\rangle_{\mathbb{H} ^2} \nonumber \\
& = & \sum _\alpha Z^{\alpha ^\mathrm{t}} \left\langle L^\alpha 1, (\mathfrak{H} (R) + \mathfrak{H} (R) ^*) 1\right\rangle_{\mathbb{H} ^2} - \left\langle 1,  \mathrm{Re} \, \mathfrak{H} (R)  1\right\rangle_{\mathbb{H} ^2} \nonumber \\
& = & \mathfrak{H}   (Z)  +  I_n \overline{ \mathfrak{H} (0)} -  I_n \mathrm{Re} \, \mathfrak{H} (0) \nonumber \\
& = & \mathfrak{H} (Z) -i  I_n \mathrm{Im} \, \mathfrak{H} (0), \nonumber \end{eqnarray}
see \cite{JM-freeAC,JM-freeCE}. Since $\mathfrak{H} _T$ and $\mathfrak{H}$ are NC Herglotz functions which differ by an imaginary constant, $\mu _T$ is the NC Clark measure of both $\mathfrak{b} _T$ and $\mathfrak{b}$, where $\mathfrak{b} _T$ is the inverse Cayley transform of $\mathfrak{H} _T$ and $\mathfrak{b}$ is the inverse Cayley transform of $\mathfrak{H}$ \cite{JMS-ncratClark}. 
\end{proof}

\begin{thm}[NC rational Fej\'er--Riesz] \label{ncFR}
Any NC rational, positive semi-definite left Toeplitz operator, $T = \mathrm{Re} \, \mathfrak{H} (R) \geq 0 $, $\mathfrak{H} \in \mathbb{H} ^\infty _d$, is factorable. If $\mathfrak{b} = (\mathfrak{H} -1)(\mathfrak{H} +1) ^{-1} \in [\mathbb{H} ^\infty _d ] _1$ is the NC rational and contractive inverse Cayley transform of $\mathfrak{H}$, then the factorization is
$$ T = \mathfrak{D} (R) ^* \mathfrak{D} (R), \quad \mbox{where} \quad \mathfrak{D} (R) = \mathfrak{a} (R) (I -\mathfrak{b} (R ) ) ^{-1} \in \mathbb{H} ^\infty _d, $$ $\mathfrak{a} \in [\mathbb{H} ^\infty _d ] _1$ is the NC rational Sarason outer function of $\mathfrak{b}$, $\mathrm{dim} _{FM} \, \mathfrak{a} \leq \mathrm{dim} _{FM} \, \mathfrak{b} = \mathrm{dim} _{FM} \, \mathfrak{H}$ and $\mathrm{dim} _{FM} \, \mathfrak{D} \leq 2 \mathrm{dim} _{FM} \, \mathfrak{H}$. The NC domain of $\mathfrak{D}$ contains a row-ball, $r\mathbb{B} ^d _\mathbb{N}$ of radius $r>1$ and $\mathrm{Dom} \, \mathfrak{D} \supseteq \mathrm{Dom} \, \mathfrak{a} \, \cap \, \mathrm{Dom} \, \mathfrak{H}$. 
\end{thm}
\begin{proof}
Let $T = \mathrm{Re} \, \mathfrak{H} (R)  \geq 0$ be a bounded, positive and NC rational left Toeplitz operator. That is, $\mathfrak{H} \in \mathbb{H} ^\infty _d$ is bounded, NC Herglotz and NC rational. Since we assume that $T$ is bounded, the NC Fatou theorem implies, as before in the proof of Theorem \ref{main1}, that 
$$ T = (I - \mathfrak{b}  (R) ^*) ^{-1} (I - \mathfrak{b}  (R) ^* \mathfrak{b} (R) ) (I - \mathfrak{b} (R) ) ^{-1}, $$
where $\mathfrak{b} = (\mathfrak{H} - I ) (\mathfrak{H} +I ) ^{-1} \in [ \mathbb{H} ^\infty _d ] _1$ is the NC rational inverse Cayley transform of $\mathfrak{H}$ \cite{JM-ncFatou}. We can assume that $T \neq I$, as then it is trivially factorable. If $\mathfrak{b}$ is inner then observe that $T \equiv 0$, which is again trivially factorable. By Theorem \ref{CEisinner}, if $\mathfrak{b}$ is not inner, it cannot be CE. Since $\mathfrak{b}$ is non-CE, Theorem \ref{CEisinner} implies that $I - \mathfrak{b} (R) ^* \mathfrak{b} (R) = \mathfrak{a} (R) ^* \mathfrak{a} (R)$ is factorable, where $\mathfrak{a} \in [\mathbb{H} ^\infty _d ] _1$ is the NC rational Sarason outer function of $\mathfrak{b}$. Hence,
$$ T = (I -\mathfrak{b} (R) ^* ) ^{-1} \mathfrak{a} (R) ^* \mathfrak{a} (R) (I -\mathfrak{b} (R) ) ^{-1} =: \mathfrak{D} (R) ^* \mathfrak{D} (R). $$ Since we assume that $T$ is bounded, so is $\mathfrak{D}$ so that $\mathfrak{D} \in \mathbb{H} ^\infty _d$. In particular, since $\mathfrak{D} \in \mathbb{H} ^\infty _d$ is NC rational, $\mathrm{Dom} \, \mathfrak{D}$ contains a row-ball, $r \mathbb{B} ^d _\mathbb{N}$ of radius $r>1$, and by its definition, $\mathrm{Dom} \, \mathfrak{D} \supseteq \mathrm{Dom} \, (1- \mathfrak{b} ) ^{-1} \cap \mathrm{Dom} \, \mathfrak{a}$. 

Let $(A,B,C,D)$ be a minimal FM realization of $\mathfrak{H}$. Since $\mathfrak{H} = 2 (1- \mathfrak{b} ) ^{-1} -1$, the minimal FM realization of $(1-\mathfrak{b} ) ^{-1}$ is $(A,B, C' , D')$ with $C' = \frac{1}{2} C$ and $D' = (1 - \mathfrak{b} (0) ) ^{-1}$. The minimal FM realization of $1-\mathfrak{b}$ is then given by the `flip' realization $(A^{(-1)}, B^{(-1)}, C^{(-1)}, D^{(-1)})$ where 
$$ A^{(-1)} _j := A_j -\frac{1}{2 D'} B_j C, \quad B^{(-1)} = \frac{1}{D'} B, \quad C^{(-1)} = -\frac{1}{2D'} C \quad \mbox{and} \quad D^{(-1)} = 1 - \mathfrak{b} (0), $$ see \cite[Section 5.2]{HKV-ncpolyfact}. Finally, by the construction of the minimal FM realization of the NC Sarason function, $\mathfrak{a}$ of $\mathfrak{b}$ in Proposition \ref{ratfactor}, $\mathfrak{a}$ has the finite FM realization, $(A ^{(-1)}, \frac{1}{D'} B , \widetilde{C}, \widetilde{D} )$. A finite and generally non-minimal FM realization for $\mathfrak{D} = \mathfrak{a} (1 -\mathfrak{b} )^{-1}$ is then given by $(\hat{A}, \hat{B} , \hat{C} , \hat{D} )$ where 
$$ \hat{A} _j = \begin{pNiceMatrix} A^{(-1)} _j & \frac{1}{2D'} B_j C \\ & A_j \end{pNiceMatrix}. $$ In particular, the minimal FM realization size of $\mathfrak{D}$ is at most twice that of $\mathfrak{H}$.
\end{proof}

In \cite{JM-ncFatou} and \cite{JM-ncld}, we developed the Lebesgue decomposition of any positive NC measure, $\mu \in  \left( \mathscr{A} _d   \right) ^\dag _+$ with respect to NC Lebesgue measure, $m \in  \left( \mathscr{A} _d   \right) ^\dag _+$. Namely, $\mu = \mu _{ac} + \mu _s$, where $\mu _{ac}$ is a absolutely continuous with respect to $m$ and $\mu _s$ is singular with respect to $m$. 
\begin{cor} \label{ncratFatcor}
Let $\mathfrak{b} \in [ \mathbb{H} ^\infty _d ] _1$ be a contractive NC rational left multiplier. Then the NC Clark measure, $\mu _\mathfrak{b} \in  \left( \mathscr{A} _d   \right) ^\dag _+$ is singular if and only if $\mathfrak{b}$ is inner.
\end{cor}
\begin{proof}
In \cite[Corollary 6.29]{JM-ncFatou}, we proved that if $b \in [ \mathbb{H} ^\infty _d ] _1$ is inner then $\mu _b$ is singular with respect to NC Lebesgue measure, $m$. Following the proof, we observed that if $T := 2I - b(R) - b(R) ^* >0$ is a factorable left Toeplitz operator then the converse also holds and we further proved that $T \geq T_0 = (I - b(R) ^* ) (I - b(R) ) > 0$ so that $T$ has a factorable left Toeplitz minorant. By the NC rational Fej\'er--Riesz Theorem, if $b = \mathfrak{b}$ is NC rational, $T$ is factorable and the claim follows. 
\end{proof}

\begin{thm} \label{ncratRN}
Let $\mathfrak{b} \in [ \mathbb{H} ^\infty _d ] _1$ be a non-CE and contractive NC rational left multiplier. Then $\mu _{ac} := \mu _{\mathfrak{b} ; ac}$ is of type$-L$ and 
$$ \mu _{ac} (L^\omega ) = \left\langle \mathfrak{h}, L^\omega \mathfrak{h}\right\rangle_{\mathbb{H} ^2} = \left\langle 1, \mathfrak{h}^\mathrm{t} (R)^* \mathfrak{h}^\mathrm{t} (R) L^\omega 1\right\rangle_{\mathbb{H} ^2}, $$ where $\mathfrak{h} ^\mathrm{t} = \mathfrak{a} (1 - \mathfrak{b} ) ^{-1} \in \mathbb{H} ^\infty _d$, with $\mathfrak{a}$ the outer Sarason function of $\mathfrak{b}$. The function $\mathfrak{h} := \mathfrak{h} ^\mathrm{t} (R) 1 \in \mathbb{H} ^2 _d$ is NC rational, hence in $\mathbb{H} ^\infty _d$ and $T = \mathfrak{h}^\mathrm{t} (R) ^* \mathfrak{h}^\mathrm{t} (R)$ is a bounded left Toeplitz operator.
\end{thm}
In the above statement, recall that a Gelfand--Naimark--Segal (GNS) construction applied to any $\mu \in  \left( \mathscr{A} _d   \right) ^\dag _+$ produces a GNS--Hilbert space, $\mathbb{H} ^2 _d (\mu)$, and a GNS--row isometry, $\Pi _\mu$, acting on $\mathbb{H} ^2 _d (\mu)$ \cite{JM-freeCE,JM-ncld}. A positive NC measure, $\mu \in  \left( \mathscr{A} _d   \right) ^\dag _+$, is said to be of \emph{type$-L$}, if $\Pi _\mu$ is jointly unitarily equivalent to $L$ \cite{JM-ncld}.
\begin{proof}
By the NC Fatou theorem,  if 
$$ T_r := (I - \mathfrak{b} (rR) ^*)  ^{-1}  (I - \mathfrak{b} _r (R) ^* \mathfrak{b} _r(R) ) (I - \mathfrak{b} (rR) ) ^{-1}, $$ then $(I + T_r ) ^{-1} \stackrel{SOT}{\rightarrow} ( I +T ) ^{-1}$ where 
$$ \mu _{ac} (L^\omega ) = \left\langle \sqrt{T} 1, \sqrt{T} L^\omega 1\right\rangle_{\mathbb{H} ^2}, $$ and $T$ is a closed, positive semi-definite and densely--defined left Toeplitz operator with the free polynomials as a core, see \cite{JM-ncFatou}. Here,
$$ I + T_r = (I - \mathfrak{b} (rR) ^* ) ^{-1} \left( 2I - \mathfrak{b} (rR) - \mathfrak{b} (rR) ^* \right) (I - \mathfrak{b} (R) ) ^{-1}. $$ By the NC rational Fej\'er--Riesz theorem, 
$$ I + T_r = (I - \mathfrak{b} (rR) ^* ) ^{-1} \mathfrak{d} _r (R) ^* \mathfrak{d} _r (R) (I - \mathfrak{b} (rR ) ) ^{-1}, $$ where $\| \mathfrak{d} _r (R) \| ^2 \leq 2$, $\mathfrak{d} \in \mathbb{H} ^\infty _d$ is NC rational and $\mathfrak{d} _r(R) ^* \mathfrak{d} _r (R) = 2I - \mathfrak{b} (rR) - \mathfrak{b} (rR) ^*$. Hence
$$ (I + T_r ) ^{-1} = (I - \mathfrak{b} (rR) ) \mathfrak{d} _r (R) ^{-1} \mathfrak{d} _r (R) ^{-*} (I - \mathfrak{b} (rR) ^* ). $$ 
Since $(I + T_r ) ^{-1} \stackrel{SOT}{\rightarrow} (I+T ) ^{-1}$, taking inner products against NC kernels shows that $\mathfrak{d} _r ^\mathrm{t} (Z) ^{-1}$ and hence $\mathfrak{d} _r ^\mathrm{t} (Z)$ converges pointwise to some $\widetilde{\mathfrak{d}} ^\mathrm{t} (Z)$ in the unit row-ball. This and uniform boundedness of the net $\mathfrak{d} _r (R)$ implies $WOT$ convergence of $\mathfrak{d} _r (R)$ to some $\widetilde{\mathfrak{d}} (R)$, $\widetilde{\mathfrak{d}} \in \mathbb{H} ^\infty _d$. In particular, it follows that 
\begin{equation} (I +T ) ^{-1} = (I - \mathfrak{b} (R) ) \widetilde{\mathfrak{d}}  (R) ^{-1} \widetilde{\mathfrak{d}} (R) ^{-*} (I - \mathfrak{b} (R) ^* ). \label{eq1} \end{equation}

On the other hand, setting 
$\tau _r = 2I - \mathfrak{b} (rR) - \mathfrak{b} (rR) ^*$ for $0 < r \leq 1$, consider the NC measure 
$$ \mu _r (L^\omega ) := \left\langle  1, \tau _r L^\omega 1\right\rangle_{\mathbb{H} ^2}. $$ This has left free Herglotz--Riesz transform:
\begin{eqnarray} \mathfrak{H} _r (Z) &:= & 2 \sum _\omega Z^\omega \left\langle L^{\omega ^\mathrm{t}}1, \tau _r 1 \right\rangle_{\mathbb{H} ^2 } - I_n \left\langle 1, \tau _r 1\right\rangle_{\mathbb{H} ^2} \nonumber \\
& = & 4I - 2 \overline{\mathfrak{b} (0) } I_n - 2 \mathfrak{b} (r Z ) - 2I + (\mathfrak{b} (0) + \overline{\mathfrak{b} (0)}) I_n \nonumber \\
& = & (2 + 2i \, \mathrm{Im} \, \mathfrak{b} (0) ) I_n - 2 \mathfrak{b} (r Z), \nonumber \end{eqnarray} so that $\mathfrak{H} _r (Z) = \mathfrak{H} (rZ)$ is NC rational. This has inverse Cayley transform
$$  \mathfrak{B} (rZ) = ( \mathfrak{H} (rZ) - I_n ) (\mathfrak{H} (rZ) + I_n ) ^{-1}, \quad \quad r \in (0 , 1], $$ for some NC rational $\mathfrak{B} \in [\mathbb{H} ^\infty _d ] _1$. Since $\mathfrak{B} (rR) \stackrel{SOT-*}{\rightarrow} \mathfrak{B} (R)$ by \cite[Lemma 4]{JM-ncFatou}, and $\mathfrak{C} := \left( \begin{smallmatrix} \mathfrak{B} \\ \mathfrak{A} \end{smallmatrix} \right)$ is inner since $\mathfrak{B}$ is NC rational, Proposition \ref{afunconverge} implies that if $\mathfrak{A} _r$ is the NC rational Sarason outer function of $\mathfrak{B} (r Z)$ for $r \in (0 , 1]$ that $\mathfrak{A} _r (R) \stackrel{SOT-*}{\rightarrow} \mathfrak{A} (R)$. In particular it follows that 
$$ \mathfrak{d} _r ^\mathrm{t} (Z) =  (I_n - \mathfrak{B} ^\mathrm{t} (rZ) ) ^{-1} \mathfrak{A} _r ^\mathrm{t} (Z) \rightarrow \mathfrak{d} ^\mathrm{t} (Z) =  (I_n - \mathfrak{B} ^\mathrm{t} (Z) ) ^{-1} \mathfrak{A} ^\mathrm{t} (Z). $$ This pointwise convergence in the row ball and uniform norm boundedness implies that $\mathfrak{d} _r (R) \stackrel{WOT}{\rightarrow} \mathfrak{d} (R)$. Hence,
\begin{eqnarray} (I + T ) ^{-1} & = & (I - \mathfrak{b} (R) ) \widetilde{\mathfrak{d}}  (R) ^{-1} \widetilde{\mathfrak{d}} (R) ^{-*} (I - \mathfrak{b} (R) ^* ) \nonumber \\
& = & SOT-\lim \, (I - \mathfrak{b} (rR) ) \mathfrak{d} _r  (R) ^{-1} \mathfrak{d} _r (R) ^{-*} (I - \mathfrak{b} (rR) ^* ). \nonumber \end{eqnarray} Taking inner products against NC Szeg\"o kernels and using that $\mathfrak{d} _r ^\mathrm{t} (Z)$ converges pointwise to $\mathfrak{d} ^\mathrm{t} (Z)$ where 
$$ \mathfrak{d} (R) ^* \mathfrak{d} (R) = \tau = 2I - \mathfrak{b} (R) - \mathfrak{b} (R) ^*, $$ shows that $\mathfrak{d} = \widetilde{\mathfrak{d} }$. In conclusion,
$$ I +T = (I - \mathfrak{b} (R) ^* ) ^{-1} \underbrace{\tau}_{= \mathfrak{d} (R) ^* \mathfrak{d} (R) } (I -\mathfrak{b} (R) ) ^{-1}, $$ so that 
\begin{eqnarray} T & = &  (I - \mathfrak{b} (R) ^* ) ^{-1} (\mathfrak{d} (R) ^* \mathfrak{d} (R) - (I- \mathfrak{b} (R) ^*) (I-\mathfrak{b} (R)) ) (I -\mathfrak{b} (R) ) ^{-1} \nonumber \\
& = & (I - \mathfrak{b} (R) ^* ) ^{-1} (I - \mathfrak{b} (R) ^* \mathfrak{b} (R) ) (I -\mathfrak{b} (R) ) ^{-1} \nonumber \\
& = & (I - \mathfrak{b} (R) ^* ) ^{-1} \mathfrak{a} (R) ^* \mathfrak{a} (R) (I - \mathfrak{b} (R ) ) ^{-1}. \nonumber \end{eqnarray} 
By Proposition \ref{ratfactor}, $\mathfrak{a} \in [\mathbb{H} ^\infty _d ] _1$ is outer and NC rational. By \cite[Theorem 1, Theorem 2]{JMS-ncratClark}, $\mu _{ac} := \mu _{\mathfrak{b} ;ac}$ is purely of type$-L$. Hence $\mu _{ac} (L^\omega ) = m_\mathfrak{h} (L^\omega ) = \left\langle \mathfrak{h}, L^\omega \mathfrak{h}\right\rangle_{\mathbb{H} ^2}$ is a positive symmetric vector functional for some $\mathfrak{h} \in \mathbb{H} ^2 _d$, \cite[Corollary 6.23]{JM-ncld}. By \cite{JM-freeSmirnov}, $\mathfrak{h} ^\mathrm{t} (R) := M^R _{\mathfrak{h}}$ is a closed, densely--defined right multiplier with the free polynomials as a core, so that $\mathfrak{h} ^\mathrm{t} (R) ^* \mathfrak{h} ^\mathrm{t} (R) \geq 0$ is a closed positive semi-definite and self-adjoint operator. By the uniqueness of the Riesz representation of closed, densely--defined and positive semi-definite sesquilinear forms, \cite[Chapter VI, Theorem 2.1, Theorem 2.23]{Kato}, it follows that 
$$ \mathfrak{h} ^\mathrm{t} (R) ^* \mathfrak{h} ^\mathrm{t} (R) = T = (I - \mathfrak{b} (R) ^* ) ^{-1} \mathfrak{a} (R) ^* \mathfrak{a} (R) (I - \mathfrak{b} (R ) ) ^{-1}. $$ The unbounded Douglas factorization lemma, \cite[Theorem 2]{DFL} then implies that there are contractions $C, D$ so that 
$$ C \mathfrak{h} ^\mathrm{t} (R) = \mathfrak{a} (R) (I - \mathfrak{b} (R) ) ^{-1} \quad \mbox{and} \quad D \mathfrak{a} (R) (I -\mathfrak{b} (R) ) ^{-1} = \mathfrak{h} ^\mathrm{t} (R). $$ Since $\mathfrak{a} (R) (I -\mathfrak{b} (R) ) ^{-1}$ has dense range, we conclude that $C, D$ commute with the left shifts and $CD = DC = I$. Since both $C,D$ are contractive, it follows that $C, D$ must both be isometries and $D=C^*$ so that $C$ is unitary. By \cite[Theorem 1.2]{DP-inv}, $C = C(R)$ and $D = D(R) \in \mathscr{R} ^\infty _d$ are unitary right multipliers, so that $C(R) = \zeta I$, $\zeta \in \partial \mathbb{D}$ is constant by \cite[Corollary 1.5]{DP-inv}. In conclusion we can assume, without loss in generality that 
$$ \mathfrak{h} ^\mathrm{t} (R) = \mathfrak{a} (R) (I -\mathfrak{b} (R) ) ^{-1}. $$ Hence $\mathfrak{h} = \mathfrak{a} (R) (I -\mathfrak{b} (R) ) ^{-1} 1 \in \mathbb{H} ^2 _d$ is NC rational so that $\mathfrak{h} ^\mathrm{t} \in \mathbb{H} ^\infty _d$ by \cite[Theorem A]{JMS-ncrat}. In particular, $T = \mathfrak{h} ^\mathrm{t} (R) ^* \mathfrak{h} ^\mathrm{t} (R)$ is a bounded left Toeplitz operator. 
\end{proof}

Suppose that $\mathfrak{H} \sim \mathbb{H} ^\infty _d$ is an NC rational Smirnov multiplier of Fock space. That is, $\mathfrak{H} \in \mathscr{O} (\mathbb{B} ^d _\mathbb{N} )$ is a locally bounded NC function and left multiplication by $\mathfrak{H}$ is a densely--defined and closed operator on its maximal domain in $\mathbb{H} ^2 _d$. Equivalently, right multiplication by $\mathfrak{H} ^\mathrm{t}$, $M^R _{\mathfrak{H} ^\mathrm{t}} = \mathfrak{H} (R)$, is densely--defined and closed on its maximal domain. The above NC Radon--Nikodym formula allows us to extend our NC rational Fej\'er--Riesz theorem to the case where $\mathrm{Re} \, \mathfrak{H} (R) \geq 0$ and $\mathfrak{H} \sim \mathbb{H} ^\infty _d$ belongs to the Smirnov class. Here, given any $H \sim \mathbb{H} ^\infty _d$, we say that $\mathrm{Re} \, H (R) \geq 0$ if this holds in the \emph{quadratic form sense}: 
$$ q_{\mathrm{Re} \, H} (p, q ) := \frac{1}{2}\left\langle H (R)^* p, q\right\rangle_{\mathbb{H} ^2} + \frac{1}{2} \left\langle p, H (R) ^* q\right\rangle_{\mathbb{H} ^2} \geq 0, $$ for any $p,q \in \mathbb{C} \{ \mathbb{\mathfrak{z}} \} $. That is, $H$ is positive in this quadratic form sense if and only if $H$ is an NC Herglotz function. Suppose then that $\mathfrak{H}$ is an NC rational Herglotz function. This also defines a positive NC measure on the free disk system by $\mu _\mathfrak{H} (L^\omega) := q _{\mathrm{Re} \, \mathfrak{H}} (1, L^\omega 1)$. In fact, all positive NC measures are obtained in this way \cite{JM-freeCE}. 

\begin{cor}
Let $\mathfrak{H} \sim \mathbb{H} ^\infty _d$ be an NC Herglotz and NC rational Smirnov multiplier so that $\mathrm{Re} \, \mathfrak{H} \geq 0$ in the quadratic form sense, and suppose that $q_{\mathrm{Re} \, \mathfrak{H}}$ is a closeable form, \emph{i.e.} that $\mu _\mathfrak{H} \in  \left( \mathscr{A} _d   \right) ^\dag _+$ is an absolutely continuous NC measure. Then, $\mathrm{Re} \, \mathfrak{H}$ is factorable in the sense that 
$$ q _{\mathrm{Re} \, \mathfrak{H}} (p , p') = \left\langle p, T p'\right\rangle_{\mathbb{H} ^2}; \quad \quad p, p' \in \mathbb{C} \{ \mathbb{\mathfrak{z}} \} , $$ where 
$$ T = (I - \mathfrak{b} (R) ^* ) ^{-1} \mathfrak{a} (R) ^* \mathfrak{a} (R) (I - \mathfrak{b} (R) ) ^{-1}, \quad \quad \mathfrak{b} = (\mathfrak{H} -I)(\mathfrak{H} +I ) ^{-1} \in [\mathbb{H} ^\infty _d ]_1, $$ $\mathfrak{b}$ is NC rational and non-CE, and $\mathfrak{a}$ is the NC rational outer Sarason function of $\mathfrak{b}$. Moreover, $T$ and $\mathfrak{a} (I-\mathfrak{b}) ^{-1} \in \mathbb{H} ^\infty _d$ are bounded.
\end{cor}
\begin{proof}
Since we assume that $\mathfrak{H}$ is NC rational and Herglotz, its Cayley transform, $\mathfrak{b} \in [\mathbb{H} ^\infty _d ] _1$ is NC rational and contractive. Since we assume that $\mu _\mathfrak{H}$ is absolutely continuous, $\mathfrak{b}$ must be non-CE, so that it has a non-zero, NC rational Sarason outer function, $\mathfrak{a} \in [\mathbb{H} ^\infty _d ] _1$ so that $\left( \begin{smallmatrix} \mathfrak{b} \\ \mathfrak{a} \end{smallmatrix} \right)$ is inner and $\mu _{\mathfrak{H}} = \mu _\mathfrak{b}$ is the NC Clark measure of $\mathfrak{b}$. The claim now follows from the previous theorem.
\end{proof}

%\bibliography{Bibs/ncBla}

\bigskip
  \footnotesize

  M.~T.~Jury, \textsc{Department of Mathematics, University of Florida}\par\nopagebreak
  \textit{E-mail address:} \texttt{mjury@ad.ufl.edu}

  \medskip

  R.~T.~W.~Martin, \textsc{Department of Mathematics, University of Manitoba}\par\nopagebreak
  \textit{E-mail address:} \texttt{Robert.Martin@umanitoba.ca}

\vspace{1cm}

\end{document}